\documentclass[reqno,11pt]{article}
\usepackage{bbm}
\usepackage{mathrsfs}
\usepackage{amssymb}

\usepackage[usenames,dvipsnames]{xcolor}
\usepackage{amsthm}
\usepackage{amsmath}
\usepackage{amsfonts}
\usepackage{enumerate}
\usepackage{txfonts}
\usepackage{bm}
\usepackage{psfrag}

\usepackage{graphicx}

\numberwithin{equation}{section}

\newtheorem{theorem}{Theorem}[section]
\newtheorem{lemma}[theorem]{Lemma}
\newtheorem{corollary}[theorem]{Corollary}
\newtheorem{proposition}[theorem]{Proposition}

\newtheorem{remark}[theorem]{Remark}

\theoremstyle{definition}
\newtheorem{definition}[theorem]{Definition}
\newtheorem{assumption}[theorem]{Assumption}
\newtheorem{example}[theorem]{Example}

\begin{document}

\def\be{\begin{eqnarray}}
\def\ee{\end{eqnarray}}
\def\p{\partial}
\def\no{\nonumber}
\def\e{\epsilon}
\def\de{\delta}
\def\De{\Delta}
\def\om{\omega}
\def\Om{\Omega}
\def\f{\frac}
\def\th{\theta}
\def\la{\lambda}
\def\lab{\label}
\def\b{\bigg}
\def\var{\varphi}
\def\na{\nabla}
\def\ka{\kappa}
\def\al{\alpha}
\def\La{\Lambda}
\def\ga{\gamma}
\def\Ga{\Gamma}
\def\ti{\tilde}
\def\wti{\widetilde}
\def\wh{\widehat}
\def\ol{\overline}
\def\ul{\underline}
\def\Th{\Theta}
\def\si{\sigma}
\def\Si{\Sigma}
\def\oo{\infty}
\def\q{\quad}
\def\z{\zeta}
\def\co{\coloneqq}
\def\eqq{\eqqcolon}
\def\bt{\begin{theorem}}
\def\et{\end{theorem}}
\def\bc{\begin{corollary}}
\def\ec{\end{corollary}}
\def\bl{\begin{lemma}}
\def\el{\end{lemma}}
\def\bp{\begin{proposition}}
\def\ep{\end{proposition}}
\def\br{\begin{remark}}
\def\er{\end{remark}}
\def\bd{\begin{definition}}
\def\ed{\end{definition}}
\def\bpf{\begin{proof}}
\def\epf{\end{proof}}
\def\bex{\begin{example}}
\def\eex{\end{example}}
\def\bq{\begin{question}}
\def\eq{\end{question}}
\def\bas{\begin{assumption}}
\def\eas{\end{assumption}}
\def\ber{\begin{exercise}}
\def\eer{\end{exercise}}
\def\mb{\mathbb}
\def\mbR{\mb{R}}
\def\mbZ{\mb{Z}}
\def\mc{\mathcal}
\def\mcS{\mc{S}}
\def\ms{\mathscr}
\def\lan{\langle}
\def\ran{\rangle}
\def\lb{\llbracket}
\def\rb{\rrbracket}

\title{Smooth transonic flows with nonzero vorticity to a quasi two dimensional steady Euler flow model}

\author{Shangkun Weng\thanks{School of Mathematics and Statistics, Wuhan University, Wuhan, Hubei Province, China, 430072. Email: skweng@whu.edu.cn}\and Zhouping Xin\thanks{The Institute of Mathematical Sciences, The Chinese University of Hong Kong, Shatin, NT, Hong Kong. E-mail: zpxin@ims.cuhk.edu.hk}}

\date{}

\pagestyle{myheadings} \markboth{Smooth transonic flows to the quasi two dimensional model}{Smooth transonic flows to the quasi two dimensional model}\maketitle

\begin{abstract}
  This paper concerns studies on smooth transonic flows with nonzero vorticity in De Laval nozzles for a quasi two dimensional steady Euler flow model which is a generalization of the classical quasi one dimensional model. First, the existence and uniqueness of smooth transonic flows to the quasi one-dimensional model, which start from a subsonic state at the entrance and accelerate to reach a sonic state at the throat and then become supersonic are proved by a reduction of degeneracy of the velocity near the sonic point and the implicit function theorem. These flows can have positive or zero acceleration at their sonic points and the degeneracy types near the sonic point are classified precisely. We then establish the structural stability of the smooth one dimensional transonic flow with positive acceleration at the sonic point for the quasi two dimensional steady Euler flow model under small perturbations of suitable boundary conditions, which yields the existence and uniqueness of a class of smooth transonic flows with nonzero vorticity and positive acceleration to the quasi two dimensional model. The positive acceleration of the one dimensional transonic solutions plays an important role in searching for an appropriate multiplier for the linearized second order mixed type equations. A deformation-curl decomposition for the quasi two dimensional model is utilized to deal with the transonic flows with nonzero vorticity.
\end{abstract}

\begin{center}
\begin{minipage}{5.5in}
\end{minipage}
\end{center}

\section{Introduction and the main results}
\subsection{The motivations}\noindent

This paper concerns transonic flows occurring in inviscid compressible fluids, where the flow contains both subsonic and supersonic region. A general transonic flow may contain a shock, and an upstream supersonic flow immediately turns to subsonic after crossing the shock surface. Morawetz \cite{Morawetz1956} proved the nonexistence of a smooth solution to the perturbation for flow with a local supersonic region over a solid airfoil and is unstable even it exists. Here we focus on the smooth transonic flows for inviscid compressible fluids in a nozzle. There are two types of smooth transonic flows called Taylor and Meyer types in a De Laval nozzle whose cross section converges first and then diverges. For a Taylor type transonic flow, there are supersonic enclosures attached to the nozzle wall, and it was shown that such a smooth transonic flow does not exist in general and is unstable under small perturbations of the shape of the nozzle even it exists (see \cite{Bers1958}). For Meyer type transonic flows, the flow accelerates from subsonic and smoothly pass through the sonic surface and finally becomes supersonic, the sonic surface is observed in experiment to be located near the throat of the nozzle.

Recently, Wang and Xin \cite{WX2013,WX2016,WX2019,WX2020} have made progresses on Meyer type transonic flows and established the existence and uniqueness of such kind of transonic flows satisfying some physical boundary conditions on the De Laval nozzle for two dimensional steady irrotational compressible Euler equations. They have shown that the sonic points can locate only at the throat of the nozzle and the points on the nozzle wall with positive curvature, under the assumption that the nozzle walls  are required to be suitably flat at the throat in \cite{WX2019,WX2020}. Such a flatness condition near the throat is almost necessary as was shown in \cite{WX2019}. Moreover, the velocity constructed in \cite{WX2019} is along the $x_1$-axis and the acceleration must be zero at the throat where the flow becomes sonic. However, the methods developed in \cite{WX2019,WX2020} depend crucially on that the flow is irrotational and the nozzle is suitably flat at the throat and it seems quite difficult to extend the approach to the steady Euler flows and general De Laval nozzles. One natural question arises:

{\bf Problem.} \textit{Do there exist smooth transonic flows of Meyer type in De Laval nozzles with nonzero vorticity and positive acceleration at the set of sonic points near the throat of the nozzle?}

\subsection{A quasi two dimensional Euler flow model}\label{model}\noindent

To answer the {\bf Problem}, we propose to study the corresponding problem for a quasi two dimensional steady Euler flow model which is derived as follows. Consider the steady three dimensional isentropic Euler equations in a slowly varying nozzle $\mathbb{D}$:
\be\label{3deuler}\begin{cases}
\p_{y_1}(\varrho v_1) + \p_{y_2}(\varrho v_2) +\p_{y_3}(\varrho v_3)=0, \\
\varrho v_1 \p_{y_1} v_1 + \varrho v_2 \p_{y_2} v_1 + \varrho v_3 \p_{y_3} v_1 + \p_{y_1} P(\varrho)=0,\\
\varrho v_1 \p_{y_1} v_2 + \varrho v_2 \p_{y_2} v_2 + \varrho v_3 \p_{y_3} v_2 + \p_{y_2} P(\varrho)=0,\\
\varrho v_1 \p_{y_1} v_3 + \varrho v_2 \p_{y_2} v_3 + \varrho v_3 \p_{y_3} v_3 + \p_{y_3} P(\varrho)=0,
\end{cases}\ee
where ${\bf v}=(v_1, v_2,v_3)$, $\varrho$ and $P$ stand for the velocity, density and pressure respectively. For polytropic gases, the equation of state takes the form
\begin{equation}\label{eqstate}
P(\varrho)=\varrho^{\gamma}\quad \ \ \ \
\end{equation}
where $\gamma> 1$ is the adiabatic exponent. The domain $\mathbb{D}$ is given by
\be\no
\mathbb{D}=\{(y_1,y_2,y_3): L_0<y_1<L_1, -1<y_2<1,\ 0<y_3<a(y_1)\},
\ee
where the positive function $a(y_1)\in C^{\infty}([L_0,L_1])$ with $L_0<0<L_1$ satisfies
\be\label{qu1}\begin{cases}
a'(y_1)<0,\ \  \ \forall y_1\in [L_0,0),\\
a'(0)=0,\\
a'(y_1)>0,\ \  \ \forall y_1\in (0, L_1].
\end{cases}\ee
Here $a(y_1)$ is assumed to belong to $C^{\infty}([L_0,L_1])$ just for simplicity and is not optimal.

Perform a change of variables:
\be\no
x_1= y_1, \ \ x_2=y_2,\ \ \ x_3=\frac{y_3}{a(y_1)},
\ee
and define new unknowns as
\be\no
(\rho, u_1, u_2 ,u_3)(x_1,x_2,x_3)= (\varrho, v_1,v_2,v_3)(x_1,x_2, a(x_1)x_3).
\ee
Then it holds that on $\{(x_1,x_2,x_3):L_0<x_1<L_1, |x_2|<1, 0<x_3<1\}$:
\be\label{3d1}\begin{cases}
\p_{x_1}(a(x_1) \rho u_1)+ \p_{x_2}(a(x_1)\rho u_2)- a'(x_1) \rho u_1- a'(x_1) x_3\p_{x_3} (\rho u_1) + \p_{x_3}(\rho u_3)=0,\\
\rho u_1(\p_{x_1} u_1-\frac{a'(x_1)}{a(x_1)}x_3 \p_{x_3} u_1)+ \rho u_2 \p_{x_2} u_1 + \frac{\rho u_3}{a(x_1)}\p_{x_3} u_1 + \p_{x_1} P(\rho)- \frac{a'(x_1)}{a(x_1)} x_3 \p_{x_3} P(\rho)=0,\\
\rho u_1(\p_{x_1} u_2-\frac{a'(x_1)}{a(x_1)}x_3 \p_{x_3} u_2)+ \rho u_2 \p_{x_2} u_2 + \frac{\rho u_3}{a(x_1)}\p_{x_3} u_2+ \p_{x_2} P(\rho)=0,\\
\rho u_1(\p_{x_1} u_3-\frac{a'(x_1)}{a(x_1)}x_3 \p_{x_3} u_3)+ \rho u_2 \p_{x_2} u_3 + \frac{\rho u_3}{a(x_1)}\p_{x_3} u_3+ \frac{1}{a(x_1)}\p_{x_3} P(\rho)=0.
\end{cases}\ee

\begin{figure}
  \centering
  \includegraphics[width=11cm,height=6cm]{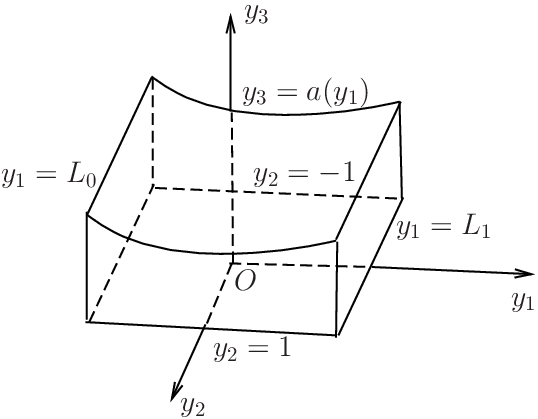}
  \caption{A three dimensional slowly varying nozzle}
\end{figure}

Assume that the nozzle is slowly varying, i.e. for some small positive constant $\epsilon$
\be\label{a1}
a(x_1)=a(0)+O(\epsilon),  a'(x_1)=O(\epsilon).
\ee
Then it is reasonable to assume that the variation of the velocity $u_3$ is also of order $O(\epsilon)$:
\be\label{a2}
u_3=O(\epsilon).
\ee
Thus neglecting the terms of order $O(\epsilon)$ and the last equation for $u_3$ in \eqref{3d1}, we obtain a quasi two dimensional steady Euler flow model for $(\rho, u_1, u_2)$ in $\Om=\{(x_1,x_2): L_0<x_1<L_1, -1<x_2<1\}$:
\be\label{q2deuler}\begin{cases}
\p_{x_1}(a(x_1) \rho u_1)+ \p_{x_2}(a(x_1)\rho u_2)=0,\\
\rho u_1\p_{x_1} u_1+ \rho u_2 \p_{x_2} u_1+ \p_{x_1} P(\rho)=0,\\
\rho u_1\p_{x_1} u_2+ \rho u_2 \p_{x_2} u_2+ \p_{x_2} P(\rho)=0.
\end{cases}\ee

Note that one can derive the classical quasi one dimensional model from the two dimensional isentropic Euler equations by same arguments as above. Indeed, suppose that the flow parameters do not depend on the variable $x_2$, then the system \eqref{q2deuler} reduces to the quasi one dimensional model
\be\label{quasi}\begin{cases}
(\rho u_1)'(x_1)+\frac{a'(x_1)}{a(x_1)}\rho u_1=0,\\
\rho u_1u_1'(x_1)+(P(\rho))'(x_1)=0.
\end{cases}\ee

The quasi one dimensional model \eqref{quasi} is usually derived under the approximation that the flow is one dimensional, i.e., the flow parameters across each section are uniform, one may refer to \cite[Chapter 2]{lr57} for a detailed derivation. The above argument provides a multi-scale derivation of \eqref{quasi}, which seems to be new. The model \eqref{quasi} had been used in aerodynamics and engineering to describe the gas flow in a duct of varying cross-section and been extensively studied by many mathematicians (see \cite{Courant1948,egm84,gmp84,L82a,RXY13} and the references therein). Steady transonic shock solutions were constructed in \cite{egm84,L82b} satisfying suitable boundary conditions. It was proved in \cite{L82a} that when $\frac{a'(x_1)}{a(x_1)}$ is suitably small, the flow along a contracting duct which contains a shock wave is dynamically unstable, while a flow with a standing shock wave along an expanding duct is dynamically stable. It was shown in \cite{RXY13} proved that a steady transonic shock wave in divergent quasi one dimensional nozzles is dynamically global stable without assuming either the smallness of the relative slope of the nozzle or the weakness of the shock wave.

Note that the second and third equations in \eqref{q2deuler} are just the momentum equations in the two dimensional steady Euler equations, and the Bernoulli's law also holds for this system:
\be\label{ber}
u_1\p_{x_1} B + u_2\p_{x_2} B=0,
\ee
where $B=\frac{1}{2}(u_1^2+u_2^2)+\frac{p'(\rho)}{\gamma-1}$. Thus the model \eqref{q2deuler} has its own interest and is expected to describe the variation of the flow parameters across each section in three dimensional slowly varying nozzles.

Concerning the {\bf Problem}, we will first prove the existence and uniqueness of smooth transonic flows of Meyer type to the quasi one-dimensional model \eqref{quasi}. These flows could have positive or zero acceleration at their sonic points and the degeneracy types near the sonic points will be classified precisely. By investigating the structural stability of these one dimensional transonic solutions to \eqref{quasi}, we will further establish the existence and uniqueness of a class of smooth transonic flows with nonzero vorticity and positive acceleration to the quasi two dimensional model \eqref{q2deuler}. These give a positive answer to the {\bf Problem} at least for the quasi one dimensional model \eqref{quasi} and the quasi two dimensional model \eqref{q2deuler}, which may shed light on the solvability of the {\bf Problem} for the steady compressible Euler equations.

\subsection{Smooth transonic flows to the quasi one dimensional model}\label{quasi1d}\noindent

To obtain an accelerating transonic flow $(\bar{\rho},\bar{u})$ to \eqref{quasi} on the interval $[L_0,L_1]$ with sonic state occurring at the point $x_1=0$, one rewrites \eqref{quasi} as
\begin{eqnarray}\label{qu1d}\begin{cases}
(a\bar{\rho} \bar{u})'(x_1)\equiv 0,\\
\bar{\rho} \bar{u} \bar{u}'(x_1)+ P'(\bar{\rho}) \bar{\rho}'(x_1)=0,\\
\bar{\rho}(L_0)=\rho_0>0,\ \ \bar{u}(L_0)= u_0>0,
\end{cases}\end{eqnarray}
where the initial state at $x_1=L_0$ is subsonic, i.e. $u_0^2<c^2(\rho_0)=\gamma \rho_0^{\gamma-1}$. Set $J=a(x_1) \bar{\rho} \bar{u}(x_1)=\rho_0 u_0 a(L_0)>0$. Then it follows from \eqref{qu1d} that
\begin{eqnarray}\label{qu1d2}
\frac{1}{2}(\bar{u}(x_1))^2+ \frac{\gamma}{\gamma-1}\bar{\rho}^{\gamma-1} \equiv B_0:=\frac{1}{2} u_0^2 + \frac{\gamma\rho_0^{\gamma-1}}{\gamma-1}.
\end{eqnarray}
and
\be\label{qu1d1}\begin{cases}
\bar{\rho}(x_1)=\frac{J}{a(x_1)\bar{u}(x_1)},\\
(a^{\gamma}\bar{u}^{\gamma+1}-\gamma J^{\gamma-1}a) \bar{u}'= \gamma J^{\gamma-1} a'(x_1)\bar{u}.
\end{cases}\ee
Define $\bar{M}^2(x_1)=\frac{\bar{u}^2}{c^2(\bar{\rho})}=\frac{\bar{u}^2}{\gamma \bar{\rho}^{\gamma-1}}$. Then
\be\label{qu1d101}
&&\bar{u}'(x_1)=-\frac{\bar{u}(x_1)}{1-\bar{M}^2}b(x_1)=\frac{\gamma J^{\gamma-1} \bar{u}}{a^{\gamma-1} \bar{u}^{\gamma+1}-\gamma J^{\gamma-1}}b(x_1),\\\label{qu1d102}
&&\bar{\rho}'(x_1)=\frac{\bar{\rho}\bar{M}^2 b(x_1)}{1-\bar{M}^2},\\\label{qu1d103}
&&\frac{d}{dx_1} \bar{M}^2=-\frac{\bar{M}^2(2+(\gamma-1)\bar{M}^2)}{1-\bar{M}^2}b(x_1),
\ee
where $b(x_1)=\frac{a'(x_1)}{a(x_1)}$.

Suppose that there exists an accelerating transonic flow to \eqref{qu1d} with the sonic state located at $x_1=0$, i.e. $\bar{u}^2(0)=c^2(\bar{\rho}(0))=\gamma (\bar{\rho}(0))^{\gamma-1}=\gamma (\frac{J}{a(0)\bar{u}(0)})^{\gamma-1}$. By \eqref{qu1d2}, there holds
\be\label{qu1d3}
\bar{u}(0)=\sqrt{\frac{2(\gamma-1)B_0}{\gamma+1}}=\left(\frac{\gamma J^{\gamma-1}}{(a(0))^{\gamma-1}}\right)^{\frac{1}{\gamma+1}}=:c_*.
\ee
Thus
\be\label{qu1d4}
\left(\frac{a(0)}{a(L_0)}\right)^{\gamma-1}=\frac{\gamma (\rho_0 u_0)^{\gamma-1}}{\left(\frac{2(\gamma-1)B_0}{\gamma+1}\right)^{\frac{\gamma+1}{2}}}<1.
\ee
We start with the existence of general accelerating transonic flows.

\begin{proposition}\label{gat}({\bf General accelerating transonic flows.}) Suppose that the initial data $(\rho_0,u_0)$ is subsonic and the function $a(x_1)$ satisfies \eqref{qu1} and \eqref{qu1d4}. Then there exists a unique accelerating transonic flow $(\bar{\rho}(x_1), \bar{u}(x_1))\in C([L_0,L_1])$ which is subsonic in $[L_0,0)$, supersonic in $(0, L_1]$ with a sonic state at $x_1=0$. Furthermore, $(\bar{\rho}(x_1), \bar{u}(x_1))$ is smooth and satisfies the equations \eqref{qu1d} on $[L_0,0)\cup (0,L_1]$.
\end{proposition}

\begin{proof}
For smooth solutions, the system \eqref{qu1d} is equivalent to
\be\label{q1}
F(x_1,\bar{u}(x_1);J)=0,\ \ \ \bar{\rho}(x_1)=\frac{J}{a(x_1)\bar{u}(x_1)},
\ee
where
\be\label{q2}
F(x_1,t;J)= \frac{1}{2} t^2 + \frac{\gamma J^{\gamma-1}}{(\gamma-1) (a(x_1))^{\gamma-1}}\frac{1}{t^{\gamma-1}}-B_0.
\ee
For fixed $x_1\in [L_0,L_1]$ and $J$, it is easy to see that on $(0,+\infty)$, $F(x_1,t;J)$ attains its minimum value at $t=t_*(x_1)=\left(\frac{\gamma J^{\gamma-1}}{(a(x_1))^{\gamma-1}}\right)^{\frac{1}{\gamma+1}}$. By \eqref{qu1} and \eqref{qu1d4}, for any $x_1\in [L_0,0)\cup (0,L_1]$ there holds
\be\no
F(x_1,t_*(x_1),J)&=&\frac{\gamma+1}{2(\gamma-1)}\left(\frac{\gamma J^{\gamma-1}}{(a(x_1))^{\gamma-1}}\right)^{\frac{2}{\gamma+1}}-B_0\\\no
&=&\frac{\gamma+1}{2(\gamma-1)}\left(\gamma J^{\gamma-1}\right)^{\frac{2}{\gamma+1}}\left((a(x_1))^{-\frac{2(\gamma-1)}{\gamma+1}}-(a(0))^{-\frac{2(\gamma-1)}{\gamma+1}}\right)<0.
\ee

Since $\gamma>1$, it is easy to see that $\displaystyle\lim_{t\to 0+} F(x_1,t;J)=\lim_{t\to +\infty} F(x_1,t;J)=+\infty$ and $F(x_1,t;J)$ is monotone decreasing on $(0,t_*(x_1)]$ and monotone increasing on $[t_*(x_1),+\infty)$. Thus for each $x_1\in [L_0,0)\cup (0,L_1]$, $F(x_1,t;J)=0$ has exactly two solutions $0<t_{sub}(x_1)<t_*(x_1)<t_{sup}(x_1)<+\infty$. For $x_1=0$, $F(0,t;J)=0$ has exactly one solution $c_*=t_*(0)$. Define the function $\bar{u}(x_1)$ as follows:
\be\no
\bar{u}(x_1)=\begin{cases}
t_{sub}(x_1), \ &\ \ \forall x_1\in [L_0,0),\\
t_*(0),\ &\ \ \ x_1=0,\\
t_{sup}(x_1),\ &\ \ \forall x_1\in (0,L_1].
\end{cases}
\ee
For any $\delta\in (0,1)$, there holds that
\be\no
&&F(x_1,(1\pm \delta)t_*(x_1);J)=(\gamma J^{\gamma-1})^{\frac{2}{\gamma+1}}\bigg\{\bigg[\frac{(1\pm \delta)^2}{2}+\frac{(1\pm \delta)^{1-\gamma}}{\gamma-1}\bigg](a(x_1))^{-\frac{2(\gamma-1)}{\gamma+1}}-\frac{\gamma+1}{2(\gamma-1)}(a(0))^{-\frac{2(\gamma-1)}{\gamma+1}}\bigg\}\\\no
&&=(\gamma J^{\gamma-1})^{\frac{2}{\gamma+1}}\bigg\{\bigg[\frac{\gamma+1}{2}\delta^2+ O(|\delta|^3)\bigg](a(x_1))^{-\frac{2(\gamma-1)}{\gamma+1}}+\frac{\gamma+1}{2(\gamma-1)}\bigg[(a(x_1))^{-\frac{2(\gamma-1)}{\gamma+1}}-
(a(0))^{-\frac{2(\gamma-1)}{\gamma+1}}\bigg]\bigg\}.
\ee
Fix $\delta\in (0,1)$, then for sufficiently small $x_1$, it is easy to see that $F(x_1,(1\pm \delta)t_*(x_1);J)>0$, which implies that for sufficiently small $x_1$
\be\no
(1-\delta) t_*(x_1)<t_{sub}(x_1)<t_*(x_1)<t_{sup}(x_1)<(1+\delta) t_*(x_1).
\ee
Thus one has $\displaystyle\lim_{x_1\to 0} t_{sub}(x_1)=\lim_{x_1\to 0} t_{sup}(x_1)=\lim_{x_1\to 0} t_{*}(x_1)=t_*(0)$, and $(\bar{u}(x_1),\bar{\rho}(x_1)=\frac{J}{a(x_1)\bar{u}(x_1)})$ belongs to $C([L_0,L_1]))^2$ and is subsonic in $[L_0,0)$, supersonic in $(0, L_1]$ with a sonic state at $x_1=0$. Furthermore, for each $x_1\in [L_0,0)\cup (0,L_1]$, one has
\be\no
\p_t F(x_1,t_{sub}(x_1);J)<0,\ \ \p_t F(x_1,t_{sup}(x_1);J)>0
\ee
By the implicit function theorem, $(\bar{\rho}(x_1), \bar{u}(x_1))$ is smooth and satisfies the equations \eqref{qu1d} on $[L_0,0)\cup (0,L_1]$. The proof of Proposition \ref{gat} is complete.

\end{proof}

%

Note that there is no information on the differentiability of the transonic solution at the sonic point $x_1=0$ in Proposition \ref{gat}. Thus we will consider the regularity and behavior of the transonic solution at the sonic point, which depend on the geometry of the nozzle wall near the throat. Suppose that the solution $(\bar{u}(x_1),\bar{\rho}(x_1))$ is smooth at $x_1=0$, then \eqref{qu1d} is also satisfied at $x_1=0$. Differentiating the second equation in \eqref{qu1d1} and evaluating at $x_1=0$ yield
\be\no
\gamma J^{\gamma-1} a''(0)= (\gamma+1) (a(0))^{\gamma} c_*^{\gamma-1} (\bar{u}'(0))^2\geq 0.
\ee

Consider first the case that the smooth transonic flow has a positive acceleration at the sonic point:
\be\label{qu1d40}
a''(0)>0.
\ee

\begin{proposition}\label{qpo}({\bf Smooth transonic flows with positive acceleration at the sonic point.}) Suppose that the initial data $(\rho_0,u_0)$ is subsonic and the function $a(x_1)$ satisfies \eqref{qu1}, \eqref{qu1d4} and \eqref{qu1d40}. Then there exists a unique smooth transonic flow $(\bar{\rho}(x_1), \bar{u}(x_1))\in C^{\infty}([L_0,L_1])$ to \eqref{qu1d} which is subsonic in $[L_0,0)$, supersonic in $(0, L_1]$ with a sonic state at $x_1=0$.
\end{proposition}

\begin{proof}

By Proposition \ref{gat}, there exists a unique accelerating transonic flow $(\bar{\rho}(x_1), \bar{u}(x_1))\in C([L_0,L_1])$ which is subsonic in $[L_0,0)$ and supersonic in $(0, L_1]$ with a sonic state at $x_1=0$. It remains to prove that the solution passes smoothly through the point $x_1=0$. Here we employ an argument to reduce the degeneracy of the solution near the sonic point so that the implicit function theorem can be applied to find a smooth solution to \eqref{qu1d} which coincides with $(\bar{\rho}(x_1), \bar{u}(x_1))$.

It is easy to see that $F(0,c_*;J)=\frac{\p F}{\p t}(0,c_*;J)=\frac{\p F}{\p x_1}(0,c_*; J)=0$ and
\begin{eqnarray}\label{qu1d7}
\frac{\p^2 F}{\p t^2}(0,c_*;J)=1+\gamma, \ \  \frac{\p^2 F}{\p t \p x_1}(0,c_*;J)= 0,\ \  \frac{\p^2 F}{\p x_1^2}(0,c_*;J)=- \frac{c_*^2}{a(0)}a''(0) <0.
\end{eqnarray}
Thus Taylor's expansion yields
\begin{eqnarray}\label{qu1d8}
F(x_1,t;J)= \frac{1}{2}(1+\gamma) (t-c_*)^2- \frac{1}{2} \frac{c_*^2}{a(0)}a''(0) x_1^2+ G(x_1,t-c_*),
\end{eqnarray}
where for some positive constants $C_1$ and $\sigma_1$
\be\no
|G(x_1,t-c_*)|\leq C_1 (|t-c_*|^3+ |x_1|^3),\ \  \text{for any } |t-c_*|+|x_1|\leq \sigma_1.
\ee
Set $t-c_*=x_1 y(x_1)$, where $y=y(x_1)$ is a positive function defined on a neighborhood of $x_1=0$ to be determined later. Then the equation $F(t,x_1;J)=0$ can be rewritten as
\be\label{qu1d9}
y^2-\frac{c_*^2 a''(0)}{(\gamma+1)a(0)}  + \frac{2}{\gamma+1} G_1(x_1, y)=0,
\ee
where
\be\label{quld10}
|G_1(x_1,y)|=\left|\frac{1}{x_1^2}G(x_1y, x_1)\right|\leq C_1(|x_1| |y|^3+ |x_1|),\ \ \text{for any }|x_1|\leq \sigma_1.
\ee
Thus
\be\label{quld11}
H(x_1,y):=y-\sqrt{\mu^2 - \frac{2}{\gamma+1} G_1(x_1, y)}=0,\ \text{where}\ \ \mu=c_*\sqrt{\frac{a''(0)}{(\gamma+1)a(0)}}.
\ee
By \eqref{quld10}, $H(0,\mu)=0$. Also $\p_{y} G_1(x_1,y)=\frac{1}{x_1}\p_{t} G(x_1, x_1y)$, where
\be\no
\p_t G(x_1, t-c_*)&=&t-\frac{\gamma J^{\gamma-1}}{(a(x_1))^{\gamma-1}} t^{-\gamma}-(1+\gamma)(t-c_*)
\\\no&=& -\gamma (t-c_*)+c_*^{-\gamma}\frac{\gamma J^{\gamma-1}}{(a(0))^{\gamma-1}}-\frac{\gamma J^{\gamma-1}}{(a(x_1))^{\gamma-1}}t^{-\gamma}.
\ee
Therefore
\be\no
\frac{\p H}{\partial y}(0,\mu)&=&1+\frac{1}{(\gamma+1)\mu}\lim_{x_1\to 0}\frac{1}{x_1}\p_t G(x_1, \mu x_1)\\\no
&=&1+\frac{1}{(\gamma+1)\mu}\left\{-\gamma \mu+\lim_{x_1\to 0}\frac{1}{x_1}\left(c_*^{-\gamma}\frac{\gamma J^{\gamma-1}}{(a(0))^{\gamma-1}}-\frac{\gamma J^{\gamma-1}}{(a(x_1))^{\gamma-1}}(c_*+\mu x_1)^{-\gamma}\right)\right\}=1
.
\ee
Thus by the implicit function theorem, there exists a unique smooth positive function $y=y(x_1)$ defined on the interval $[-\sigma_2, \sigma_2]$ for some $0<\sigma_2\leq \sigma_1$ such that \eqref{quld11} holds. Moreover, the function $\bar{u}_1(x_1):= c_*+ x_1 y(x_1)\in C^{\infty}([-\sigma_2, \sigma_2])$ solves the equation \eqref{q1} on the interval $[-\sigma_2, \sigma_2]$, and $(\bar{u}_1(x_1),\frac{J}{a(x_1)\bar{u}_1(x_1)})$ is subsonic in $[-\sigma_2, 0)$ and supersonic in $(0,\sigma_2]$. Thanks to the uniqueness of an accelerating transonic flow to \eqref{q1}, one has $(\bar{u},\bar{\rho})\equiv (\bar{u}_1(x_1),\frac{J}{a(x_1)\bar{u}_1(x_1)})$ on $[-\sigma_2,\sigma_2]$.

Thus, we have obtained the desired smooth accelerating transonic flow $(\bar{u}(x_1),\frac{J}{a(x_1)\bar{u}(x_1)})$ to \eqref{qu1d} on the interval $[L_0,L_1]$ with a sonic point located at $x_1=0$. The derivative of $\bar{u}$ at $x_1=0$ exists and equals to $\bar{u}'(0)=y(0)=\mu=c_*\sqrt{\frac{a''(0)}{(\gamma+1) a(0)}}$.

\end{proof}

Next, we turn to the case that the smooth transonic flow may have zero acceleration at the sonic point (i.e. $a''(0)=0$). Suppose that there exists a smooth transonic flow near $x_1=0$ in this case, then $\bar{u}'(0)=0$. Note that $\bar{u}'(x_1)>0$ for any $x_1\in [L_0,0)\cup (0, L_1]$, this further implies $\bar{u}''(0)=0$. Rewrite the second equation in \eqref{qu1d1} as
\be\label{qu1d41}
\gamma J^{\gamma-1} a'(x_1)\bar{u}(x_1)= D(x_1) \bar{u}'(x_1),\ \text{where}\ \ D(x_1)=a^{\gamma}\bar{u}^{\gamma+1}-\gamma J^{\gamma-1} a(x_1).
\ee
Simple calculations show that $D(0)=D'(0)=D''(0)=0$ and $D^{(3)}(0)= (\gamma+1) c_*^{\gamma}(a(0))^{\gamma} \bar{u}^{(3)}(0)$. By \eqref{qu1d41}, further computations yield that $a^{(3)}(0)=a^{(4)}(0)=\cdots=a^{(5)}(0)=0$ and
\be\label{qu1d42}
a^{(6)}(0)=\frac{10(\gamma+1)}{\gamma J^{\gamma-1}}c_*^{\gamma-1}(a(0))^{\gamma} (\bar{u}^{(3)}(0))^2\geq 0.
\ee

If $a^{(6)}(0)>0$, one could prove that there exists a unique smooth accelerating transonic flow $(\bar{u},\bar{\rho})$ to \eqref{qu1d} with \eqref{qu1d1}, \eqref{qu1d4} and $a''(0)=\cdots=a^{(5)}(0)=0$. Indeed, we have the following general existence theorem.

\begin{proposition}\label{qzero1}({\bf Smooth transonic flows with zero acceleration at the sonic point: case 1.}) Suppose that the initial data $(\rho_0,u_0)$ is subsonic, the function $a(x_1)$ satisfies \eqref{qu1}, \eqref{qu1d4}, and for some nonnegative integer $m\geq 1$, it holds that
\be\label{qu1d200}
a''(0)=a^{(3)}(0)=\cdots =a^{(4m+1)}(0)=0,\ \ \ a^{(4m+2)}(0)>0.
\ee
Then there exists a unique smooth accelerating transonic solution $(\bar{\rho}(x_1), \bar{u}(x_1))\in C^{\infty}([L_0,L_1])$ to \eqref{qu1d} such that the solution is subsonic in $[L_0,0)$, supersonic in $(0, L_1]$ with a sonic state at $x_1=0$. The velocity can be represented as $\bar{u}(x_1)=c_*+ x_1^{2m+1} y(x_1)$ with a positive smooth function $y\in C^{\infty}([L_0,L_1])$ and
\be\no
&&\bar{u}'(0)=\bar{u}''(0)=\cdots= \bar{u}^{(2m)}(0)=0,\\\no
&&\bar{u}^{(2m+1)}(0)=(2m+1)! y(0)=(2m+1)!\sqrt{\frac{2}{(\gamma+1)}\frac{1}{(4m+2)!}\frac{c_*^2 a^{(4m+2)}(0)}{a(0)}}>0.
\ee
\end{proposition}

\begin{proof}

It follows from \eqref{qu1d200} and Taylor's expansion that
\begin{eqnarray}\no
F(x_1,t;c_*)= \frac{1}{2}(1+\gamma) (t-c_*)^2-\frac{a^{(4m+2)}(0)}{(4m+2)!} \frac{c_*^2}{a(0)} x_1^{4m+2} + G_m(x_1,t-c_*),
\end{eqnarray}
where for some positive constants $C_2$ and $\sigma_2$
\be\no
|G_m(x_1,t-c_*)|\leq C_2 (|t-c_*|^3+|t-c_*||x_1|^{4m+2}+|x_1|^{4m+3}),\ \  \text{for any } |t-c_*|+|x_1|\leq \sigma_2.
\ee
We would like to find the solution $\bar{u}(x_1)$ to \eqref{qu1d} with the form $\bar{u}(x_1)=c_*+x_1^{2m+1} y(x_1)$, where $y$ is a positive smooth function on $[-\sigma_2, \sigma_2]$, then the equation $F(x_1,\bar{u}(x_1);c_*)=0$ can be rewritten as
\be\no
y^2(x_1)-\frac{2}{(\gamma+1)}\frac{a^{(4m+2)}(0)}{(4m+2)!} \frac{c_*^2}{a(0)} + \frac{2}{\gamma+1} H_m(x_1, y(x_1))=0,
\ee
where
\be\no
|H_m(x_1,y)|=\left|\frac{1}{x_1^{2(2m+1)}}G_m(x_1,x_1y)\right|\leq C_2(|x_1|^{2m+1} |y|+ |x_1|),\ \ \text{for any }|x_1|\leq \sigma_2.
\ee
Since $y(x_1)>0$ for any $x_1\in [-\sigma_2,\sigma_2]$, thus
\be\no
y(x_1)=\sqrt{\frac{2}{(\gamma+1)}\frac{a^{(4m+2)}(0)}{(4m+2)!}\frac{c_*^2}{a(0)} - \frac{2}{\gamma+1} H_m(x_1, y(x_1))}.
\ee
The rest of the arguments are similar to the previous ones in the proof of Proposition \ref{qpo} so omitted.

\end{proof}

In the case $a''(0)=0$, there is another possibility that the solution $\bar{u}$ is only one order differentiable at $x_1=0$, that is $\bar{u}''(x_1)$ has a discontinuity at $x_1=0$. Then one can not conclude that $a^{(3)}(0)=\cdots =a^{(5)}(0)=0$ as above. Yet, the following existence result holds.

\begin{proposition}\label{qzero2}({\bf Smooth transonic flows with zero acceleration at the sonic point: case 2.}) Under the same assumptions as in Proposition \ref{qzero1} except \eqref{qu1d200}, which is replaced by
\be\label{qu1d201}
a''(0)=a^{(3)}(0)=\cdots =a^{(4m-1)}(0)=0,\ \ \ a^{(4m)}(0)>0,\ \ \text{for some integer $m\geq 1$},
\ee
then there exists a unique $C^{2m-1,1}$ smooth accelerating transonic solution $(\bar{\rho}(x_1), \bar{u}(x_1))\in C^{2m-1,1}([L_0,L_1])$ to \eqref{qu1d} such that the solution is subsonic in $[L_0,0)$, supersonic in $(0, L_1]$ with a sonic state at $x_1=0$. The velocity can be represented as $\bar{u}(x_1)=c_*+ x_1^{2m} y(x_1)$, where the function $y$ is defined on $[L_0,L_1]$ with a discontinuity at $x_1=0$:
\be\no
y(x_1)=\begin{cases}
y_-(x_1)<0,\ \ x_1\in [L_0, 0),\\
y_+(x_1)>0,\ \ x_1\in (0, L_1],
\end{cases}\ee
with $y_-\in C^{\infty}([L_0,0])$ and $y_+\in C^{\infty}([0,L_1])$.

Furthermore,
\be\no
&&\bar{u}'(0)=\bar{u}''(0)=\cdots= \bar{u}^{(2m-1)}(0)=0,\\\no
&&\bar{u}^{(2m)}(0-)= (2m)! y_1(0)<0,\ \ \bar{u}^{(2m)}(0+)=(2m)! y_2(0)>0.
\ee
\end{proposition}

\begin{proof}

It follows from \eqref{qu1d201} and Taylor's expansion that
\begin{eqnarray}\label{qu1d202}
F(x_1,t;c_*)= \frac{1}{2}(1+\gamma) (t-c_*)^2-\frac{1}{(4m)!} \frac{c_*^2 a^{(4m)}(0)}{a(0)}x_1^{4m} + G_m(x_1,t-c_*),
\end{eqnarray}
where for some positive constants $C_3$ and $\sigma_3$
\be\no
|G_m(x_1,t-c_*)|\leq C_3 (|t-c_*|^3+|t-c_*||x_1|^{4m}+|x_1|^{4m+1}),\ \  \text{for any } |t-c_*|+|x_1|\leq \sigma_3.
\ee
We will find the solution $\bar{u}(x_1)$ to \eqref{qu1d} with the form $\bar{u}(x_1)=c_*+x_1^{2m} y(x_1)$, where $y$ is defined on $[L_0,L_1]$ with a discontinuity at $x_1=0$:
\be\no
y(x_1)=\begin{cases}
y_-(x_1)<0,\ \ x_1\in [L_0, 0),\\
y_+(x_1)>0,\ \ x_1\in (0, L_1].
\end{cases}\ee
This, together with \eqref{qu1d202}, implies that the equation $F(\bar{u}(x_1),x_1;c_*)=0$ becomes
\be\no
y^2(x_1)-\frac{2}{(\gamma+1)}\frac{a^{(4m)}(0)}{(4m)!} \frac{c_*^2}{a(0)}+ \frac{2}{\gamma+1} H_m(x_1, y(x_1))=0,
\ee
where
\be\no
|H_m(x_1,y)|=\left|\frac{1}{x_1^{4m}}G_m(x_1, x_1y)\right|\leq C_2(|x_1|^{2m} |y|+ |x_1|),\ \ \text{for any }|x_1|\leq \sigma_3.
\ee
Since $y_-(x_1)<0$ and $y_+(x_1)>0$ for $x_1\in [-\sigma_3,0]$ and $x_1\in [0,\sigma_3]$ respectively, thus
\be\no
&&y_-(x_1)=-\sqrt{\frac{2}{(\gamma+1)}\frac{a^{(4m)}(0)}{(4m)!} \frac{c_*^2}{a(0)}  - \frac{2}{\gamma+1} H_m(x_1, y_1(x_1))},\ \ \forall x_1\in [-\sigma_3,0],\\\no
&&y_+(x_1)=\sqrt{\frac{2}{(\gamma+1)}\frac{a^{(4m)}(0)}{(4m)!} \frac{c_*^2}{a(0)} - \frac{2}{\gamma+1} H_m(x_1, y_2(x_1))},\ \ \forall x_1\in [0,\sigma_3].
\ee
The existence and uniqueness of $y_-$ and $y_+$ on $x_1\in [-\sigma_3,0)$ and $x_1\in(0,\sigma_3]$ respectively can also be obtained as in the proof of Proposition \ref{qpo}.

\end{proof}

\begin{remark}
{\it Let $m=1$ in Proposition \ref{qzero2}. Then under the assumption that $a'(0)=a''(0)=a^{(3)}(0)=0, a^{(4)}(0)>0$, there exists a $C^{1,1}$ smooth transonic flow on $[L_0,L_1]$ with zero acceleration at $x_1=0$ where the flow becomes sonic. This corresponds with the result obtained in \cite{WX2019,WX2020}, where regular transonic potential flows with zero acceleration at the sonic points are obtained under the following flatness condition near the throat $a''_{\pm}(x_1)=o(x_1^2)$ as $x_1\to 0$ where $a_{+}(x_1)$ and $a_-(x_1)$ represent the upper and lower walls of the nozzle, respectively.
}\end{remark}

Finally, if $a'(x_1)$ or higher odd order derivatives of $a(x_1)$ has a discontinuity at $x_1=0$, then the following existence result holds.

\begin{proposition}\label{qzero3} Suppose that the initial data $(\rho_0,u_0)$ is subsonic and the function $a(x_1)$ satisfies \eqref{qu1} and \eqref{qu1d4}, where the second condition in \eqref{qu1} is replaced by
\be\label{qu1d301}\begin{cases}
\displaystyle a''(0)=a^{(3)}(0)=\cdots =a^{(2m)}(0)=0, \ \ a^{(2m+1)}(0+)=\lim_{x_1\to 0+}a^{(2m+1)}(x_1)>0,\\
\displaystyle a^{(2m+1)}(0-)=\lim_{x_1\to 0-}a^{(2m+1)}(x_1)<0\ \ \text{for some integer $m\geq 0$}.
\end{cases}\ee
Then there exists a unique $C^{m,\frac12}$ smooth accelerating transonic solution $(\bar{\rho}(x_1), \bar{u}(x_1))\in C^{m,\f12}([L_0,L_1])$ to \eqref{qu1d} such that the solution is subsonic in $[L_0,0)$, supersonic in $(0, L_1]$ with a sonic state at $x_1=0$. The velocity can be represented as $\bar{u}(x_1)=c_*+ |x_1|^{m+\frac12} y(x_1)$, where the function $y$ is defined on $[L_0,L_1]$ with a discontinuity at $x_1=0$:
\be\no
y(x_1)=\begin{cases}
y_-(x_1)<0,\ \ x_1\in [L_0, 0),\\
y_+(x_1)>0,\ \ x_1\in (0, L_1],
\end{cases}\ee
with $y_-\in C^{\infty}([L_0,0])$ and $y_+\in C^{\infty}([0,L_1])$. And $\bar{u}'(0)=\bar{u}''(0)=\cdots= \bar{u}^{(m)}(0)=0$.
\end{proposition}

\begin{proof}

It follows from \eqref{qu1d301} and Taylor's expansion that
\begin{eqnarray}\label{qu1d302}
F(x_1,t;c_*)=\begin{cases}
\frac{1}{2}(1+\gamma) (t-c_*)^2-\frac{1}{(2m+1)!} \frac{\gamma J^{\gamma+1}}{c_*^{\gamma-1}}\frac{a^{(2m+1)}(0-)}{(a(0))^{\gamma}}x_1^{2m+1} + G_m^-(x_1,t-c_*), x_1\in [L_0,0)\\
\frac{1}{2}(1+\gamma) (t-c_*)^2-\frac{1}{(2m+1)!} \frac{\gamma J^{\gamma+1}}{c_*^{\gamma-1}}\frac{a^{(2m+1)}(0+)}{(a(0))^{\gamma}}x_1^{2m+1} + G_m(x_1,t-c_*), x_1\in (0,L_1]
\end{cases}\end{eqnarray}
where for some positive constants $C_3$ and $\sigma_3$
\be\no
|G_m^{\pm}(x_1,t-c_*)|\leq C_3 (|t-c_*|^3+|t-c_*||x_1|^{2m+1}+|x_1|^{2m+2}),\ \  \text{for any } |t-c_*|+|x_1|\leq \sigma_3.
\ee
We will find the solution $\bar{u}(x_1)$ to \eqref{qu1d} with the form $\bar{u}(x_1)=c_*+|x_1|^{m+\f12} y(x_1)$, where $y$ is defined on $[L_0,L_1]$ with a discontinuity at $x_1=0$:
\be\no
y(x_1)=\begin{cases}
y_-(x_1)<0,\ \ x_1\in [L_0, 0),\\
y_+(x_1)>0,\ \ x_1\in (0, L_1].
\end{cases}\ee
This, together with \eqref{qu1d202}, implies that the equation $F(\bar{u}(x_1),x_1;c_*)=0$ becomes
\be\no\begin{cases}
y_-^2(x_1)+\frac{2}{(\gamma+1)}\frac{a^{(2m+1)}(0-)}{(2m+1)!} \frac{c_*^2}{a(0)}+ \frac{2}{\gamma+1} H_m^-(x_1, y(x_1))=0,\ \ x_1\in [-\si_3,0],\\
y_+^2(x_1)-\frac{2}{(\gamma+1)}\frac{a^{(2m+1)}(0+)}{(2m+1)!} \frac{c_*^2}{a(0)}+ \frac{2}{\gamma+1} H_m^+(x_1, y(x_1))=0,\ \ x_1\in [0,\si_3],
\end{cases}\ee
where
\be\no
|H_m^{\pm}(x_1,y)|=\left|\frac{1}{x_1^{2m+1}}G_m^{\pm}(x_1,x_1y)\right|\leq C_2(|x_1|^{m+\f12} |y|+ |x_1|),\ \ \text{for any }|x_1|\leq \sigma_3.
\ee
Since $y_-(x_1)<0$ and $y_+(x_1)>0$ for $x_1\in [-\sigma_3,0]$ and $x_1\in [0,\sigma_3]$ respectively, thus
\be\no
&&y_-(x_1)=-\sqrt{\frac{2}{(\gamma+1)}\frac{-a^{(2m+1)}(0-)}{(2m+1)!} \frac{c_*^2}{a(0)} - \frac{2}{\gamma+1} H_m^-(x_1, y_1(x_1))},\ \ \forall x_1\in [-\sigma_3,0],\\\no
&&y_+(x_1)=\sqrt{\frac{2}{(\gamma+1)}\frac{a^{(2m+1)}(0+)}{(2m+1)!} \frac{c_*^2}{a(0)} - \frac{2}{\gamma+1} H_m^+(x_1, y_2(x_1))},\ \ \forall x_1\in [0,\sigma_3].
\ee
The existence and uniqueness of $y_-$ and $y_+$ on $x_1\in [-\sigma_3,0)$ and $x_1\in(0,\sigma_3]$ respectively can also be obtained as in the proof of Proposition \ref{qpo}.

\end{proof}

\begin{remark}
{\it Let $m=0$ in Proposition \ref{qzero3}. Then under the assumption that $a'(0-)<0, a'(0+)>0$, there exists a unique $C^{0,\frac12}$ H\"{o}lder continuous transonic flow on $[L_0,L_1]$ whose acceleration blows up at the sonic point $x_1=0$. This corresponds with the result obtained in \cite{WX2013}, where $C^{0,\frac12}$ H\"{o}lder continuous transonic potential flows with infinity acceleration at the sonic points was constructed in symmetric converging nozzles with straight wall. Propositions \ref{qzero1}, \ref{qzero2} and \ref{qzero3} further indicates the close relation between the degeneracy rate of the velocity field near the sonic points and the degree of the flatness for the nozzle wall near the throat.
}\end{remark}

\begin{remark}
{\it As in \cite{WX2019,WX2020}, for the smooth transonic flows obtained in Propositions \ref{qpo}, \ref{qzero1},\ref{qzero2} and \ref{qzero3}, all the sonic points are exceptional and characteristic degenerate from subsonic region. These are quite different from the smooth transonic spiral flows constructed in \cite{WXY21a,WXY21b} in an annulus, where all the sonic points  are nonexceptional and noncharacteristically degenerate.}
\end{remark}

As an application of Proposition \ref{gat}, one can establish the existence and uniqueness of the transonic shock flow patterns in de Laval nozzles as described in Courant and Friedrichs \cite[Section 147]{Courant1948}: if an
upcoming flow starting from a subsonic state at the entrance, will accelerate due to the converging effect of the nozzle and attain the sonic state at the throat of the nozzle, and become supersonic in the divergent part of the nozzle, to match the prescribed appropriately large pressure at the exit, a shock front must intervene at some place in the divergent part of the nozzle and the gas is compressed and slowed down to subsonic speed.

The mathematical formulation of such a transonic shock phenomena is as follows. One looks for piecewise smooth functions $(\bar{\rho}^{\pm}, \bar{u}^{\pm})$ defined on $I^-=(L_0, L_s)$, $I^+=(L_{s}, L_1)$ respectively, which solve \eqref{quasi} on $I^{\pm}$ with a shock $x_1=L_{s}\in (0, L_1)$ satisfying the physical entropy condition $[ p({\bar\rho}(L_{s}))]=\bar{p}^+(L_s)-\bar{p}^-(L_s)>0$ and the Rankine-Hugoniot conditions
\be\label{rh}\begin{cases}
[ {\bar \rho} {\bar u}](L_s)=0,\\
[{\bar\rho} {\bar u}^2+P({\bar\rho})] (L_s)=0.
\end{cases}\ee
and also the boundary conditions
\begin{eqnarray}\label{bd1}
&&\rho(L_0)=\rho_0,\ u(L_0)=u_0>0,\\\label{bd2}
&&p(L_1)= p_{e}.
\end{eqnarray}

Then we have the following existence and uniqueness theorem for the transonic shock phenomena described by the quasi one-dimensional model \eqref{quasi}.
\begin{proposition}\label{qu1dshock} Suppose that the initial state $(\rho_0,u_0)$ at $x_1=L_0>0$ is subsonic and the function $a(x_1)$ satisfying \eqref{qu1} and \eqref{qu1d4}. Then there exist two positive constants $0<p_{min}<p_{max}$ such that for the end pressure $p_e\in (p_{min}, p_{max})$, the transonic shock problem described as above has a unique solution $(\bar{\rho}^{\pm},\bar{u}^{\pm})$ in the sense that $(\bar{\rho}^-,\bar{u}^-)$ and $(\bar{\rho}^+,\bar{u}^{+})$ are smooth and defined on $[L_0,0)\cup (0,L_s]$ and $I^+=[L_s, L_1]$ respectively, with a shock located at $x_1=L_{s}\in (0,L_1)$, which satisfy the equations \eqref{quasi}, the Rankine-Hugoniot jump condition \eqref{rh}, and the boundary conditions \eqref{bd1}-\eqref{bd2}.

Moreover, the flow $(\bar{\rho}^-,\bar{u}^-)\in C([L_0,L_s])$ is subsonic on $[L_0, 0)$ and becomes sonic at $x_1=0$ and then accelerates to be supersonic on $(0,L_s]$. The flow $(\bar{\rho}^+,\bar{u}^+)$ is subsonic on $[L_s, L_1]$. In addition, the shock position $x_1=L_s$ increases as the exit pressure $p_{e}$ decreases. Furthermore, the shock position $L_s$ approaches to $L_1$ if $p_{e}$ goes to $p_{min}$ and $L_s$ tends to $0$ if $p_{e}$ goes to $p_{max}$.
\end{proposition}

The existence and uniqueness of a radially symmetric transonic shock solution to the steady Euler system in a divergence sector or circular cone had been proved in \cite{Courant1948,xy08b}. Since the existence and uniqueness of continuous accelerating transonic flows to \eqref{quasi} on $[L_1,L_s)$ has been proved in Proposition \ref{gat}, the existence of transonic shock downstream flow $(\bar{\rho}^{+},\bar{u}^{+})$ and the shock position $x_1=L_s$ can be proved in a similar way as in \cite{Courant1948,xy08b}, which leads to the proof of Proposition \ref{qu1dshock}. The structural stability of multidimensional transonic shocks in the absence of sonic state in flat or divergent nozzles were extensively studied in the past twenty years and have obtained many interesting and important progress. One may refer to \cite{cf03, xy05, xy08a} for the stability of transonic shocks using the potential flows with different kinds of boundary conditions, and refer to \cite{chen08,fx21,lxy09a,lxy09b,lxy13,wxx21} and the references therein for the stability analysis using the steady Euler equations with the exit pressure condition. In particular, in our recent paper \cite{wx23}, we established the existence and stability of cylindrical transonic shock solutions under three dimensional perturbations of the incoming flows and the exit pressure without any further restrictions on the background transonic shock solutions. The strength and position of the perturbed transonic shock are completely determined by the incoming flows and the exit pressure.

\subsection{Smooth transonic flows with nonzero vorticity for the quasi two dimensional steady Euler flow model}\label{formulation}\noindent

Note that the one dimensional smooth transonic flow patterns with positive acceleration at the sonic point $x_1=0$ in Proposition \ref{qpo} are also special solutions to the quasi two dimensional model \eqref{q2deuler}. In this section, we further establish the structural stability of such transonic flow patterns under suitable two dimensional perturbations of the boundary conditions at the entrance and exit of the nozzle for the quasi two dimensional model \eqref{q2deuler}.

The boundary conditions we should prescribe take the form
\be\label{qbcs}\begin{cases}
B(L_0,x_2)=B_0+ \epsilon B_{in}(x_2),\ \ \ &\forall x_2\in [-1,1],\\
u_2(L_0,x_2)= \epsilon h_1(x_2), \ \ &\forall x_2\in [-1,1],\\
u_2(x_1, \pm 1)=0, \ \ &\forall x_1\in [L_0,L_1],
\end{cases}\ee
where $B_{in}\in C^{4,\alpha}([-1,1])$, $h_1\in C^{3,\alpha}([-1,1])$ satisfying the compatibility conditions
\be\label{cp1}
h_1(\pm 1)=h_1''(\pm 1)=0,\ \ B_{in}'(\pm 1)= B_{in}^{(3)}(\pm 1)=0.
\ee
Note here that due to the Bernoulli's law \eqref{ber}, it is natural to prescribe the Bernoulli's quantity at the entrance, and the third condition in \eqref{qbcs} is just the slip boundary condition on the walls. We prescribe some restrictions on the flow angle (i.e. the second equation in \eqref{qbcs}) at the entrance, which is physically acceptable and experimentally controllable. The last two boundary conditions in \eqref{qbcs} are also admissible for the linearized mixed type potential equation from the mathematical point of view (see Lemma \ref{qH1estimate}), and are helpful to yield the important basic energy estimates. There is no need to prescribe any boundary conditions at the exit of the nozzle.

The following theorem states structural stability of the quasi one dimensional transonic flow pattern, which also yields the existence and uniqueness of smooth transonic flows with nonzero vorticity and positive acceleration to the quasi two dimensional model \eqref{q2deuler}.
\begin{theorem}\label{2dmain}
{\it Let $(\bar{\rho},\bar{u})$ be a smooth transonic flow with positive acceleration at the sonic $x_1=0$ given in Proposition \ref{qpo}. Assume that $\gamma>1$, $h_1\in C^{3,\alpha}([-1,1])$ and $B_{in}\in C^{4,\alpha}([-1,1])$ for some $\alpha\in (0,1)$ satisfy \eqref{cp1}. Then there exists a small
constant $\epsilon_0$ depending on the background flow and the boundary datum $h_1, B_{in}$, such that for any $0<\epsilon<\epsilon_0$, the problem \eqref{q2deuler} with \eqref{qbcs} has a unique smooth transonic solution with nonzero vorticity $(u_1,u_2, B)\in (H^3(\Omega))^2\times H^4(\Omega)$, which satisfies the estimate
\be\label{q2d10}
\|u_1-\bar{u}\|_{H^3(\Omega)}+\|u_2\|_{H^3(\Omega)}+\|B-B_0\|_{H^4(\Omega)}\leq C\epsilon,
\ee
for some constant $C$ depending only on the background flow and the boundary datum.

Moreover, all the sonic points form a $C^1$ smooth curve given by $x_1=\xi(x_2)\in C^{1}([-1,1])$. The sonic curve is closed to the background sonic line $x_1=0$ in the sense that
\be\label{qsonic}
\|\xi(x_2)\|_{C^1([-1,1])}\leq C\epsilon.
\ee
}\end{theorem}

\br\label{qsub-reg}
{\it In fact, in the uniformly subsonic flow region $\Omega_{us}$, the regularity of the transonic flows can be improved to be $(u_1,u_2, B)\in
(C^{3,\alpha}(\overline{\Omega_{us}}))^2\times C^{4,\alpha}(\overline{\Omega_{us}})$, where $\Omega_{us}=\{(x_1,x_2): L_0< x_1< -\eta_0, x_2\in [-1,1]\}$ for any $0<\eta_0<|L_0|$.}
\er

\br\label{nonzerovorticity}
{\it
Compared with the existence results of continuous subsonic-sonic or smooth transonic flows obtained in \cite{WX2013,WX2019,WX2020}, the flow constructed in Theorem
\ref{2dmain} has nonzero vorticity and positive acceleration and its sonic curve may not be corresponding to the throat of the nozzle.
}\er

The development of the mathematical theory of transonic flows is closely related to the studies on the boundary value problem for the mixed type partial differential equations. Tricomi \cite{tricomi23} initiated the investigation of the well-posedness of the boundary value problem to the famous Tricomi equation $x_2 \p_{x_1}^2 u+ \p_{x_2}^2 u=0$, which is a mixed elliptic-hyperbolic type PDE. Frankl \cite{frankl45} first revealed the closed connection between this theory and the transonic flow dynamics and attracted much attention of many mathematicians since then. Friedrichs \cite{Friedrichs1958} developed a
general and powerful theory for positive symmetric systems of first order and there have been many important further progress and applications to boundary value problems for equations of mixed type. Kuzmin \cite{Kuzmin2002} had investigated the nonlinear perturbation problem of an accelerating smooth transonic
irrotational basic flow with some artificial boundary conditions in the potential and stream function plane. However, the existence of such a basic flow to the Chaplygin equation was not shown and the physical meaning of the boundary conditions was also not clear. Utilizing the compensated compactness, the subsonic-sonic limit to the 2-D or three dimensional axisymmetric steady irrotational flows were proved in \cite{cdsw07,xx07} and later on these results were extended to the multidimensional potential flows and steady Euler flows cases in \cite{chw16,hww11}. Subsonic and subsonic-sonic spiral flows outside a porous body were obtained recently in \cite{wz21}. However, the solutions obtained by the subsonic-sonic limit only satisfy the equations in the sense of distribution and there is no information about the regularity and degeneracy properties near sonic points and their distribution in flow region.

The authors in \cite{WX2013,WX2016,WX2019,WX2020} have established the existence and uniqueness of regular subsonic-sonic flows and smooth transonic flows of Meyer type in De Laval nozzles with a detailed description of sonic curve for irrotational steady two dimensional Euler equations. Courant and Friedrichs \cite[Section 104]{Courant1948} had used the hodograph method to find a class of spiral flows which may change smoothly from subsonic to supersonic or vice verse and these can take place only outside a limiting circular cylinder where the Jacobian of the hodograph transformation is zero. In \cite{WXY21a}, the authors have further examined this class of radially symmetric transonic flows with nonzero angular velocity in an annulus and analyzed their special properties, whose structural stability with respect to the perturbations of suitable boundary conditions was investigated in \cite{WXY21b}, and the existence and uniqueness of smooth transonic flows with nonzero vorticity were established by the multiplier method and the deformation-curl decomposition to the steady Euler equations. There is also an interesting work on the stability analysis for one dimensional smooth transonic accelerating flows to the steady Euler-Poisson system \cite{bdx}.


We now discuss some key ingredients in our analysis for Theorem \ref{2dmain}. We will combine the approach initiated in \cite{Kuzmin2002} with the technique developed in \cite{WXY21b} to construct a class of smooth transonic rotational flow adjacent to the background transonic flows in Proposition \ref{qpo}.

It is worthy to point out the main differences between the current case and the one in \cite{WXY21b}. In \cite{WXY21b}, we have proved some class of smooth transonic steady Euler flows in annulus and concentric cylinders with nonzero angular velocity and nonzero vorticity, where the key element of analysis is based on a linear mixed type second order equation of Tricomi type, which takes the form (after a coordinate transformation)
\be\label{mix-eq1}\begin{cases}
\p_{y_1}^2 \phi + k_{b22}(y_1) \p_{y_2}^2 \phi + k_{b1}(y_1) \p_{y_1}\phi= F(y_1,y_2),\ (y_1,y_2)\in [r_0,r_1]\times \mathbb{T}_{2\pi}\\
r_0\partial_{y_1}\phi(r_0,y_2)+(r_0f'(r_0)-l_0)\partial_{y_2}\phi(r_0,y_2)=g_2(y_2),\ \forall y_2\in \mathbb{T}_{2\pi},\\
\partial_{y_2}\phi(r_1,y_2)=g_3(y_2),\ \ \forall y_2\in\mathbb{T}_{2\pi},\ \ \ \ \phi(r_1,0)=0,
\end{cases}\ee
where $k_{b22}(y_1)=\frac{1-|{\bf M}_b(y_1)|^2}{y_1^2(1-M_{b1}^2(y_1))^2}$ changes sign when crossing the sonic curve and $\mathbb{T}_{2\pi}$ is a 1-d torus with period $2\pi$. While in the present case, the basic linear mixed type second order equation is of Keldysh type, which reads as
\be\label{mix-eq2}\begin{cases}
\bar{k}_{11}(x_1)\p_{x_1}^2 \psi + \p_{x_2}^2 \psi + \bar{k}_1(x_1) \p_{x_1}\psi= F(x_1,x_2),\ \ \forall (x_1,x_2)\in (L_0,L_1)\times (-1,1),\\
\p_{x_2} \psi(L_0,x_2)=h_1(x_2),\ \ x_2\in [-1,1],\\
\partial_{x_2}\psi(x_1,\pm 1)=0,\ \ \ \ \ \psi(L_0,-1)=0,
\end{cases}\ee
where $\bar{k}_{11}(x_1)=1-\bar{M}^2(x_1)$ changes sign when crossing the sonic curve. These two different kinds of degeneracies cause several essential differences in the analysis:
\begin{enumerate}[(i)]
  \item Different boundary conditions are needed at the entrance and exit of the flow region for \eqref{mix-eq1} and \eqref{mix-eq2}. Some restrictions on the flow angles at the entrance and exit must be prescribed for \eqref{mix-eq1} and no boundary conditions at the exit is required in \eqref{mix-eq2}. The basic $H^1$ energy estimates to \eqref{mix-eq1} and \eqref{mix-eq2} follow from the same strategy by finding out an appropriate multiplier based on some special properties of the background transonic flows. The multiplier for \eqref{mix-eq2} is just a linear function (not the exponential function used in \cite{Kuzmin2002}), which provides stronger energy estimates and thus simplifies some arguments in \cite{Kuzmin2002}. It should be noted that the positive acceleration of the background transonic flow is crucial to obtain the basic $H^1(\Omega)$ estimate to \eqref{mix-eq2}. Indeed, for the corresponding mixed type equation obtained by linearizing at the smooth transonic flows with zero acceleration at the sonic point given in Propositions \ref{qzero1} and \ref{qzero2}, one can derive only a weighted $H^1$ energy estimate with a weight degenerating at the sonic point (See Remark \ref{zero-energy}), which is insufficient for the stability of the nonlinear problem.

  \item The construction of the approximate solutions to \eqref{mix-eq1} follows from the finite Fourier series approximation, the energy estimates and the Fredholm alternative theorem for second order elliptic equations. However, due to the degeneracy of $\bar{k}_{11}(x_1)$ at the sonic point $x_1=0$, to show the existence of a strong solution to \eqref{mix-eq2}, we will use the strategy of \cite{Kuzmin2002} by adding a third order dissipation term $\sigma \p_{x_1}^3\psi$ to \eqref{mix-eq2}. Yet, to gain better uniform regularity estimates, different from \cite{Kuzmin2002}, we supplement the approximate equations with two new boundary conditions $\p_{x_1}^2 \psi(L_0,\cdot)= \p_{x_1}^2 \psi(L_1,\cdot)=0$, not the one $\p_{x_1} \psi(L_0,\cdot)= \p_{x_1} \psi(L_1,\cdot)=0$ used in \cite{Kuzmin2002}, which yields weaker boundary layers than that of \cite{Kuzmin2002} and enables us to obtain a uniform $H^2$ energy estimate with respect to $\sigma>0$. This leads to the $H^2$ strong solution to \eqref{mix-eq2} by the weak convergence.

  \item The higher order $H^4$ energy estimates to the linear mixed type equation \eqref{mix-eq2} are much more involved than those for \eqref{mix-eq1}. We extend the problem \eqref{mix-eq2} to an auxiliary problem in a larger domain where the governing equation is elliptic near the exit of nozzle. The solution to the auxiliary problem coincides with that of the original problem in nozzles. A cut-off technique is employed to derive estimates for the higher order derivatives to the auxiliary problem on the transonic region. Finally, we improve the estimate so that the constant in the $H^4(\Omega)$ estimate obtained depends only on the $H^3(\Omega)$ norm of the coefficients in the linear mixed type equations.
\end{enumerate}

To extend the stability analysis of the smooth transonic irrotational flows to the transonic rotational flows, we use the deformation-curl decomposition to the quasi two dimensional model \eqref{q2deuler} to effectively decouple the hyperbolic and elliptic modes. The deformation-curl decomposition to the steady Euler equations is developed by the authors in \cite{WengXin19,weng2019}. The vorticity is resolved by an algebraic
equation for the Bernoulli's function and there is a loss of one derivative in the equation for the vorticity when dealing with transonic flows. Similar to our previous work \cite{WXY21b}, we design an elaborate two-layer iteration scheme by choosing some appropriate function spaces. We utilize the advantage of one order higher regularity of the stream function than the velocity in the whole flow region to represent the Bernoulli's function as a function of the stream function. However, this function involves the inverse function of the restriction of the stream function at the entrance. There is still a loss of $\frac{1}{2}$ derivatives if the stream function belongs to $H^4(\Omega)$ only. We further observe that the regularity of the flows in the subsonic region can be improved be $C^{3,\alpha}$ if the data at the entrance have better $C^{3,\alpha}$ regularity so that the regularity of the stream function near the entrance are improved to be $C^{4,\alpha}$. Thus we finally recover the loss of the derivative.

The rest of this paper will be arranged as follows. In Section \ref{irrotational}, we establish the basic and higher order energy estimates to the linearized mixed potential
equations and construct approximated solutions by a Galerkin method. In Section \ref{rotation} we employ the deformation-curl decomposition for the quasi two dimensional steady Euler
flow model \eqref{q2deuler} and design a two-layer iteration to demonstrate the existence of smooth transonic rotational flows.

\section{The stability analysis within the irrotational flows}\label{irrotational}\noindent

In this section, we first consider the structural stability of the background transonic flows within the irrotational flows. Thus, consider a smooth flow with zero vorticity $\omega= \p_1 u_2-\p_2 u_1\equiv 0$. The quasi two dimensional model \eqref{q2deuler} and the boundary conditions \eqref{qbcs} can be reduced to
\begin{eqnarray}\label{q2dpo}
\begin{cases}
\p_{x_1} (a(x_1) \rho u_1)+ \p_{x_2}(a(x_1)\rho u_2)=0,\\
\p_{x_1} u_2- \p_{x_2} u_1\equiv 0,\\
B_0=\frac{1}{2} |{\bf u}|^2 + h(\rho)
\end{cases}
\end{eqnarray}
with
\be\label{qpbcs}\begin{cases}
u_2(L_0,x_2)= \epsilon h_1(x_2),\\
u_2(x_1, \pm 1)=0,
\end{cases}\ee
where $h_1\in H^3([-1,1])$ satisfying the compatibility conditions
\be\label{cp2}
h_1(\pm 1)=h_1''(\pm 1)=0.
\ee
The main result in this section can be stated as follows.
\begin{theorem}\label{q-irro}
{\it Let $(\bar{\rho},\bar{u})$ be the smooth transonic flow with positive acceleration at the sonic $x_1=0$ given in Proposition \ref{qpo}. Assume that $\gamma>1$ and $h_1\in H^3([-1,1])$ satisfies \eqref{cp2}. Then there exists a small
constant $\epsilon_0$ depending on the background flow and the boundary datum $h_1$, such that for any $0<\epsilon<\epsilon_0$, the problem \eqref{q2dpo} with \eqref{qpbcs} has a unique smooth transonic irrotational solution $(u_1,u_2)\in (H^3(\Omega))^2$ with the estimate
\be\label{qir1}
\|u_1-\bar{u}\|_{H^3(\Omega)}+\|u_2\|_{H^3(\Omega)}\leq C\epsilon,
\ee
for some constant $C$ depending only on the background flow and the boundary datum.

Moreover, all the sonic points form a $C^1$ smooth curve given by $x_1=\xi(x_2)\in C^{1}([-1,1])$. The sonic curve is closed to the background sonic line $x_1=0$ in the sense that
\be\label{qso1}
\|\xi(x_2)\|_{C^1([-1,1])}\leq C\epsilon.
\ee
}\end{theorem}

\begin{remark}
{\it For irrotational flows, the regularity requiremtn of the boundary data $h_1$ is weaken to be $h_1\in H^{3}([-1,1])$ in Theorem \ref{q-irro}.
}\end{remark}

In this section and section \ref{rotation}, a background flow always refers to the smooth transonic flows with positive acceleration at the sonic $x_1=0$ given in Proposition \ref{qpo} unless specified otherwise.

We start to prove Theorem \ref{q-irro}. It follows from the second equation in \eqref{q2dpo} that there exists a potential function $\varphi=\varphi(x_1,x_2)$ such that $u_i=\p_{x_i}\varphi$ for $i=1,2$. Then the density can be represented as a function of $|\nabla \varphi|^2$:
\be\label{q13}
\rho=\rho(|\nabla\varphi|^2)=\left(\frac{\gamma-1}{\gamma}\right)^{\frac{1}{\gamma-1}}\left(B_0-\frac12 |\nabla \varphi|^2\right)^{\frac{1}{\gamma-1}}.
\ee

Substituting \eqref{q13} into the continuity equation leads to
\be\label{q15}
(c^2(\rho)-(\p_{x_1}\varphi)^2)\p_{x_1}^2\varphi- 2 \p_{x_1} \varphi \p_{x_2} \varphi \p_{x_1x_2}^2 \varphi + (c^2(\rho)-(\p_{x_2}\varphi)^2)\p_{x_2}^2\varphi + b(x_1)c^2(\rho)\p_{x_1} \varphi=0,
\ee
where $c^2(\rho)=(\gamma-1)(B_0-\frac{1}{2}|\nabla \varphi|^2)$. For the 1-D background solution, $\bar{\varphi}=\bar{\varphi}(x_1)=\int_{L_0}^{x_1} \bar{u}(s) ds$ solves
\be\label{q16}
(c^2(\bar{\rho})-(\p_{x_1} \bar{\varphi}))\p_{x_1}^2 \bar{\varphi} + b(x_1)c^2(\bar{\rho})\p_{x_1} \bar{\varphi}=0.
\ee

Denote $\psi_1= \varphi-\bar{\varphi}$. Then $\psi$ satisfies
\be\label{q18}\begin{cases}
\displaystyle\sum_{i,j=1}^2 k_{ij}(\nabla\psi_1) \p_{x_ix_j}^2 \psi_1 + k_1(\nabla\psi)\p_{x_1} \psi_1= G(\nabla\psi_1),\\
\p_{x_2}\psi_1(L_0,x_2)= \epsilon h_1(x_2),\  \ \forall x_2\in (-1,1), \ \psi_1(L_0,-1)=0,\\
\p_{x_2}\psi_1(x_1,-1)=\p_2\psi(x_1,1)=0,\ \ \ \forall x_1\in (L_0,L_1).
\end{cases}\ee
where
\be\label{q171}\begin{cases}
k_{11}(\nabla\psi_1)= \frac{c^2(\rho)-(\bar{u}+ \p_{x_1} \psi_1)^2}{c^2(\rho)-(\p_{x_2}\psi_1)^2},\ \ \ k_{12}(\nabla\psi_1)=k_{21}(\nabla\psi_1)=-\frac{(\bar{u}+ \p_{x_1} \psi_1)\p_{x_2}\psi_1}{c^2(\rho)-(\p_{x_2}\psi_1)^2}, \\
k_{22}(\nabla\psi_1)\equiv 1,\ \ k_{1}(\nabla\psi_1)=\frac{b(x_1) (c^2(\rho)-(\gamma-1)\bar{u}^2)- (\gamma+1) \bar{u}(x_1)\bar{u}'(x_1)}{c^2(\rho)-(\p_{x_2}\psi_1)^2},\\
G(\nabla\psi_1)=\frac{\bar{u}'(x_1)\left((\gamma+1)(\p_{x_1}\psi_1)^2+ (\gamma-1)(\p_{x_2}\psi_1)^2\right)}{2 [c^2(\rho)-(\p_{x_2}\psi_1)^2]}+\frac{(\gamma-1)b(x_1) \bar{u}(x_1)|\nabla \psi_1|^2}{2[c^2(\rho)-(\p_{x_2}\psi_1)^2]},\\
c^2(\rho)=(\gamma-1)[B_0-\frac{1}{2}((\bar{u}+\p_{x_1} \psi_1)^2+(\p_{x_2}\psi_1)^2)].
\end{cases}\ee

Define a monotonic decreasing cut-off function $\eta_0\in C^{\infty}([L_0,L_1])$ such that $0\leq \eta_0(x_1)\leq 1$ for all $x_1\in [L_0,L_1]$ and
\be\label{eta}
\eta_0(x_1)=\begin{cases}
1,\ \ \ & L_0\leq x_1\leq \frac{15L_0}{16},\\
0,\ \ \ & \frac{7L_0}{8}\leq x_1\leq L_1.
\end{cases}\ee
Set $\psi(x_1,x_2)=\psi_1(x_1,x_2)- \epsilon \psi_0$, where $\psi_0=\eta_0(x_1)\int_{-1}^{x_2} h_1(s) ds$. Then
\begin{eqnarray}\label{qlinearized1}\begin{cases}
\mathcal{L}\tilde{\psi}:=\displaystyle\sum_{i,j=1}^2k_{ij}(\nabla \psi+ \epsilon \nabla \psi_0) \p_{ij}^2 \psi +  k_1(\nabla \psi+ \epsilon \nabla \psi_0)\p_{x_1} \psi= G_0(\nabla \psi),\ \ & \text{in } \Omega,\\
\psi(L_0,x_2)= 0,\ \ \ & x_2\in (-1,1),\\
\p_{x_2}\psi(x_1,-1)=\p_{x_2}\psi(x_1,1)=0,\ \ \ \ &\forall x_1\in (L_0,L_1),
\end{cases}\end{eqnarray}
where
\be\label{g0}
G_0(\nabla \psi)=G(\nabla \psi+\epsilon \nabla \psi_0)-\epsilon\bigg(\sum_{i,j=1}^2 k_{ij}(\nabla \psi+ \epsilon \nabla \psi_0)\p_{x_ix_j}^2 \psi_0 + k_1(\nabla \psi+ \epsilon \nabla \psi_0)\p_{x_1}\psi_0\bigg).
\ee

Define
\begin{eqnarray}\nonumber
\Sigma_{\delta_0}=\left\{\psi\in H^4(\Omega): \|\psi\|_{H^4(\Omega)}\leq \delta_0, \p_{x_2}\psi(x_1,\pm 1)=\p_{x_2}^3\psi(x_1,\pm 1)=0\right\},
\end{eqnarray}
where $\delta_0>0$ will be specified later. Note that for $\psi\in\Sigma_{\delta_0}$, $\p_{x_2}^3\psi\in H^1(\Omega)$, and $\p_{x_2}^3\psi(x_1,\pm 1)=0$ hold in the sense of trace. For any given $\hat{\psi}\in \Sigma_{\delta_0}$, we define a mapping $\mathcal{T}$ from $\Sigma_{\delta_0}$ to itself by solving the following boundary value problem for a linearized mixed type second order equations
\begin{eqnarray}\label{qlinearized2}\begin{cases}
\displaystyle\sum_{i,j=1}^2 k_{ij}(\nabla\hat{\psi}+\epsilon \nabla \psi_0) \p_{x_ix_j}^2 \psi + k_1(\nabla\hat{\psi}+\epsilon \nabla \psi_0)\p_{x_1} \psi= G_0(\nabla\hat{\psi}),\\
\p_{x_2}\psi(L_0,x_2)= \epsilon h_1(x_2),\ \ \forall x_2\in (-1,1), \ \psi(L_0,-1)=0,\\
\p_{x_2}\psi(x_1,-1)=\p_2\psi(x_1,1)=0,\ \ \ \forall x_1\in (L_0,L_1).
\end{cases}\end{eqnarray}

Since $\hat{\psi}\in \Sigma_{\delta_0}$, the coefficients $k_{1i}(\nabla \hat{\psi}+\epsilon \nabla \psi_0), i=1,2$ and $k_1(\nabla \hat{\psi}+\epsilon \nabla \psi_0)$ satisfy
\begin{eqnarray}\label{qcoe}
\begin{cases}
\|k_{11}(\nabla \hat{\psi}+\epsilon \nabla \psi_0)-\bar{k}_{11}\|_{H^3(\Omega)}+\|k_{12}(\nabla \hat{\psi})\|_{H^3(\Omega)}\leq C_0(\epsilon+\delta_0),\ \\
\|k_{1}(\nabla \hat{\psi}+\epsilon \nabla \psi_0)-\bar{k}_1\|_{H^3(\Omega)} \leq C_0(\epsilon+\delta_0),\ \ \|G(\nabla \hat{\psi})\|_{H^3(\Omega)}\leq C_0(\epsilon+\delta_0^2),\ \ \\
\{k_{12}(\nabla \hat{\psi}+\epsilon \nabla \psi_0)\}(x_1,\pm 1)=\p_{x_2}^2\{k_{12}(\nabla \hat{\psi}+\epsilon \nabla \psi_0)\}(x_1,\pm 1)=0,\ \ \ \forall x_1\in [L_0,L_1],\\
\p_{x_2}\{k_{11}(\nabla \hat{\psi}+\epsilon \nabla \psi_0)\}(x_1,\pm 1)=\p_{x_2}\{ k_1(\nabla\hat{\psi}+\epsilon \nabla \psi_0)\}(x_1,\pm 1))=0,\\
\p_{x_2} G(\nabla \hat{\psi})(x_1,\pm 1)=0.
\end{cases}
\end{eqnarray}
where
\be\label{q19}\begin{cases}
\bar{k}_{11}(x_1)=1-\bar{M}_1^2(x_1),\\
\bar{k}_{1}(x_1)=\frac{b(x_1)(c^2(\bar{\rho})-(\gamma-1)\bar{u}^2)- (\gamma+1) \bar{u}(x_1) \bar{u}'(x_1)}{c^2(\bar{\rho}(x_1))}=\frac{1+ \bar{M}^2+(\gamma-1)\bar{M}^4}{1-\bar{M}^2}b(x_1).
\end{cases}\ee

Since $\bar{k}_{11}(x_1)\geq 2\kappa_*>0$ for any $x_1\in [L_0,\frac{L_0}{8}]$ with some positive constant $\kappa_*$, there exists a constant $\delta_*$ such that if $0<\delta_0\leq \delta_*$ in \eqref{qcoe}, then
\be\label{elliptic}
k_{11}(x_1,x_2)\geq \kappa_*>0, \ \ \text{for any }\ \ (x_1,x_2)\in [L_0,\frac{L_0}{8}]\times [-1,1].
\ee

The following properties for the background transonic flows are of great importance in our following stability analysis.
\begin{lemma}
{\it For any given smooth transonic flows constructed in Proposition \ref{qpo} and $\bar{k}_1, \bar{k}_{11}$ defined in \eqref{q19}, there exists a positive number $k_*>0$, such that for any $x_1\in [L_0,L_1]$,
\be\label{apos1}
2\bar{k}_1(x_1)+(2j-1)\bar{k}_{11}'(x_1)\leq -\kappa_*, j=0,1,2,3,
\ee
Thus there exists another positive number $d_0>0$, such that $d(x_1)= 6(x_1-d_0)<0$ and
\be\label{apos2}
(\bar{k}_1+j \bar{k}_{11}') d -\frac{1}{2}(\bar{k}_{11} d)'\geq 3,\ \ \ j=0,1,2,3,
\ee
for any $x_1\in [L_0,L_1]$.
}
\end{lemma}

\begin{proof}

Note that for any $x_1\in [L_0,0)\cup (0,L_1]$
\be\label{q23}
2\bar{k}_1(x_1)+(2j-1)\bar{k}_{11}'(x_1)=\frac{2 +(\gamma-1)\bar{M}^4+ 2j \bar{M}^2(2+(\gamma-1)\bar{M}^2)}{a(x_1)}\frac{a'(x_1)}{1-\bar{M}^2}<0,
\ee
and
\be\no
\displaystyle\lim_{x_1\to 0} \frac{a'(x_1)}{1-\bar{M}^2}&=&-\sqrt{\frac{a(0)a''(0)}{\gamma+1}}<0.
\ee
Thus there exists a positive constant $\kappa_*>0$ depending only on the background flow such that for $x_1\in [L_0,L_1]$
\be\no
2\bar{k}_1(x_1)+(2j-1)\bar{k}_{11}'(x_1)=\frac{2 +(\gamma-1)\bar{M}^4+ 2j \bar{M}^2(2+(\gamma-1)\bar{M}^2)}{a(x_1)}\frac{a'(x_1)}{1-\bar{M}^2}\leq -\kappa_*,\ \ \ j=0,1,2,3.
\ee

Let $d(x_1)= -6 d_0 + 6 x_1$, where $d_0>L_1$ is a constant large enough such that $d(x_1)<0$ for every $x_1\in [L_0,L_1]$ and
\be\label{q221}
&&(\bar{k}_1+j\bar{k}_{11}') d -\frac{1}{2}(\bar{k}_{11} d)'= \frac{1}{2} (2\bar{k}_1+(2j-1)\bar{k}_{11}') d -\frac{1}{2}\bar{k}_{11} d'\\\nonumber
&&= 3(2\bar{k}_1+(2j-1)\bar{k}_{11}') (x_1-d_0)- 3\bar{k}_{11}(x_1)\\\nonumber
&&\geq 3(d_0-x_1) \kappa_* - 3\bar{k}_{11}(x_1)\geq 3, \forall j=0,1,2,3,
\ee
if $d_0>0$ is suitable large.

\end{proof}

The functions $k_{11}, k_{12}, k_1$ and $G_0$ in \eqref{qlinearized2} belonging to $H^3(\Omega)$ can be approximated by a sequence of $C^3(\overline{\Omega})$ smooth functions in $H^3(\Omega)$ which also satisfy the compatibility conditions listed in \eqref{qcoe}. In the following subsections $\S$\ref{se31},$\S$\ref{se32}, and $\S$\ref{se33}, we assume that the coefficients $k_{11}, k_{12}, k_1$ belong to $C^{3}(\overline{\Omega})$ satisfying the properties \eqref{qcoe}.

\subsection{The $H^1$ energy estimates of the linear mixed type second order equation \eqref{qlinearized2}}\label{se31}\noindent

We first derive the following $H^1$ energy estimate for \eqref{qlinearized2}.

\begin{lemma}\label{qH1estimate}
{\it There exists a constant $\delta_*>0$ depending only on the background flow, such that if $0<\delta_0\leq \delta_*$ in \eqref{qcoe}, the classical solution to
\eqref{qlinearized2} satisfies the following energy estimate
\be\label{qH1}
\iint_{\Omega}|\psi(x)|^2+|\nabla\psi(x)|^2 dx  + \int_{-1}^1 (\p_{x_1}\psi(L_0,x_2))^2+\sum_{j=1}^2(\p_{x_j}\psi(L_1,x_2))^2 dx_2\leq C_* \iint_{\Omega} G_0^2 dx,
\ee
where the constant $C_*$ depends only on the $H^3(\Omega)$ norms of the coefficients $k_{11}, k_{12}$ and $k_1$.
}\end{lemma}

\begin{proof}
We use an old but powerful idea, which could be traced back to the positive operator theory developed by Friedrichs \cite{Friedrichs1958}, to find a multiplier
and identify a class of admissible boundary conditions at the entrance and exit, where some key properties of the background flow play a crucial role. Let $d(x_1)=6(x_1-d_0)<0$ for $x_1\in[L_0,L_1]$. Integration by parts leads to
\be\no
&&\iint_{\Omega} d(x_1)\p_{x_1} \psi G_0 dx_1 dx_2=  \iint_{\Omega} d(x_1)\p_{x_1} \psi \mathcal{L}\psi dx_1 dx_2\\\label{13}
&&=\iint_{\Omega}\left(k_1 d-\frac{1}{2}\p_{x_1}(k_{11}d)-\p_{x_2} k_{12} d\right)(\p_{x_1} \psi)^2  + \frac{1}{2} d'(x_1) (\p_{x_2}\psi)^2  dx_1 dx_2\\\no
&&\quad+\frac{1}{2}\int_{-1}^1 \left(k_{11} d(\p_{x_1} \psi)^2 -  d (\p_{x_2} \psi)^2\right)\b|_{x_1=L_0}^{L_1} dx_2 + \int_{L_0}^{L_1} \left(k_{12} d (\p_{x_1} \psi)^2 + d \p_{x_1} \psi \p_{x_2} \psi\right)\b|_{x_2=-1}^1 dx_1.
\ee
Since $k_{12}(x_1,\pm 1)\equiv 0$ and $\p_{x_2}\psi(x_1,\pm 1)=0$ for every $x_2\in [-1,1]$, the last boundary integral vanishes.

Using \eqref{apos1}-\eqref{apos2}, there exists a constant $\delta_*>0$ such that if $0<\delta_0\leq \delta_*$ in \eqref{qcoe}, there holds
\begin{eqnarray}\nonumber
&&k_1 d-\frac12 \p_{x_1}(k_{11}d)- d\partial_{x_2} k_{12}=\bar{k}_1 d -\frac{1}{2}(\bar{k}_{11} d)' + (k_1-\bar{k}_{1})d-\frac{1}{2}\p_{x_1}((k_{11}-\bar{k}_{11})d)- d\partial_{x_2} k_{12}\\\no
&&\geq 3- \|d(k_1-\bar{k}_{1})\|_{L^{\infty}}-\frac{1}{2}\|\p_{x_1}((k_{11}-\bar{k}_{11})d)\|_{L^{\infty}}- \|d\partial_{x_2} k_{12}\|_{L^{\infty}}\geq 2>0, \ \ \forall
(x_1,x_2)\in \Omega,\\\no
&&\frac{1}{2}d'(x_1)\equiv 3,\forall (x_1,x_2)\in \Omega.
\end{eqnarray}
due to the Sobolev embedding $H^3(\Omega)\subset C^{1,\alpha}(\overline{\Omega})$ with $\alpha\in(0,1)$. Note also $d(L_0)<0, d(L_1)<0$, thus it follows from \eqref{13} that
\be\no
\iint_{\Omega} d(x_1)\p_{x_1}\psi G_0 dx_1 dx_2\geq 2\iint_{\Omega} |\nabla \psi|^2 dx_1 dx_2 + \int_{-1}^1 (\p_{x_1}\psi(L_0,x_2))^2+\sum_{j=1}^2(\p_{x_j}\psi(L_1,x_2))^2 dx_2.
\ee
Thus the estimate \eqref{qH1} is obtained.

\end{proof}

\begin{remark}\label{zero-energy}
{\it It should be noted that the positive acceleration of the background flow is crucial to establish the stability estimate \eqref{qH1}. Indeed, if the background flow has zero acceleration as in Propositions \ref{qzero1} and \ref{qzero2}, then the corresponding linearized problem takes the following form
\be\label{deg1}\begin{cases}
\mathcal{L}\psi= \bar{k}_{11}(x_1) \p_{x_1}^2 \psi + \p_{x_2}^2 \psi + \bar{k}_1\p_{x_1} \psi=x_1^{2m} G_0(x_1,x_2), \ \text{in }\  \Omega,\\
\psi(L_0,x_2) = 0,\ \ \forall x_2\in [-1,1],\\
\p_{x_2}\psi(x_1,\pm 1)=0,\ \ \ \forall x_1\in [L_0,L_1],
\end{cases}\ee
where we rewrite the source term as $x_1^{2m} G_0(x_1,x_2)$ which comes from the definition of $G(\nabla\psi)$ in \eqref{q171}.

Same as in Lemma \ref{qH1estimate}, one can get from \eqref{deg1} that
\be\label{deg2}
&&\iint_{\Omega} x_1^m G_0 d(x_1)\p_{x_1}\psi dx_1 dx_2 = \iint_{\Omega} \mathcal{L} \psi \cdot d(x_1)\p_{x_1}\psi dx_1 dx_2\\\no
&&= \iint_{\Omega} \left(\bar{k}_1 d-\frac{1}{2} (\bar{k}_{11} d)'\right)(\p_{x_1}\psi)^2+ \frac12 d' (\p_{x_2} \psi)^2 dx+  \frac 12\int_{-1}^1 \bar{k}_{11} d(\p_{x_1}\psi)^2-d(\p_{x_2} \psi)^2\bigg|_{x_1=L_0}^{L_1} dx_2.
\ee
Note that
\be\no
2\bar{k}_1(x_1)- \bar{k}_{11}'(x_1)&=&\frac{(2+(\gamma-1)\bar{M}^4(x_1))}{1-\bar{M}^2}b(x_1)=\frac{(2+(\gamma-1)\bar{M}^4(x_1))c^2(\bar{\rho})}{c^2(\bar{\rho})-\bar{u}^2(x_1)}\frac{a'(x_1)}{a(x_1)}\\\no
&=&\frac{\gamma J^{\gamma-1}(a(x_1))^{1-\gamma}(2+(\gamma-1)\bar{M}^4(x_1))}{a(x_1)}\frac{a'(x_1)}{(a(x_1))^{1-\gamma}(a(0))^{\gamma-1}c_*^{\gamma+1}-(\bar{u}(x_1))^{\gamma+1}}.
\ee
Then it holds that away from $x_1=0$, $2\bar{k}_1(x_1)- \bar{k}_{11}'(x_1)<0$.

Let us take the one dimensional smooth transonic solution given in Proposition \ref{qzero1} for example. Since $a(x_1)$ satisfies the conditions \eqref{qu1} and \eqref{qu1d200} and $\bar{u}(x_1)=c_*+x_1^{2m+1} y(x_1)$, then
\be\no
&&(a(x_1))^{1-\gamma}(a(0))^{\gamma-1}c_*^{\gamma+1}-(\bar{u}(x_1))^{\gamma+1}\\\no
&&=-c_*^{\gamma+1}\frac{(a(x_1))^{\gamma-1}-(a(0))^{\gamma-1}}{(a(x_1))^{\gamma-1}}-(\gamma+1)\int_0^{x_1^{2m+1} y(x_1)} (c_*+s)^{\gamma} ds\\\no
&&=-x_1^{2m+1} Y(x_1),
\ee
where $Y(x_1)>0$ for every $x_1\in [L_0,L_1]$. By \eqref{qu1d200}, $a'(x_1)$ has a form $a'(x_1)=x_1^{4m+1} b_1(x_1)$ with $b_1(x_1)>0$ for every $x_1\in [L_0,L_1]$. Thus one derives that
\be\no
2\bar{k}_1(x_1)- \bar{k}_{11}'(x_1)=- x_1^{2m} Q_1(x_1),\ \ 1-\bar{M}^2(x_1)=- x_1^{2m+1} Q_2(x_1),
\ee
for some smooth positive functions $Q_1(x_1)$ and $Q_2(x_1)$ for $\forall x_1\in [L_0,L_1]$. Hence for $d(x_1)=6(x_1-d_0)$, one may select a
large enough constant $d_0>0$ such that
\be\no
&&\bar{k}_1 d-\frac{1}{2}(\bar{k}_{11}d)'=3(x_1-d_0)(2\bar{k}_1(x_1)- \bar{k}_{11}'(x_1))-3(1-\bar{M}^2(x_1))
\\\no
&&=3(d_0-x_1)x_1^{2m} Q_1(x_1)+3x_1^{2m+1} Q_2(x_1)\geq 3 x_1^{2m}.
\ee
It follows from \eqref{deg2} that
\be\label{deg3}
&&\iint_{\Omega} x_1^{2m} (\p_{x_1}\psi)^2 + (\p_{x_2}\psi)^2 dx_1 dx_2 \leq C_*\iint_{\Omega} x_1^{2m} G_0^2(x_1,x_2) dx_1 dx_2.
\ee
In contrast to the uniform $H^1$ energy estimate \eqref{qH1}, \eqref{deg3} yields only a weighted energy estimate with a weight degenerating at the sonic point $x_1=0$. Furthermore, the solvability of the corresponding nonlinear problem cannot be on the estimate \eqref{deg3} for \eqref{deg1} since it seems to be difficult to derive a priori estimate similar to \eqref{deg3} for solutions to the mixed type equation obtained by linearizing any flow near the background flow in Proposition \ref{qzero1}, due to the unknown location of the sonic curve and the rate at which the linearized equation degenerates in general. While for the smooth one dimensional transonic flow with positive acceleration, the inequality \eqref{apos2} is stable under small perturbations of $k_{11}, k_{12}$ and $k_1$ satisfying \eqref{qcoe}, the above two issues do not cause any essential difficulties in this case.
}
\end{remark}

\subsection{The construction of the approximated solutions}\label{se32}\noindent

To prove the existence of strong solutions to the problem \eqref{qlinearized2}, we use the Galerkin method with Fourier series approximation. Thanks to the degeneracy of the coefficient $k_{11}$ near the sonic curve, following the idea introduced in \cite{Kuzmin2002}, we first consider the following singular perturbation problem to \eqref{qlinearized2} with an additional third order dissipation term and two additional boundary conditions:
\be\label{qap}\begin{cases}
\mathcal{L}^{\sigma} \psi^{\sigma}=\sigma \p_{x_1}^3 \psi^{\sigma} + \displaystyle\sum_{i,j=1}^2 k_{ij} \p_{x_ix_j}^2\psi^{\sigma} + k_1 \p_{x_1} \psi^{\sigma} = G_0(x_1,x_2),\\
\psi^{\sigma}(L_0,x_2)=0,\ \ \p_{x_1}^2\psi^{\sigma}(L_0,x_2)=0,\\
\p_{x_2}\psi^{\sigma}(x_1,\pm 1)=0,\\
\p_{x_1}^2\psi^{\sigma}(L_1,x_2)=0.
\end{cases}\ee
Note that though the third order perturbations of the equation in \eqref{qap} is same as in \cite{Kuzmin2002}, yet the additional boundary conditions $\p_{x_1}^2\psi^{\sigma}(L_0,x_2)=\p_{x_1}^2\psi^{\sigma}(L_1,x_2)=0$ are different from the ones in \cite{Kuzmin2002}. This choice of the boundary conditions leads to weaker boundary layers than those in \cite{Kuzmin2002}, and thus enables us to establish a $H^2$ energy estimate to \eqref{qap} which is uniform with respect to $\sigma$. Therefore a global $H^2(\Omega)$ estimate for the weak solution $\psi$ to \eqref{qlinearized2} can be derived directly.

\begin{lemma}\label{qH1-appro}
{\it There exists a constant $\delta_*>0$ depending only on the background flow, such that if $0<\delta_0\leq \delta_*$ in \eqref{qcoe}, the classical solution to
\eqref{qap} satisfies the following energy estimate
\be\label{qH1-appro0}
&&\sigma \iint_{\Omega} |\p_{x_1}^2\psi^{\sigma}|^2 dx_1 dx_2+\iint_{\Omega}|\psi^{\sigma}|^2+|\nabla\psi^{\sigma}|^2 dx  \\\no
&&\quad\quad+ \int_{-1}^1 (\p_{x_1}\psi^{\sigma}(L_0,x_2))^2+\sum_{j=1}^2(\p_{x_j}\psi^{\sigma}(L_1,x_2))^2 dx_2
\leq C_* \iint_{\Omega} G_0^2 dx,\\\label{qH2-appro}
&&\iint_{\Omega}|\nabla^2\psi^{\sigma}|^2 dx \leq C_* \iint_{\Omega} G_0^2+|\nabla G_0|^2 dx,
\ee
where the constant $C_*$ depends only on the $H^3(\Omega)$ norms of the coefficients $k_{11}, k_{12}$ and $k_1$.
}\end{lemma}

\begin{proof}

We omit the superscript $\sigma$ in the following argument. Choosing the same multiplier as in Lemma \ref{qH1estimate} yields
\be\no
&&\iint_{\Omega} d(x_1)\p_{x_1} \psi G_0 dx_1 dx_2=  \iint_{\Omega} d(x_1)\p_{x_1} \psi \mathcal{L}^{\sigma}\psi dx_1 dx_2\\\no
&&=\iint_{\Omega}-\sigma d (\p_{x_1}^2\psi)^2- 6 \sigma \p_{x_1}^2 \psi \p_{x_1} \psi+ \left(k_1 d-\frac{1}{2}\p_{x_1}(k_{11}d)-\p_{x_2} k_{12} d\right)(\p_{x_1} \psi)^2+ \frac{1}{2}d'(x_1) (\p_{x_2}\psi)^2  dx\\\no
&&\quad+\int_{-1}^1\sigma d \p_{x_1}\psi \p_{x_1}^2 \psi \b|_{x_1=L_0}^{L_1}dx_2+\frac{1}{2}\int_{-1}^1 \left(k_{11} d(\p_{x_1} \psi)^2 - d (\p_{x_2}
\psi)^2\right)\b|_{x_1=L_0}^{L_1} dx_2.
\ee

Since $d(x_1)=6x_1-6 d_0<0$ for all $x_1\in [L_0,L_1]$, the first term is a positive term, the first boundary term vanishes due to \eqref{qap}, and
\be\no
6\sigma \iint_{\Omega} \p_{x_1} \psi\p_{x_1}^2\psi dx_1 dx_2\leq \frac{-\sigma}{4}\iint_{\Omega}d(x_1) (\p_{x_1}^2\psi)^2 dx_1 dx_2 + 36\sigma\iint_{\Omega} (\p_{x_1} \psi)^2 dx_1 dx_2.
\ee

As in Lemma \ref{qH1estimate}, for sufficiently small $\sigma>0$, there holds
\be\label{q2d22}
&&\sigma \iint_{\Omega} |\p_{x_1}^2\psi|^2 dx_1 dx_2 + \iint_{\Omega} |\nabla \psi|^2 dx_1 dx_2
\\\no
&&\quad\quad +\int_{-1}^1 (\p_{x_1}\psi(L_0,x_2))^2+\sum_{j=1}^2(\p_{x_j}\psi(L_1,x_2))^2 dx_2\leq C_*\iint_{\Omega} G_0^2 dx_1 dx_2.
\ee
Since $\psi(L_0,x_2)\equiv 0$, \eqref{qH1-appro0} follows.

Choose a monotonic decreasing cut-off function $\eta_1\in C^{\infty}([L_0,L_1])$ such that $0\leq \eta_1(x_1)\leq 1$ for all $x_1\in [L_0,L_1]$ and
\be\no
\eta_1(x_1)=\begin{cases}
1,\ \ \ & L_0\leq x_1\leq \frac{L_0}{2},\\
0,\ \ \ & \frac{L_0}{4}\leq x_1\leq L_1.
\end{cases}\ee

Multiplying the equation \eqref{qap} by $\eta_1^2 \p_{x_1}^2 \psi$ and integrating by parts give
\be\no
&&\iint_{\Omega} \eta_1^2 (k_{11}(\p_{x_1}^2 \psi)^2 +(\p_{x_1x_2}^2 \psi)^2)- \sigma \eta_1 \eta_1' (\p_{x_1}^2 \psi)^2 dx_1 dx_2 =-\iint_{\Omega} 2 \eta_1^2
k_{12} \p_{x_1x_2}^2 \psi\p_{x_1}^2 \psi \\\no
&& -\iint_{\Omega} 2 \eta_1 \eta_1'  \p_{x_1x_2}^2\psi \p_{x_2}\psi dx_1 dx_2 +\iint_{\Omega}\eta_1^2(G_0(x_1,x_2) - k_1\p_{x_1} \psi) \p_{x_1}^2 \psi dx_1 dx_2.
\ee
Since $\eta_1$ is monotonically decreasing, so $- \iint_{\Omega}\sigma \eta_1 \eta_1' (\p_{x_1}^2 \psi)^2$ is a positive term. Also by \eqref{elliptic}, $k_{11}(x_1,x_2)\geq \kappa_*$ for some $\kappa_*>0$ on $(x_1,x_2)\in [L_0,L_0/8]\times [-1,1]$. It follows from the H\"{o}der's inequality and \eqref{qH1} that
\be\nonumber
&\quad&\kappa_*\int_{L_0}^{\frac{L_0}{2}}\int_{-1}^1 |\p_{x_1}^2 \psi|^2 + |\p_{x_1x_2}^2 \psi|^2 dx_1 dx_2\leq \kappa_*\iint_{\Omega}\eta_1^2( |\p_{x_1}^2 \psi|^2 + |\p_{x_1x_2}^2 \psi|^2) dx_1 dx_2\\\no
&\leq& \frac{1}{2}\kappa_*\iint_{\Omega} \eta_1^2 (|\p_{x_1}^2 \psi|^2 + |\p_{x_1x_2}^2 \psi|^2) dx_1 dx_2+\frac{C}{\kappa_*}\iint_{\Omega}G_0^2(x_1,x_2)+(k_1^2+|\eta_1'|^2)|\nabla \psi|^2 dx_1 dx_2\\\label{qaH21}
&\leq&C_*\iint_{\Omega} G_0^2(x_1,x_2) dx_1 dx_2.
\ee

Choose another monotonic increasing cut-off function $\eta_2\in C^{\infty}([L_0,L_1])$ such that $0\leq \eta_2(x_1)\leq 1$ for all $x_1\in [L_0,L_1]$ and
\be\no
\eta_2(x_1)=\begin{cases}
0,\ \ \ & L_0\leq x_1\leq \frac{3L_0}{4},\\
1,\ \ \ & \frac{L_0}{2}\leq x_1\leq L_1.
\end{cases}\ee

Denote $w_1=\p_{x_1}\psi$. Then $w_1$ solves
\be\label{qappro-w1}\begin{cases}
\displaystyle\sigma \p_{x_1}^3 w_1 + \sum_{i,j=1}^2 k_{ij}\p_{x_i x_j}^2 w_1 + k_3 \p_{x_1} w_1 + k_4 \p_{x_2} w_1=G_1,\\
\p_{x_2} w_1(x_1,\pm 1)=0,\ \ \ \forall x_1\in [L_0,L_1],\\
\p_{x_1} w_1(L_0,x_2)= \p_{x_1} w_1(L_1,x_2)=0,\ \forall x_2\in [-1,1],
\end{cases}\ee
where
\be\no
&&k_3= k_1 + \p_{x_1} k_{11},\ \ \  \ k_4=2 \p_{x_1} k_{12},\ \ \ G_1= \p_{x_1} G_0-\p_{x_1} k_1\p_{x_1} \psi.
\ee

Multiplying the equation in \eqref{qappro-w1} by $\eta_2^2 d(x_1) \p_{x_1} w_1$ and integrating over $\Omega$, one gets after integration by parts that
\be\no
&&\iint_{\Omega}-\sigma\eta_2^2 d (\p_{x_1}^2 w_1)^2- \sigma\p_{x_1}(\eta_2^2 d)\p_{x_1}^2 w_1\p_{x_1} w_1dx_1 dx_2+ \sigma\int_{-1}^1 d \eta_2^2 \p_{x_1} w_1 \p_{x_1}^2 w_1\bigg|_{x_1=L_0}^{L_1} dx_2\\\no
&&+\iint_{\Omega} [\eta_2^2d k_3 -\frac{1}{2}\p_{x_1}(\eta_2^2d k_{11})-\eta_2^2d \p_{x_2} k_{12}](\p_{x_1} w_1)^2 dx_1 dx_2\\\no
&&+\iint_{\Omega} [\frac{\eta_2^2}{2}d'(x_1)+\eta_2\eta_2' d ](\p_{x_2} w_1)^2+ \eta_2^2 k_4d \p_{x_1} w_1 \p_{x_2} w_1  dx_1
dx_2\\\no
&& + \frac{1}{2}\int_{-1}^1\eta_2^2 d [k_{11} (\p_{x_1} w_1)^2- (\p_{x_2}
w_1)^2]\bigg|_{x_1=L_0}^{L_1} dx_2= \iint_{\Omega} \eta_2^2 d \p_{x_1} w_1 G_1
dx_1 dx_2.
\ee
Due to the boundary conditions, the first boundary integral vanishes. It follows from \eqref{apos1}-\eqref{apos2} and \eqref{qcoe} that there exists a constant $\delta_*>0$ such that if $0<\delta_0\leq \delta_*$, one has
\begin{eqnarray}\no
&&\eta_2^2 k_3 d-\frac12 \p_1(\eta_2^2 d k_{11})-\eta_2^2 d\partial_{x_2} k_{12}\\\no
&&=\frac{1}{2}\eta_2^2 d(2 k_1+\p_{x_1} k_{11}-2 \p_2 k_{12})- \frac{1}{2}k_{11}(\eta_2^2 d'+ 2 \eta_2 \eta_2' d)\\\no
&&\geq \frac12\eta_2^2 d(2\bar{k}_1 + \bar{k}_{11}' -\|k_1-\bar{k}_{1}\|_{L^{\infty}}-\|\p_1 k_{11}-\bar{k}_{11}'\|_{L^{\infty}}-\|\partial_{x_2} k_{12}\|_{L^{\infty}})\\\no
&&\quad\quad\quad-\frac{1}{2}k_{11}(\eta_2^2 d'+ 2 \eta_2 \eta_2' d)\geq 2\eta_2^2- k_{11}\eta_2 \eta_2' d, \ \ \forall (x_1,x_2)\in \Omega,\\\no
&&\frac{1}{2}\eta_2^2 d'(x_1)+\eta_2 \eta_2' d = 3\eta_2^2 + \eta_2 \eta_2' d,\ \forall (x_1,x_2)\in \Omega,\\\no
&&|\eta_2^2k_4d|\leq C_*\delta_0\eta_2^2\leq C_* \delta_*\eta_2^2,
\end{eqnarray}
due to the Sobolev embedding $H^3(\Omega)\subset C^{1,\alpha}(\overline{\Omega})$ with $\alpha\in(0,1)$.

 Then using \eqref{qaH21}, one could infer that
\be\no
&&\iint_{\Omega}\eta_2^2 |\nabla w_1|^2 dx_1 dx_2+\sigma\int_{\frac{1}{2}L_0}^{L_1}\int_{-1}^1 (\p_{x_1}^2 w_1)^2 dx_1 dx_2+ \int_{-1}^1 (\p_{x_2}
w_1)^2(L_1,x_2)dx_2\\\no
&&\leq C_*\int_{\frac{3}{4}L_0}^{\frac{1}{2}L_0}\int_{-1}^1 |\nabla w_1|^2+ C_*\iint_{\Omega} G_1^2 dx_1 dx_2.
\ee
This, together with \eqref{qaH21}, shows that
\be\label{qaH22}
\iint_{\Omega} |\p_{x_1}^2\psi|^2 + |\p_{x_1x_2}^2 \psi|^2 dx_1 dx_2 \leq \iint_{\Omega} |G_0|^2 +|\nabla G_0|^2 dx_1 dx_2.
\ee

It remains to estimate $\p_{x_2}^2\psi$. Define $v_1=\p_{x_2} \psi$. Then one has
\be\label{qappro-v1}\begin{cases}
\displaystyle\sigma \p_{x_1}^3 v_1 + \sum_{i,j=1}^2 k_{ij} \p_{x_i x_j}^2 v_1 + (k_1 + 2 \p_{x_2} k_{12})\p_{x_1} v_1=\p_{x_2}G_0-\p_{x_2} k_{11} \p_{x_1}^2 \psi- \p_{x_2} k_1 \p_{x_1}\psi,\\
v_1(L_0,x_2)=\p_{x_1}^2 v_1(L_0,x_2)=0,\ \ \ \forall x_2\in [-1, 1],\\
\p_{x_1}^2 v_1(L_1,x_2)=0,\ \ \ \forall x_2\in [-1, 1],\\
v_1(x_1,\pm 1)=0,\ \ \forall x_1\in [L_0,L_1].
\end{cases}\ee
Multiplying the equation in \eqref{qappro-v1} by $d(x_1)\p_{x_1} v_1$ and integrating over $\Omega$, one gets by integration by parts that
\be\no
&&\iint_{\Omega} d(x_1)\p_{x_1} v_1 (\p_{x_2}G_0-\p_{x_2} k_{11} \p_{x_1}^2 \psi- \p_{x_2} k_1 \p_{x_1}\psi) dx\\\no
&&=\iint_{\Omega}-d \sigma (\p_{x_1}^2v_1)^2-6\sigma \p_{x_1} v_1\p_{x_1}^2 v_1 + \left((k_1+\p_{x_2} k_{12}) d-\frac{1}{2}\p_{x_1}(k_{11}d)\right)(\p_{x_1} v_1)^2+\frac{1}{2}d'(x_1) (\p_{x_2}v_1)^2 dx\\\no
&&\quad+\int_{-1}^1 \sigma d \p_{x_1}v_1 \p_{x_1}^2 v_1\b|_{x_1=L_0}^{L_1}dx_2+\frac{1}{2}\int_{-1}^1 \left(k_{11} d(\p_{x_1} v_1)^2 - d (\p_{x_2}
v_1)^2\right)\b|_{x_1=L_0}^{L_1} dx_2\\\no
&&\quad+ \int_{L_0}^{L_1} \left(k_{12} d (\p_{x_1} v_1)^2 + d \p_{x_1} v_1 \p_{x_2} v_1\right)\b|_{x_2=-1}^1 dx_1.
\ee
Since $k_{11}(L_1,x_2)<0, k_{11}(L_0,x_2)>0$ for any $x_2\in [-1,1]$ and $d(x_1)<0$ for any $x_1\in [L_0,L_1]$, the above equality further implies
\be\no
&&\sigma \iint_{\Omega} |\p_{x_1}^2v_1|^2 dx_1 dx_2 + \iint_{\Omega} |\nabla v_1|^2 dx_1 dx_2 +\int_{-1}^1 (\p_{x_1}v_1(L_0,x_2))^2+\sum_{j=1}^2(\p_{x_j} v_1(L_1,x_2))^2
dx_2\\\label{qaH23}
&&\leq C_*\iint_{\Omega} |\p_{x_2}G_0|^2+ (\p_{x_1}^2 \psi)^2 +|\nabla \psi|^2 dx_1 dx_2\leq \iint_{\Omega} |G_0|^2+|\nabla G_0|^2 dx_1 dx_2.
\ee
The proof of Lemma \ref{qH1-appro} is completed.

\end{proof}

The approximate solutions can be constructed by using the Galerkin's method with Fourier series expansion as follows. Let $\{b_j(x_2)\}_{j=1}^{\infty}$ be a family of all eigenfunctions associated to the eigenvalue problem
\be\label{eigenfunction}\begin{cases}
-u''(x_2)= \lambda u(x_2),\  x_2 \in (-1,1),\\
u'(-1)= u'(1)=0.
\end{cases}\ee
Indeed, one may choose $\{b_j(x_2)\}$ as
\be\no
\{b_j(x_2)\}_{j=1}^{\infty}= \bigg\{\frac{1}{\sqrt{2}}\bigg\}\cup \bigg\{\cos(j\pi x_2)\bigg\}_{j=1}^{\infty} \cup \bigg\{\sin (\frac{2j+1}{2}\pi x_2)\bigg\}_{j=0}^{\infty},
\ee
which forms a complete orthonormal basis in $L^2((-1,1))$ and an orthogonal basis in $H^1((-1,1))$.

Define the approximate solutions as
\be\no
\psi^{N,\sigma}(x_1,x_2)= \sum_{j=1}^N A_j^{N,\sigma}(x_1)b_j(x_2),
\ee
which satisfies the following $N$ linear equations on $(L_0,L_1)$ with boundary conditions
\be\label{qapp1}\begin{cases}
\int_{-1}^1 \mathcal{L}^{\sigma} \psi^{N,\sigma}(x_1,x_2) b_{m}(x_2) dx_2= \int_{-1}^1 G_0(x_1,x_2) b_m(x_2) dx_2,\ m=1,\cdots, N,\\
\psi^{N,\sigma}(L_0,x_2)=\p_{x_1}^2 \psi^{N,\sigma}(L_0,x_2)=0,\\
\p_{x_1}^2 \psi^{N,\sigma}(L_1,x_2)=0.
\end{cases}\ee

Therefore $A_j^{N,\sigma}$ should solve the following boundary value problem for an ordinary differential system
\be\label{qappro1}\begin{cases}
\displaystyle\sigma \frac{d^3}{dx_1^3} A_m^{N,\sigma} +\displaystyle\sum_{j=1}^N  a_{jm}^{N,\sigma}\frac{d^2}{dx_1^2} A_j^{N,\sigma}+  b_{jm}^{N,\sigma}\frac{d}{dx_1}
A_j^{N,\sigma}
+ c_{jm}^{N,\sigma} A_j^{N,\sigma} = G_{0m}(x_1),\ \\
A_{m}^{N,\sigma}(L_0)=\frac{d^2}{dx_1^2}A_m^{N,\sigma}(L_0)=0,\ \ m=1,\cdots, N,\\
\frac{d^2}{dx_1^2}A_m^{N,\sigma}(L_1)=0,
\end{cases}\ee
where
\be\no
&&a_{jm}^{N,\sigma}(x_1)=\int_{-1}^1 k_{11}(x_1,x_2)b_j(x_2)b_m(x_2) dx_2,\\\no
&&b_{jm}^{N,\sigma}(x_1)=\int_{-1}^1 (2 k_{12}(x_1,x_2)b_j'(x_2)b_m(x_2)+k_1(x_1,x_2)b_j(x_2)b_m(x_2) dx_2,\\\no
&&c_{jm}^{N,\sigma}(x_1)=\int_{-1}^1 -\lambda_j b_j(x_2)b_m(x_2)dx_2=-\lambda_j \delta_{jm},\\\no
&&G_{0m}(x_1)= \int_{-1}^1 G_0(x_1,x_2) b_m(x_2) dx_2.
\ee

\begin{lemma}\label{exist1}
{\it There exists a unique smooth solution $\{A_{m}^{N,\sigma}(x_1)\}_{m=1}^N$ to \eqref{qappro1} such that $\psi^{N,\sigma}(x_1,x_2)= \sum_{j=1}^N A_j^{N,\sigma}(x_1)b_j(x_2)$ satisfies the following estimate
\be\label{qaH25}
\iint_{\Omega} |\psi^{N,\sigma}|^2+ |\nabla\psi^{N,\sigma}|^2+ |\nabla^2\psi^{N,\sigma}|^2 dx_1 dx_2\leq C_*\iint_{\Omega} |G_0|^2 + |\nabla G_0|^2 dx_1 dx_2,
\ee
where the constant $C_*$ depends only on the $H^3(\Omega)$ norms of the coefficients $k_{11}, k_{12}$ and $k_1$.
}\end{lemma}

\begin{proof}

Using the functions $d(x_1)$ defined in Lemma \ref{qH1estimate}, multiplying the $m^{th}$ equation in \eqref{qappro1} by $d(x_1)\frac{d}{dx_1} A_m^{N,\sigma}$, summing from
$1$ to $N$, integrating over $[L_0,L_1]$, one can get that
\begin{eqnarray}\label{ODE4}
\iint_{\Omega}(\mathcal{L}^{\sigma} \psi^{N,\sigma}-G_0)d(x_1)\partial_{x_1} \psi^{N,\sigma}dx_1 dx_2=0.
\end{eqnarray}
An integration by parts as in Lemma \ref{qH1-appro} yields
\be\label{qappro2}
&&\sigma \iint_{\Omega} |\p_{x_1}^2 \psi^{N,\sigma}|^2 dx_1 d x_2 + \iint_{\Omega} |\psi^{N,\sigma}|^2 + |\nabla\psi^{N,\sigma}|^2 dx_1 dx_2\\\no
&&\quad +\int_{-1}^1 (\p_{x_1}\psi^{N,\sigma}(L_0,x_2))^2+\sum_{j=1}^2(\p_{x_j}\psi^{N,\sigma}(L_1,x_2))^2 dx_2\leq C \iint_{\Omega} G_0^2(x_1,x_2) dx_1 dx_2.
\ee
This estimate implies the uniqueness of the solution to the Problem \eqref{qappro1}. For the system of $N$ third-order equations endowed with $3N$ boundary conditions, the
uniqueness ensures the existence of the solution to \eqref{qappro1}. Indeed, set
\be\no
Y_{3m-2}(x_1)=A_m(x_1),\ \ Y_{3m-1}(x_1)= A_m'(x_1),\ \ Y_{3m}(x_1)= A_m''(x_1),\ \ m=1,2,\cdots, N,
\ee
where the superscripts $N$ and $\sigma$ are dropped to simplify the notations. Then the $3N$-dimensional vector functions ${\bf Y}(x_1)=(Y_1,Y_2,\cdots, Y_{3N})^t$ satisfy the following ODE system
\be\label{ode11}
\begin{cases}
Y_{3m-2}'(x_1)-Y_{3m-1}(x_1)=0,\ \ \ 1\leq m\leq N,\\
Y_{3m-1}'(x_1)-Y_{3m}(x_1)=0,\ \ \ 1\leq m\leq N,\\
\sigma Y_{3m}'(x_1)+\sum_{j=1}^{N}a_{jm}  Y_{3j}+ b_{jm} Y_{3j-1}+ c_{jm} Y_{3j-2}= G_{0m}(x_1),\ \ \ 1\leq m\leq N.
\end{cases}\ee

Let $\overline{{\bf Y}}$ be the unique solution to \eqref{ode11} with the initial data ${\bf Y}(L_0)=0$ and ${\bf Y}^{(j)}(x_1) (1\leq j\leq 3N)$ be the unique solution to
\be\label{ode12}
\begin{cases}
Y_{3m-2}'(x_1)-Y_{3m-1}(x_1)=0,\ \ \ 1\leq m\leq N,\\
Y_{3m-1}'(x_1)-Y_{3m}(x_1)=0,\ \ \ 1\leq m\leq N,\\
\sigma Y_{3m}'(x_1)+\displaystyle\sum_{j=1}^{N}a_{jm}  Y_{3j}+ b_{jm} Y_{3j-1}+ c_{jm} Y_{3j-2}= 0,\ \ \ 1\leq m\leq N,\\
{\bf Y}(L_0)={\bf e}^{(j)}=(0,0,\cdots,0,1,0,\cdots,0)^t,
\end{cases}\ee
where ${\bf e}^{(j)}$ represents the $j^{th}$ coordinate vector in the $3N$ dimensional Euclidean space.

Thus any smooth solution to \eqref{ode11} can be represented as
\be\no
{\bf Y}(x_1)= \overline{{\bf Y}}(x_1) + \sum_{j=1}^{3N} \mu_j {\bf Y}^{(j)}(x_1),
\ee
where $\mu_1,\cdots, \mu_{3N}$ are arbitrary real numbers. To find a solution to the Problem \eqref{qappro1}, it suffices to determine $\mu_1,\cdots, \mu_{3N}$ so that \be\label{alge1}\begin{cases}
\displaystyle\sum_{j=1}^{3N} \mu_j Y_{3m}^{(j)}(L_0)=0,\ \ \ 1\leq m\leq N,\\
\displaystyle\sum_{j=1}^{3N} \mu_j Y_{3m}^{(j)}(L_1)=-\overline{Y}_{3k}(L_1),\ \ \ 1\leq m\leq N,\\
\displaystyle\sum_{j=1}^{3N} \mu_j Y_{3m-2}^{(j)}(L_0)=0,\ \ \ 1\leq m\leq N.
\end{cases}\ee
The existence of the solution to the linear algebraic equations \eqref{alge1} will follow from the uniqueness of the corresponding homogenous linear algebraic equations in \eqref{alge1}. Suppose that $(\mu_1,\cdots, \mu_{3N})^t$ is a solution to the corresponding homogenous linear algebraic equations in \eqref{alge1}. Then $\sum_{j=1}^{3N} \mu_j {\bf Y}^{(j)}(x_1)$ solves the ODE system in \eqref{ode12} with boundary conditions
\be\no
Y_{3m}(L_0)= Y_{3m}(L_1)=Y_{3m-2}(L_0)=0,\ \ \forall 1\leq m\leq N.
\ee
Then it follows from the energy estimate \eqref{qappro2} to \eqref{qappro1} that $\mu_1=\mu_2=\cdots=\mu_{3N}=0$. This yields the existence and uniqueness of the solution to \eqref{qappro1}. Since the coefficients of the \eqref{qappro1} are $C^3$ smooth, the solutions $A_m^{N,\sigma}$ are $C^4$
smooth on $[L_0,L_1]$. Thus the existence of the approximate solution to the system \eqref{qap} is established.

Furthermore, with $\eta_1$ and $\eta_2$ given in Lemma \ref{qH1-appro}, multiplying the $m^{th}$ equation in \eqref{qappro1} by $\eta_1^2(x_1)\frac{d^2}{dx_1^2} A_m^{N,\sigma}$, summing from $1$ to $N$, and integrating over $[L_0,L_1]$, one can argue as for \eqref{qaH21} to get
\be\label{qH2a21}
\int_{L_0}^{\frac{L_0}{2}}\int_{-1}^1 |\p_{x_1}^2 \psi^{N,\sigma}|^2 +|\p_{x_1x_2}^2 \psi^{N,\sigma}|^2 dx_1 d x_2\leq \iint_{\Omega} G_0^2 dx_1 dx_2.
\ee
Set $w_1^{N,\sigma}=\p_{x_1}\psi^{N,\sigma}=\sum_{m=1}^N w_{1,m}^{N,\sigma}(x_1)b_m(x_2)$, where $w_{1,m}^{N,\sigma}(x_1)=\frac{d}{dx_1} A_m^{N,\sigma}$. Taking
$\frac{d}{dx_1}$ on each equation in \eqref{qappro1}, then multiplying it by $\eta_2^2(x_1)\frac{d}{dx_1} w_{1,m}^{N,\sigma}$, summing from $1$ to $N$ and integrating over
$[L_0,L_1]$, after some computations, one could obtain the estimate \eqref{qaH22} for $\psi^{N,\sigma}$ with a uniform constant $C_*$ for $N,\sigma$. Finally, to get the estimate $\p_{x_2}^2 \psi^{N,\sigma}$, one can multiply the $m^{th}$ equation in \eqref{qappro1} by $d(x_1)\lambda_m\frac{d}{dx_1} A_m^{N,\sigma}$, sum from $1$ to $N$, and integrate over $[L_0,L_1]$, where $\lambda_m$ is the eigenvalue associated with $b_m(x_2)$. Using $-b_m''(x_2)=\lambda_m b_m(x_2)$ and then integration by parts yield the estimate \eqref{qaH23} for $\psi^{N,\sigma}$ with a uniform constant $C_*$ for $N$ and $\sigma$. In summary, the $H^2$ estimate \eqref{qaH25} for $\psi^{N,\sigma}$ follows.

\end{proof}

\begin{lemma}\label{exist2}
{\it There exists a unique $H^2$ strong solution $\psi(x_1,x_2)$ to \eqref{qlinearized2} with the estimate
\be\label{h2}
\iint_{\Omega} |\psi|^2+ |\nabla\psi|^2+ |\nabla^2\psi|^2 dx_1 dx_2\leq C_*\iint_{\Omega} |G_0|^2 + |\nabla G_0|^2 dx_1 dx_2,
\ee
where the constant $C_*$ depends only on the $H^3(\Omega)$ norms of the coefficients $k_{11}, k_{12}$ and $k_1$.
}\end{lemma}

\begin{proof}
Consider the sequence of approximate solutions $\psi^{N,\sigma}$ as $N\to\infty$. Thanks to \eqref{qappro2}, $\|\psi^{N,\sigma}\|_{H^2(\Omega)}$ is uniformly
bounded in $N$. Therefore, due to the weak compactness of a bounded set in a Hilbert space, there exists a subsequence, denoted  by $\psi^{N,\sigma}$ for simplicity, which converges strongly in $H^1(\Omega)$ and weakly in $H^2(\Omega)$ to a limit $\psi^{\sigma}\in H^2(\Omega)$. Furthermore, $\psi^{\sigma}$ satisfies the following uniform estimate:
\be\label{qconv1}
\|\psi^{\sigma}\|_{H^2(\Omega)}\leq C_* \|G_0\|_{H^1(\Omega)}.
\ee

Due to the strong convergence $\psi^{N,\sigma}\to \psi^{\sigma}$ as $N\to \infty$ in $H^1$, $\psi^{\sigma}$ retains the boundary conditions
\be\label{sigmabcs}\begin{cases}
\psi^{\sigma}(L_0,x_2)=0, \forall x_2\in (-1,1),\\
\p_{x_2}\psi^{\sigma}(x_1,\pm 1)=0,\forall x_1\in (L_0,L_1).
\end{cases}\ee
Now we show that $\psi^{\sigma}$ is a weak solution to the system \eqref{qap}. Given any test function $\chi(x_1,x_2)= \sum_{m=1}^{N_0} \chi_m(x_1) b_m(x_2)$, where $\chi_m(x_1)\in
C_c^{\infty}((L_0,L_1))$. Let $N\geq N_0$. Multiplying each of equations in \eqref{qappro1} by $\xi_m$ ($\xi_m\equiv 0$ for any $N_0+1\leq j\leq N$), then sum up from
$m=1$ to $m=N$, and integrate with respect to $x_1$ from $L_0$ to $L_1$, one gets that
\be\label{qwf1}
\iint_{\Omega}(\sigma \p_{x_1}^3 \psi^{N,\sigma}+\sum_{i,j=1}^2 k_{ij}\p_{x_i x_j}^2 \psi^{N,\sigma}+ \sum_{i=1}^2 k_i \p_{x_i}\psi^{N,\sigma}) \xi dx_1 dx_2=0.
\ee
Integrating by parts and passing to the limit for the above weak convergent subsequence of $\psi^{N,\sigma}$ yield
\be\label{qwf2}
&&\iint_{\Omega}-\sigma \p_{x_1}^2\psi^{\sigma}\p_{x_1}\xi- \p_{x_1}\psi^{\sigma}\p_{x_1}(k_{11}\xi)\\\no
&&\quad\quad- 2\p_{x_2} \psi^{\sigma}\p_{x_1}(k_{12}\xi)- \p_{x_2}\psi^{\sigma} \p_{x_2} \xi+ \xi\sum_{i=1}^2 k_j \p_{x_j}\psi^{\sigma} dx_1dx_2=0.
\ee

By a density argument, the weak formulation \eqref{qwf2} holds for any test function $\xi\in H^1(\Omega)$ vanishing at $x_1=L_0$ and $x_1=L_1$. Next we consider a sequence of approximate solutions $\{\psi^{\sigma}\}$ as $\sigma\to 0$. Thanks to \eqref{qappro2}, $\|\psi^{\sigma}\|_{H^2(\Omega)}$ is
uniformly bounded independent of $\sigma$. This further implies the existence of a weakly convergent subsequence labeled as $\{\psi^{\sigma_j}\}_{j=1}^{\infty}$ with $\sigma_j\to 0$ as $j\to \infty$ converging weakly to a limit $\psi\in H^2(\Omega)$. Thus $\psi$ also retains the boundary condition
\be\label{psi-bc}\begin{cases}
\psi(L_0,x_2)=0, \forall x_2\in (-1,1),\\
\p_{x_2}\psi(x_1,\pm 1)=0,\forall x_1\in (L_0,L_1).
\end{cases}\ee

It follows from \eqref{qwf2} that
\be\label{qlw}
\iint_{\Omega}- \p_{x_1}\psi\p_{x_1}(k_{11}\xi)- 2\p_{x_2} \psi \p_{x_1}(k_{12}\xi)- \p_{x_2}\psi \p_{x_2} \xi+ \xi\sum_{i=1}^2 k_j \p_{x_j}\psi dx_1dx_2=0,
\ee
holds for any $\xi\in H^1(\Omega)$ vanishing at $x_1=L_0$ and $x_1=L_1$. Since $\psi\in H^2(\Omega)$, then $\psi$ indeed is a strong solution to \eqref{qlinearized2} and the equation in \eqref{qlinearized2} holds almost everywhere.

\end{proof}

\subsection{The $H^4$ estimate}\label{se33}\noindent

Since the equation in \eqref{qlinearized2} is elliptic in $\Omega_{\frac{1}{8}}:= (L_0,\frac{L_0}{8})\times (-1,1)$, it is easy to derive the $H^4$ estimate of $\psi$ on $\Omega_{\frac{1}{8}}$.

\begin{lemma}({\bf $H^4$ estimate on subsonic region.})\label{qH3}
{\it There exists a constant $\delta_*>0$ depending only on the background flow, such that if $0<\delta_0\leq \delta_*$ in \eqref{qcoe}, the solution to
\eqref{qlinearized2} satisfies the following basic energy estimate
\be\label{qH31}
\int_{L_0}^{\frac{1}{4}L_0} \int_{-1}^1 |\nabla^3\psi|^2+|\nabla^4\psi|^2 dx_1 dx_2 \leq C_*\|G_0\|_{H^2(\Omega)}^2,
\ee
with a constant $C_*$ depending only on the $H^3(\Omega)$ norms of the coefficients $k_{11}, k_{12}$ and $k_1$.
}\end{lemma}

\begin{proof}

According to \eqref{elliptic}, the equation in \eqref{qlinearized2} is elliptic in $\Omega_{\frac{1}{8}}$. Since the coefficients $k_{ij}, k_1\in C^3(\overline{\Omega})$, by the elliptic theory in \cite{gt}, $\psi\in C^{4,\alpha}(\overline{\Omega_{\frac{1}{8}}}\setminus\{(L_0, \pm 1)\})$. Differentiating the equation in \eqref{qlinearized2} with respect to $x_2$ in $\overline{\Omega_{\frac{1}{8}}}$ and evaluating at $(x_1,\pm 1)$ yield
\be\label{c1}
\p_{x_2}^3 \psi(x_1,\pm 1)=0,\ \ \ \forall x_1\in (L_0,\frac{L_0}{8}).
\ee
We now prove the $H^4$ estimate of $\psi$ in $\Omega_{\frac{1}{4}}:= (L_0,\frac{L_0}{4})\times (-1,1)$.

To improve the regularity near the corner point $(L_0,\pm 1)$, one can use the standard symmetric extension technique. Indeed, extend $k_{12}$ an $G_0$ from $(L_0,\frac{L_0}{8})\times (-1,1)$ to
$D_1:=(L_0,\frac{L_0}{8})\times (-3,3)$ as
\be\no
(K_{12}, \mathcal{G}_0)(x_1,x_2)=\begin{cases}
-(k_{12}, G_0)(x_1,2-x_2), \ \ &\forall x_2\in (1,3),\\
(k_{12}, G_0)(x_1,x_2),\ \ \ & \forall x_2\in [-1,1],\\
-(k_{12}, G_0)(x_1,-2-x_2), \ \ &\forall x_2\in (-3,-1),
\end{cases}\ee
while extend $\psi, k_{11}$ and $k_1$ as
\be\no
(\psi_e,K_{11},K_1)(x_1,x_2)=\begin{cases}
(\psi,k_{11},k_1)(x_1,2-x_2), \ \ &\forall x_2\in (1,3),\\
(\psi,k_{11},k_1)(x_1,x_2),\ \ \ & \forall x_2\in [-1,1],\\
(\psi,k_{11},k_1)(x_1,-2-x_2), \ \ &\forall x_2\in (-3,-1).
\end{cases}\ee
Since $\p_{x_2}\psi(x_1,\pm 1)=\p_{x_2}k_1(x_1,\pm 1)=\p_{x_2}^2 k_{12}(x_1,\pm 1)=\p_{x_2}^3\psi(x_1,\pm 1)=0$ for any $x_1\in [L_0,\frac{1}{8}L_0]$, then $\psi_e\in C^{4,\alpha}(D_1)$, $K_{11}, K_{12}, K_1\in C^{2,\alpha}(D_1)$ and also
\be\no
\|\psi_e\|_{H^2(D_1)}\leq C_*\|G_0\|_{H^1(\Omega)}, \ \ \|K_{ij}\|_{H^3(D_1)}+\|K_1\|_{H^3(D_1)}\leq C_*(\|k_{ij}\|_{H^3(\Omega)}+\|k_1\|_{H^3(\Omega)}).
\ee
Then $\psi_e$ satisfies
\be\label{p1}\begin{cases}
\displaystyle\sum_{i,j=1}^2 K_{ij} \p_{x_ix_j}^2 \psi_e + K_1 \p_{x_1} \psi_e =\mathcal{G}_0(x_1,x_2),\ \forall (x_1,x_2)\in
D_1,\\
\psi_e(L_0,x_2)=0,\ \ \ \forall x_2\in (-3,3),\\
\p_{x_2}\psi_e(x_1, \pm 1)=0,\ \ \forall x_1\in [L_0,\frac18 L_0],
\end{cases}
\ee

Set $V_1=\p_{x_2} \psi_e$. Since $K_{11}\geq \kappa_*>0$ on subsonic region $[L_0,\frac{L_0}{8}]\times [-1,1]$, one has
\be\label{p2}
\p_{x_1}^2 \psi_e=\frac{1}{K_{11}} (\mathcal{G}_0- 2 K_{12} \p_{x_1} V_1- \p_{x_2} V_1- K_1 \p_{x_1} \psi_e).
\ee
It follows from \eqref{p1} that $V_1$ solves
\be\label{p3}\begin{cases}
\displaystyle\sum_{i,j=1}^2 K_{ij} \p_{x_ix_j}^2 V_1 + K_3 \p_{x_1} V_1 + K_4 \p_{x_2} V_1 =\mathcal{G}_1(x_1,x_2),\ \forall (x_1,x_2)\in
D_1,\\
V_1(L_0,x_2)=0,\ \ \ \forall x_2\in (-3,3),\\
V_1(x_1, \pm 3)=0,\ \ \forall x_1\in (L_0,\frac{1}{8}L_0),
\end{cases}
\ee
where
\be\no\begin{cases}
K_3=K_1 +2 \p_{x_2} K_{12}-\frac{2 \p_{x_2} K_{11} K_{12}}{K_{11}},\quad \quad K_4=-\frac{\p_{x_2} K_{11}}{K_{11}},\\
\mathcal{G}_1=\p_{x_2} \mathcal{G}_0-  \p_{x_2} K_1\p_{x_1}\psi_e-\frac{\p_{x_2} K_{11}}{K_{11}}(\mathcal{G}_0-K_1 \p_{x_1} \psi_e).
\end{cases}\ee

By Theorems 8.8, 8.9 and 8.12 in \cite{gt}, one derives that
\be\label{p5}
&&\|V_1\|_{H^2(D_2)}\leq C_*(\|V_1\|_{L^2(D_1)}^2+ \|\mathcal{G}_1\|_{L^2(D_1)}^2)\\\no
&&\leq C_*\left(\|G_0\|_{L^2(\Omega)}+\|\mathcal{G}_0\|_{H^1(D_1)}+(\|\p_{x_2}K_1\|_{L^{\infty}(D_1)}+\|\frac{ K_{1}\p_{x_2}K_{11}}{K_{11}}\|_{L^{\infty}(D_1)})\|\nabla\psi_e\|_{L^2(D_1)}\right)\\\no
&&\leq C_*(\|G_0\|_{H^1(\Omega)}+ \|\nabla \psi\|_{L^2(\Omega)})\leq C_*\|G_0\|_{H^1(\Omega)},
\ee
where $D_2:=(L_0,\frac{3L_0}{16})\times (-2,2)$. It is noted that according to Theorems 8.8, 8.9 and 8.12 in \cite{gt}, the constant $C_*$ depends only on the ellipticity constant, $\|K_{11}, K_{12}\|_{C^{0,1}(D_1)}$ and $\|K_3,K_4\|_{L^{\infty}(D_1)}$, which can be bounded by a constant $C_*$ depending only on the $H^3(D_2)$ norm of $k_{11}, k_{12}$ and $ k_1$.

This, together with \eqref{p2}, implies that $\p_{x_1}^3 \psi_e$ also admits the same estimate. Therefore we conclude that
\be\label{p6}
\|\psi\|_{H^3(\Omega_{\frac{1}{4}})}\leq\|\psi_e\|_{H^3(D_2)}\leq C_*\|G_0\|_{H^1(\Omega)}.
\ee

On the domain $D_2$, there holds
\be\label{p7}
\p_{x_1}^2 V_1=\frac{1}{K_{11}}(\mathcal{G}_1-2 K_{12}\p_{x_1} V_2-\p_{x_2} V_2- K_3 \p_{x_1} V_1-K_4 \p_{x_2} V_1).
\ee
Denote $V_2=\p_{x_2} V_1$. Then $V_2$ solves
\be\label{p21}\begin{cases}
\displaystyle\sum_{i,j=1}^2 K_{ij} \p_{ij}^2 V_2 + K_5 \p_1 V_2 + K_6 \p_2 V_2 =\mathcal{G}_2(x_1,x_2),\ \forall (x_1,x_2)\in
D_1,\\
V_2(L_0,x_2)=0,\ \ \ \forall x_2\in [-3,3],
\end{cases}
\ee
where
\be\no\begin{cases}
K_5=K_1 +4 \p_{x_2} K_{12}-\frac{4 \p_{x_2} K_{11} K_{12}}{K_{11}},\quad \quad K_4=-\frac{2\p_{x_2} K_{11}}{K_{11}},\\
\mathcal{G}_2=\p_{x_2} \mathcal{G}_1-  \p_{x_2} K_3\p_{x_1}V_1-\p_{x_2} K_4\p_{x_2} V_1-\frac{\p_{x_2} K_{11}}{K_{11}}(\mathcal{G}_1-K_3 \p_{x_1} V_1-K_4 \p_{x_2} V_1).
\end{cases}\ee
It follows from Theorems 8.8, 8.9 and 8.12 in \cite{gt} that on $\Omega_{\frac14}=(L_0,\frac{1}{4}L_0)\times (-1,1)$
\be\label{p22}
\|V_2\|_{H^2(\Omega_{\frac14})}\leq C_*(\|V_2\|_{L^2(D_2)}+\|\mathcal{G}_2\|_{L^2(D_2)})\leq C_*(\|G_0\|_{H^1(\Omega)}+\|\mathcal{G}_2\|_{L^2(D_2)}).
\ee
The term $\|\mathcal{G}_2\|_{L^2(D_2)}$ is estimated as:
\be\no
&&\|\mathcal{G}_2\|_{L^2(D_2)}\leq \|\mathcal{G}_1\|_{H^1(D_2)}+ (\|\p_{x_2} K_3\|_{L^4(D_2)}+\|\p_{x_2} K_4\|_{L^4(D_2)})\|\nabla V_1\|_{L^4(D_2)}\\\no
&&\quad\quad+\|\frac{\p_{x_2} K_{11}}{K_{11}}\|_{L^{\infty}(D_2)}(\|\mathcal{G}_1\|_{L^2(D_2)}+(\|K_3\|_{L^{\infty}(D_2)}+\|K_4\|_{L^{\infty}(D_2)})\|\nabla V_1\|_{L^2(D_2)})\\\no
&&\leq C_*(\|\mathcal{G}_0\|_{H^2(D_2)}+\|\p_{x_2}K_1\|_{H^2(D_2)}\|\p_{x_1}\psi_e\|_{H^1(D_2)}\\\no
&&\quad\quad+\|\frac{\p_{x_2}K_{11}}{K_{11}}\|_{H^1(D_2)}(\|G_0\|_{H^2}+\|K_1\|_{H^2(D_2)}\|\p_{x_1}\psi_e\|_{H^1(D_2)})\\\no
&&\quad\quad +C_*(\|K_3\|_{H^1(D_2)}+\|K_4\|_{H^1(D_2)})\|V_1\|_{H^2(D_2)}+ C_*\|\nabla V_1\|_{L^2(D_2)}\\\no
&&\leq C_*\|\mathcal{G}_0\|_{H^2(D_2)}\leq C_*\|G_0\|_{H^2(\Omega)},
\ee
where one has used the inequality $\|fg\|_{H^1(D)}\leq \|f\|_{H^1(D)}\|g\|_{H^2(D)}$ for a two dimensional bounded domain $D$. Employing the equation \eqref{p7}, one finally derives the estimate \eqref{qH31}.

\end{proof}

To improve the regularity of $\psi$ on the domain $(\frac{1}{4}L_0, L_1)\times (-1,1)$, one can follow basically the idea introduced by Kuzmin \cite{Kuzmin2002} and extend the problem to an auxiliary problem in a larger domain where the equation in \eqref{qlinearized2} becomes elliptic near the exit of the nozzle. To this end, one can first extend the background solution to $[L_0,L_2]$ where $L_2=2L_1$ by simply extending the function $a(x_1)$ to $[L_0,L_2]$ so that $a(x_1)$ is a $C^5$ differentiable function on $[L_0,L_2]$ and $a'(x_1)$ is positive on $(0,L_2]$. By the theory of ordinary differential equation, $(\bar{\rho}, \bar{u})$ can be extended to $[L_0,L_2]$ so that the functions $\bar{k}_1$ and $\bar{k}_{11}$ defined in \eqref{q19} satisfy also the properties in \eqref{apos1}-\eqref{apos2} on $[L_0,L_2]$ if $d_0$ is chosen to be large enough.

Let $\ell=\frac{L_1}{20}$ and define two non-increasing cut-off functions on $[L_0,L_2]$ as follows
\be\no
\xi_1(x_1)=\begin{cases}
1,\ \ &L_0\leq x_1\leq L_1 +2\ell,\\
0,\ \ &L_1+4\ell\leq x_1\leq L_2,
\end{cases},\ \ \xi_2(x_1)=\begin{cases}
1,\ \ &L_0\leq x_1\leq L_1 +\ell,\\
0,\ \ &L_1+2\ell\leq x_1\leq L_2.
\end{cases}
\ee
Define
\be\no
&&\bar{a}_{11}(x_1)= \bar{k}_{11}(x_1)\xi_1(x_1)+ (1-\xi_1(x_1)),\\
&&\bar{a}_1(x_1)= \bar{k}_1(x_1)\xi_2(x_1) - k_0 (1-\xi_2(x_1)),
\ee
where $k_0$ is a positive number to be specified later. Then
\be\no
\bar{a}_{11}(x_1)=\begin{cases}
\bar{k}_{11}(x_1), \ \  x_1\in [L_0,L_1+2\ell],\\
1,\ \ \ x_1\in [L_1+4\ell, L_2],
\end{cases} \bar{a}_{1}(x_1)=\begin{cases}
\bar{k}_{1}(x_1), \ \  x_1\in [L_0,L_1+\ell],\\
-k_0,\ \ \ x_1\in [L_1+2\ell, L_2],
\end{cases}
\ee
and for $j=0,1,2,3$
\be\no
&&2\bar{a}_1 + (2j-1) \bar{a}_{11}'= 2\bar{k}_1\xi_2+ (2j-1)\bar{k}_{11}'\xi_1+ (2j-1)(\bar{k}_{11}-1) \xi_1' -2 k_0(1-\xi_2)\\\no
&&=\begin{cases}
2\bar{k}_1 + (2j-1) \bar{k}_{11}'\leq -\kappa_*<0,\ \ \ \  & L_0\leq x_1\leq L_1+\ell,\\
2\bar{k}_1 + (2j-1) \bar{k}_{11}'+2(\bar{k}_1-k_0)(1-\xi_2),\ \ \ \ &L_1+\ell< x\leq L_1+ 2\ell,\\
(2j-1)\bar{k}_{11}'\xi_1+ (2j-1) (\bar{k}_{11}-1)\xi_1'- 2k_0,\ \ \ &L_1+2\ell<x_1\leq L_1 +4\ell,\\
-2k_0,\ \ \ & L_1+4\ell<x_1\leq L_2.
\end{cases}
\ee
Thus for sufficiently large $k_0$ and $d_0>0$, there holds for any $x_1\in [L_0, L_2]$
\be\label{apos3}
2\bar{a}_1 + (2j-1) \bar{a}_{11}'\leq -\kappa_*<0, \ \ j=0,1,2,3,4,\\\label{apos4}
(\bar{a}_1+j \bar{a}_{11}') d -\frac{1}{2}(\bar{a}_{11} d)'\geq 3,\ \ \ j=0,1,2,3,
\ee
where $d(x_1)= 6(x_1-d_0)<0$ for any $x_1\in [L_0,L_2]$.

Furthermore, we define an extension operator $\mathcal{E}$ which extend any functions $f(x_1,x_2)$ on $\Omega$ to be defined on $\Omega_2=(L_0,L_2)\times (-1,1)$ as
\be\no
\mathcal{E}(f)(x_1,x_2)=\begin{cases}
f(x_1,x_2),\ \ &(x_1,x_2)\in\Omega,\\
\displaystyle\sum_{j=1}^4c_j f(L_1+\frac{1}{j}(L_1-x_1),x_2),\ \ &(x_1,x_2)\in (L_1,L_2)\times (-1,1),
\end{cases}
\ee
where the constants $c_j$ are uniquely determined by the following algebraic equations
\be\no
\sum_{j=1}^4 \left(-\frac{1}{j}\right)^k c_j=1,\ \ k=0,1,2,3.
\ee
The extension operator $\mathcal{E}$ is a bounded operator from $H^j(\Omega)$ to $H^j(\Omega_2)$ for any $j=1,2,3,4$. Then one can define the extension of the operator $\mathcal{L}$ in \eqref{qlinearized2} to the domain $\Omega_2$ as follows
\be\no
&&a_{11}=\bar{a}_{11}+ \mathcal{E}(k_{11}-\bar{k}_{11}),\ \ \ \ a_{22}\equiv 1,\\\no
&&a_{12}=a_{21}=\mathcal{E}(k_{12}), a_1=\bar{a}_1+ \mathcal{E}(k_1-\bar{k}_1),\ \ F_0=\mathcal{E} G_0.
\ee
Then the following estimates hold
\begin{eqnarray}\label{ae}
\begin{cases}
\|a_{11}-\bar{a}_{11}\|_{H^3(\Omega_2)}+\|a_{12}\|_{H^3(\Omega_2)}\leq C_*\delta_0,\ \\
\|a_{1}-\bar{a}_1\|_{H^3(\Omega_2)} \leq C_*\delta_0,\ \ \|F_0\|_{H^3(\Omega_2)}\leq C_*\|G_0\|_{H^3(\Omega)}\leq C_*(\epsilon+\delta_0^2),\ \ \\
a_{12}(x_1,\pm 1)=\p_{x_2}^2a_{12}(x_1,\pm 1)=0,\ \ \ \forall x_1\in [L_0,L_1],\\
\p_{x_2} a_{11}(x_1,\pm 1)=\p_{x_2} a_1(x_1,\pm 1)=\p_{x_2} F_0(x_1,\pm 1)=0.
\end{cases}
\end{eqnarray}

Consider the following auxiliary problem in domain $\Omega_2$:
\be\label{au1}
\begin{cases}
\mathcal{M}\Psi=\displaystyle\sum_{i,j=1}^2 a_{ij} \p_{x_ix_j}^2\Psi + a_1\p_{x_1}\Psi= F_0,\ \ & \text{on } \Omega_2,\\
\Psi(L_0,x_2)=0,\ \  &\text{on } (-1,1),\\
\p_{x_2}\Psi(x_1,\pm 1)=0,\ \ &\text{on }  (L_0,L_2),\\
\p_{x_1}\Psi(L_2,x_2)=0,\ \ \ &\text{on }  (-1,1).
\end{cases}\ee
We would like to prove the existence and uniqueness of $H^2$ strong solution $\Psi$ to \eqref{au1} and derive the higher order estimates for $\nabla\p_{x_1}^2\Psi$ and $\nabla\p_{x_1}^3\Psi$ on the subregion $(\frac{1}{4}L_0,L_1+12\ell)\times (-1,1)$. Furthermore, $\psi=\Psi$ on $\Omega$, which thus leads to the estimates for $\nabla\p_{x_1}^2\psi$ and $\nabla\p_{x_1}^3\psi$ on $\Omega$.

To find a solution to \eqref{au1}, one can still resort to the singular perturbation problem
\be\label{au2}
\begin{cases}
\mathcal{M}^{\sigma}\Psi^{\sigma}=\sigma\p_{x_1}^3\Psi^{\sigma}+\displaystyle\sum_{i,j=1}^2 a_{ij} \p_{x_ix_j}^2\Psi^{\sigma} + a_1\p_{x_1}\Psi^{\sigma}= F_0,\ \ & \text{on } \Omega_2, \ \ \sigma>0,\\
\Psi^{\sigma}(L_0,x_2)=\p_{x_1}^2\Psi^{\sigma}(L_0,x_2)=0,\ \  &\text{on }\ \  (-1,1),\\
\p_{x_2}\Psi^{\sigma}(x_1,\pm 1)=0,\ \ &\text{on }\ \  (L_0,L_2),\\
\p_{x_1}\Psi^{\sigma}(L_2,x_2)=0,\ \ \ &\text{on }\ \  (-1,1).
\end{cases}\ee

One could prove the following $H^2$ estimate for the solution $\Psi^{\sigma}$ to \eqref{au2}.

\begin{lemma}\label{aH2}
{\it There exists a constant $\delta_*>0$ depending only on the background flow, such that if $0<\delta_0\leq \delta_*$ in \eqref{ae}, the classical solution $\Psi^{\sigma}$ to
\eqref{au2} satisfies the following energy estimate
\be\label{aH1}
&&\iint_{\Omega_2} \left(\sigma |\p_{x_1}^2\Psi^{\sigma}|^2 +|\Psi^{\sigma}|^2+|\nabla\Psi^{\sigma}|^2\right) dx\\\no
&&\quad\quad\quad+\int_{-1}^1(\p_{x_1}\Psi^{\si}(L_0,x_2))^2+(\p_{x_2}\Psi^{\si}(L_2,x_2))^2 dx_2  \leq C_* \iint_{\Omega_2} F_0^2 dx,\\\label{aH21}
&&\sigma\int_{\frac{3}{4}L_0}^{L_1+16\ell}\int_{-1}^1|\p_{x_1}^3\Psi^{\sigma}|^2 dx_2 dx_1+\iint_{\Omega_2}|\nabla^2\Psi^{\sigma}(x)|^2 dx \leq C_* \iint_{\Omega_2} F_0^2+|\nabla F_0|^2 dx,
\ee
where the constant $C_*$ depends only on the $H^3(\Omega_2)$ norms of the coefficients $a_{11}, a_{12}$ and $a_1$.
}\end{lemma}

\begin{proof}
The proof is quite similar to that of Lemma \ref{qH1-appro}. We omit the superscript $\sigma$ to simplify the notations. Due to the boundary conditions in \eqref{au2}, the boundary integral term $\int_{-1}^1 \sigma d \p_{x_1}\Psi \p_{x_1}^2\Psi\b|_{L_0}^{L_2}dx_2$ vanishes. Since \eqref{apos3}-\eqref{apos4} hold, one can argue as in Lemma \ref{qH1-appro} to derive the estimate \eqref{aH1}. Similarly, choose a monotonic decreasing cut-off function $\eta_1\in C^{\infty}([L_0,L_2])$ such that $0\leq \eta_1(x_1)\leq 1$ for all $x_1\in [L_0,L_1]$ and
\be\no
\eta_1(x_1)=\begin{cases}
1,\ \ \ & L_0\leq x_1\leq \frac{L_0}{2},\\
0,\ \ \ & \frac{L_0}{4}\leq x_1\leq L_2.
\end{cases}\ee

Multiplying the equation \eqref{au2} by $\eta_1^2 \p_{x_1}^2 \Psi$ and integrating by parts give
\be\no
&&\iint_{\Omega_2} \eta_1^2 (a_{11}(\p_{x_1}^2 \Psi)^2 +(\p_{x_1x_2}^2 \Psi)^2)- \sigma \eta_1 \eta_1' (\p_{x_1}^2 \Psi)^2 dx =-\iint_{\Omega_2} 2 \eta_1^2
a_{12} \p_{x_1x_2}^2 \Psi\p_{x_1}^2 \Psi dx\\\no
&& -\iint_{\Omega_2} 2 \eta_1 \eta_1'  \p_{x_1x_2}^2\Psi \p_{x_2}\Psi dx +\iint_{\Omega_2}\eta_1^2(F_0 - a_1\p_{x_1} \Psi) \p_{x_1}^2 \Psi dx.
\ee
Then as in Lemma \ref{qH1-appro}, one may conclude that
\be\label{aH22}
\int_{L_0}^{\frac{1}{2}L_0} \int_{-1}^1 |\nabla\p_{x_1}\Psi|^2 dx_2 dx_1\leq C_*\|F_0\|_{L^2(\Omega_2)}^2.
\ee

Choose a monotonic increasing cut-off function $\eta_3\in C^{\infty}([L_0,L_2])$ such that $0\leq \eta_3\leq 1$ for all $x_1\in [L_0,L_2]$ and
\be\no
\eta_3(x_1)=\begin{cases}
0,\ \ \ & L_0\leq x_1\leq L_1+2\ell,\\
1,\ \ \ & L_1+4\ell\leq x_1\leq L_2.
\end{cases}\ee

Multiplying the equation \eqref{au2} by $\eta_3^2 \p_{x_1}^2 \Psi$ and integrating by parts yield that
\be\no
&&\frac{\sigma}{2}\int_{-1}^1 (\p_{x_1}^2\Psi(L_2,x_2))^2 dx_2+\iint_{\Omega_2} \eta_3^2 (a_{11}(\p_{x_1}^2 \Psi)^2 +(\p_{x_1x_2}^2 \Psi)^2)dx =\sigma \iint_{\Omega_2}\eta_3 \eta_3' (\p_{x_1}^2 \Psi)^2 dx\\\no
&& -\iint_{\Omega_2} 2 \eta_3^2
a_{12} \p_{x_1x_2}^2 \Psi\p_{x_1}^2 \Psi +\eta_3 \eta_3'  \p_{x_1x_2}^2\Psi \p_{x_2}\Psi dx +\iint_{\Omega_2}\eta_3^2(F_0- a_1\p_{x_1} \Psi) \p_{x_1}^2 \Psi dx.
\ee
Using \eqref{aH1} to control the term involving $\sigma$, one gets from the above identity that
\be\label{aH23}
\int_{L_1+4\ell}^{L_2}\int_{-1}^1 |\nabla\p_{x_1}\Psi|^2 dx_1 dx_2 \leq C_*\iint_{\Omega_2} F_0^2 dx_1 dx_2.
\ee

Set $W_1=\p_{x_1}\Psi$. Then $W_1$ solves
\be\label{aw1}\begin{cases}
\displaystyle\sigma \p_{x_1}^3 W_1 + \sum_{i,j=1}^2 a_{ij}\p_{x_i x_j}^2 W_1 + a_3 \p_{x_1} W_1 + a_4 \p_{x_2} W_1=F_1,\ \ \text{on }\ \ \Omega_2,\\
\p_{x_2} W_1(x_1,\pm 1)=0,\ \ \ \text{on }\ \  [L_0,L_2],\\
\p_{x_1} W_1(L_0,x_2)= W_1(L_2,x_2)=0,\ \text{on }\ \ \ [-1,1],
\end{cases}\ee
where
\be\no
a_3= a_1+ \p_{x_1} a_{11},\ a_4=2\p_{x_1} a_{12},\ F_1=\p_{x_1}F_0 -\p_{x_1} a_1\p_{x_1}\Psi.
\ee
Define the smooth cutoff function $0\leq \eta_4(x_1)\leq 1$ on $[L_0,L_2]$ as
\be\no
\eta_4(x_1)=\begin{cases}
0,\ \ \ & L_0\leq x_1\leq \frac{7}{8}L_0,\\
1,\ \ \ & \frac{3}{4}L_0\leq x_1\leq L_1+16\ell,\\
0,\ \ \ & L_1+18\ell\leq x_1\leq L_2.
\end{cases}
\ee

Multiplying the equation \eqref{aw1} by $\eta_4^2 d(x_1) \p_{x_1} W_1$ and integrating over $\Omega$, after some integrations by parts, one can obtain
\be\no
&&\iint_{\Omega_2}-\sigma\eta_4^2 d (\p_{x_1}^2 W_1)^2dx_1 dx_2+\iint_{\Omega_2} [(\eta_4^2d a_3 -\frac{1}{2}\p_{x_1}(\eta_4^2d a_{11})-\eta_4^2d \p_{x_2} a_{12})](\p_{x_1} W_1)^2 dx_1 dx_2\\\no
&&+\iint_{\Omega_2} [\frac{\eta_4^2}{2}d'(x_1)+\eta_4\eta_4' d ](\p_{x_2} W_1)^2+ \eta_4^2 a_4d \p_{x_1} W_1 \p_{x_2} W_1  dx_1
dx_2\\\no
&& = \iint_{\Omega_2} \eta_4^2 d \p_{x_1} W_1 F_1 dx_1 dx_2+\sigma\iint_{\Omega_2} \p_{x_1}(\eta_4^2 d)\p_{x_1}^2 W_1\p_{x_1} W_1 dx_1 dx_2.
\ee
Then this and \eqref{aH1} imply that
\be\no
&&\iint_{\Omega_2}\sigma \eta_4^2(\p_{x_1}^2 W_1)^2+ \eta_4^2 |\nabla W_1|^2 dx_1 dx_2\\\no
&&\leq C_*\iint_{\Omega_2} |\eta_4'| |\nabla W_1|^2+F_1^2 dx\leq C_*\iint_{\Omega_2} |F_0|^2 +|\nabla F_0|^2 dx_1 dx_2,
\ee
where one has used the fact that the support of $\eta_4'$ is contained in $(\frac{7}{8} L_0, \frac{3}{4}L_0)\cup (L_1+16\ell, L_1+18\ell)$ and thus $\iint_{\Omega_2}|\eta_4'||\nabla W_1|^2 dx$ can be controlled by \eqref{aH22}-\eqref{aH23}.

Collecting \eqref{aH22}-\eqref{aH23} yields that
\be\label{aH24}
\sigma\int_{\frac{3}{4}L_0}^{L_1+16\ell}\int_{-1}^1|\p_{x_1}^3\Psi|^2 dx_2 dx_1+\iint_{\Omega_2} |\nabla\p_{x_1}\Psi|^2 dx_1 dx_2 \leq \iint_{\Omega_2} |F_0|^2 +|\nabla F_0|^2 dx_1 dx_2.
\ee

It remains to estimate $\p_{x_2}^2\Psi$. Define $V_1=\p_{x_2} \Psi$. Then
\be\label{av1}\begin{cases}
\displaystyle\sigma \p_{x_1}^3 V_1 + \sum_{i,j=1}^2 a_{ij} \p_{x_i x_j}^2 V_1 + (a_1 + 2 \p_{x_2} a_{12})\p_{x_1} V_1=\p_{x_2}F_0-\p_{x_2} a_{11} \p_{x_1}^2 \Psi- \p_{x_2} a_1 \p_{x_1}\Psi,\\
V_1(L_0,x_2)=\p_{x_1}^2 V_1(L_0,x_2)=0,\ \ \ \forall x_2\in [-1, 1],\\
\p_{x_1} V_1(L_2,x_2)=0,\ \ \ \forall x_2\in [-1, 1],\\
V_1(x_1,\pm 1)=0,\ \ \forall x_1\in [L_0,L_2].
\end{cases}\ee
Multiplying the equation in \eqref{av1} by $d(x_1)\p_{x_1} V_1$ and integrating over $\Omega_2$ lead to
\be\no
&&\iint_{\Omega_2} d(x_1)\p_{x_1} V_1 (\p_{x_2}F_0-\p_{x_2} a_{11} \p_{x_1}^2 \Psi- \p_{x_2} a_1 \p_{x_1}\Psi) dx_1 dx_2\\\no
&&=\iint_{\Omega_2}-d \sigma (\p_{x_1}^2V_1)^2-6\sigma \p_{x_1} V_1\p_{x_1}^2 V_1 + \left((a_1+\p_{x_2} a_{12}) d-\frac{1}{2}\p_{x_1}(a_{11}d)\right)(\p_{x_1} V_1)^2dx\\\no
&&\quad+\iint_{\Omega_2}\frac{1}{2}d'(x_1) (\p_{x_2}V_1)^2 dx_1 dx_2+\frac{1}{2}\int_{-1}^1 \left(a_{11} d(\p_{x_1} V_1)^2 - d (\p_{x_2}
V_1)^2\right)\b|_{x_1=L_0}^{L_2} dx_2.
\ee
Recalling $a_{11}(L_0,x_2)>0$ for any $x_2\in [-1,1]$ and $d(x_1)<0$ for any $x_1\in [L_0,L_2]$, one gets from the above equality that
\be\no
&&\sigma \iint_{\Omega_2} |\p_{x_1}^2V_1|^2 dx + \iint_{\Omega_2} |\nabla V_1|^2 dx +\int_{-1}^1 (\p_{x_1}V_1(L_0,x_2))^2+\p_{x_2} V_1(L_2,x_2))^2
dx_2\\\label{aH25}
&&\leq C_*\iint_{\Omega_2} (|\p_{x_2}F_0|^2+ (\p_{x_1}^2 \Psi)^2 +|\nabla \Psi|^2) dx\leq \iint_{\Omega_2} (|F_0|^2+|\nabla F_0|^2) dx.
\ee
The proof of Lemma \ref{aH2} is completed.

\end{proof}

Then one can prove easily that
\begin{lemma}\label{equality}
{\it There exists a unique strong solution $\Psi\in H^2(\Omega_2)$ to \eqref{au1} with the estimate
\be\label{au100}
\|\Psi\|_{H^2(\Omega_2)}\leq C_*\|F_0\|_{H^1(\Omega_2)},
\ee
where $C_*$ depends only on the $H^3(\Omega_2)$ norms of the coefficients $a_{11}, a_{12}$ and $a_1$. Moreover, the solution $\Psi$ coincides with the unique strong solution $\psi\in H^2(\Omega)$ to \eqref{qlinearized2} on the domain $\Omega$.}
\end{lemma}

\begin{proof}
With the estimates \eqref{aH1}-\eqref{aH21} at hand, one can prove the existence and uniqueness of the strong $H^2$ solution $\Psi^{\sigma}$ to \eqref{au2} by the finite Fourier series approximation as in Lemma \ref{exist2}. Since the estimates \eqref{aH1}-\eqref{aH21} are uniformly in $\sigma$, one can further extract a subsequence $\{\Psi^{\sigma_j}\}_{j=1}^{\infty}$ which converges weakly to $\Psi$ in $H^2(\Omega_2)$ as $\sigma_j\to 0$. This function $\Psi$ satisfies the estimate \eqref{au100} and solves the problem \eqref{au2}.

Let $v=\Psi-\psi$. Then $v\in H^2(\Omega)$ satisfies
\be\no\begin{cases}
\sum_{i,j=1}^2 k_{ij} \p_{x_i x_j}^2 v + k_1\p_{x_1} v=0,\ \ & (x_1,x_2)\in \Omega,\\
v(L_0,x_2)=0,\ \ \ & x_2\in (-1,1),\\
\p_{x_2}v(x_1,\pm 1)=0,\ \ & x_1\in (L_0,L_1).
\end{cases}\ee
Then an energy estimate as in Lemma \ref{qH1estimate} yields that $\iint_{\Omega}|\nabla v|^2 dx_1 dx_2=0$ and thus $\nabla v\equiv 0$. Since $v(L_0,x_2)=0$, one has $v(x_1,x_2)\equiv 0$ on $\Omega$. Then Lemma \ref{equality} is proved.
\end{proof}

\begin{lemma}({\bf Interior $H^3$ estimate.})\label{aH3}
{\it There exists a constant $\delta_*>0$ depending only on the background flow, such that if $0<\delta_0\leq \delta_*$ in \eqref{ae}, the classical solution to
\eqref{au2} satisfies
\be\label{aH31}
\sigma\int_{\frac{1}{2}L_0}^{L_1+14\ell}\int_{-1}^1 |\p_{x_1}^4\Psi^{\sigma}|^2 dx_2 dx_1+ \int_{\frac{1}{2}L_0}^{L_1+14\ell}\int_{-1}^1 |\nabla\p_{x_1}^2\Psi^{\sigma}|^2 dx_2 dx_1 \leq C_{\sharp}\|F_0\|_{H^2(\Omega_2)}^2,
\ee
where the constant $C_{\sharp}$ depends only on the $C^3(\overline{\Omega_2})$ norms of the coefficients $a_{11}, a_{12}$ and $a_1$.
}\end{lemma}

\begin{proof}

Define smooth cutoff functions $0\leq \eta_j(x_1)\leq 1$ on $[L_0,L_2]$ for $j=5,6$ such that
\be\no
\eta_5(x_1)=\begin{cases}
0,\ \ \ & L_0\leq x_1\leq \frac{3}{4}L_0,\\
1,\ \ \ & \frac{5}{8}L_0\leq x_1\leq \frac{1}{2}L_0,\\
0,\ \ \ & \frac{3}{8}L_0\leq x_1\leq L_2,
\end{cases}\ \ \ \ \eta_6(x_1)=\begin{cases}
0,\ \ \ & L_0\leq x_1\leq L_1+13\ell,\\
1,\ \ \ & L_1+14\ell\leq x_1\leq L_1+15\ell,\\
0,\ \ \ & L_1+16\ell\leq x_1\leq L_2.
\end{cases}
\ee

Multiplying the equation \eqref{aw1} by $\eta_j^2 \p_{x_1}^2 W_1$ for $j=5,6$ respectively and integrating by parts yield that
\be\no
&&\iint_{\Omega_2} \eta_j^2 (a_{11}(\p_{x_1}^2 W_1)^2 +(\p_{x_1x_2}^2 W_1)^2) =\sigma \iint_{\Omega_2}\eta_j \eta_j' (\p_{x_1}^2 W_1)^2 dx_1 dx_2-\iint_{\Omega_2} 2 \eta_j^2
a_{12} \p_{x_1x_2}^2 W_1\p_{x_1}^2 W_1 \\\no
&& -\iint_{\Omega_2} 2 \eta_j \eta_j'  \p_{x_1x_2}^2W_1 \p_{x_2}W_1 dx_1 dx_2 +\iint_{\Omega_2}\eta_j^2(F_1 - a_3\p_{x_1} W_1-a_4\p_{x_2}W_1) \p_{x_1}^2 W_1 dx_1 dx_2.
\ee
Since the supports of $\eta_j'(x_1)$ are contained in $[\frac{3}{4}L_0, L_1+16\ell]$ for all $j=5,6$, one could use \eqref{aH21} to control the first term on the right hand side. Then it holds that
\be\no
&\quad&\int_{\frac{5}{8}L_0}^{\frac{1}{2}L_0} \int_{-1}^1 |\nabla \p_{x_1}W_1|^2 dx_2 dx_1+ \int_{L_1+14\ell}^{L_1+15\ell}\int_{-1}^1 |\nabla\p_{x_1}W_1|^2 dx_2 dx_1\\\no
&\leq& C_*\sigma \iint_{\Omega_2}|\eta_5'|+|\eta_6'|)|\p_{x_1}^3\Psi|^2 dx_1 dx_2+ C_*\iint_{\Omega_2}|\nabla W_1|^2 +F_1^2 dx_1 dx_2\\\label{aH32}
&\leq& C_*\iint_{\Omega_2}F_0^2+|\nabla F_0|^2 dx_1 dx_2.
\ee

Set $W_2=\p_{x_1}W_1$. Then $W_2$ solves
\be\label{aw2}\begin{cases}
\sigma\p_{x_1}^3 W_2+\displaystyle\sum_{i,j=1}^2 a_{ij}\p_{x_i x_j}^2 W_2 + a_5 \p_{x_1} W_2 + a_6 \p_{x_2} W_2=F_2,\ \ &\text{on }\ \ \Omega_2,\\
W_2(L_0,x_2)=0,\ \ \ &\text{on }\ \ (-1,1),\\
\p_{x_2} W_2(x_1,\pm 1)=0,\ \ \ &\text{on }\ \ [L_0,L_2],
\end{cases}\ee
where
\be\no
&&a_5= a_3 + \p_{x_1} a_{11}= a_1 + 2\p_{x_1} a_{11},\ \ \ \ a_6=a_4+ 2 \p_{x_1} a_{12}=4 \p_{x_1} a_{12},\\\no
&&F_2=\p_{x_1}^2 F_0-(2\p_{x_1}a_1+\p_{x_1}^2 a_{11}) W_2-2\p_{x_1}^2 a_{12}\p_{x_2} W_1- \p_{x_1}^2 a_1 W_1.
\ee
Since the coefficient $2\p_{x_1}a_1+\p_{x_1}^2 a_{11}$ of $W_2$ in $F_2$ may change its sign in general, it seems difficult to get an estimate if one puts the term $-(2\p_{x_1}a_1+\p_{x_1}^2 a_{11}) W_2$ on the left hand side of the equation in \eqref{aw2}. Thus we just regard it as a source term.

Define a smooth cutoff function $0\leq \eta_7(x_1)\leq 1$ on $[L_0,L_2]$ such that
\be\no
\eta_7(x_1)=\begin{cases}
0,\ \ \ & L_0\leq x_1\leq \frac{5}{8}L_0,\\
1,\ \ \ & \frac{1}{2}L_0\leq x_1\leq L_1+14\ell,\\
0,\ \ \ & L_1+15\ell\leq x_1\leq L_2.
\end{cases}
\ee

Multiplying the equation \eqref{aw2} by $\eta_7^2 d(x_1) \p_{x_1} W_2$ and integrating over $\Omega_2$ yield
\be\no
&&-\sigma \iint_{\Omega_2}\eta_7^2 d(x_1) (\p_{x_1}^2 W_2)^2 dx+ \iint_{\Omega_2}[\eta_7^2d a_5 -\frac{1}{2}\p_{x_1}(\eta_7^2d a_{11})-\eta_7^2d \p_{x_2} a_{12}](\p_{x_1} W_2)^2 dx\\\label{aw200}
&&+\iint_{\Omega_2} [\frac{\eta_7^2}{2}d'(x_1)+\eta_7\eta_7' d ](\p_{x_2} W_2)^2+ \eta_7^2  a_6 d \p_{x_1} W_2 \p_{x_2} W_2  dx_1
dx_2\\\no
&&\quad= \iint_{\Omega_2} \eta_7^2 d \p_{x_1} W_2 F_2 dx_1 dx_2+ \sigma \iint_{\Omega} (\eta_7^2 d)'(x_1)\p_{x_1} W_2 \p_{x_1}^2 W_2 dx_1 dx_2.
\ee
By \eqref{apos3}-\eqref{apos4} for $j=2$, it follows from \eqref{ae} that there exists a constant $\delta_*>0$ such that if $0<\delta_0\leq \delta_*$, it holds that
\begin{eqnarray}\nonumber
&&\eta_7^2 a_5 d-\frac12 \p_{x_1}(\eta_7^2 d a_{11})-\eta_7^2 d\partial_{x_2} a_{12}\\\no
&&=\frac{1}{2}\eta_7^2 d(2 a_1+3\p_{x_1} a_{11}-2 \p_{x_2} a_{12})- \frac{1}{2}a_{11}(\eta_7^2 d'+ 2 \eta_7 \eta_7' d)\\\label{e5}
&&\geq \frac12\eta_7^2 \{d(2\bar{a}_1 + 3\bar{a}_{11}')-6\bar{a}_{11}\} -\eta_7^2 d(\|a_1-\bar{a}_{1}\|_{L^{\infty}}+\frac32\|\p_{x_1} a_{11}-\bar{a}_{11}'\|_{L^{\infty}}+\|\partial_{x_2} a_{12}\|_{L^{\infty}})\\\no
&&\quad\quad-3\eta_7^2\|a_{11}-\bar{a}_{11}\|_{L^{\infty}} - \eta_7 \eta_7'a_{11} d\geq 3\eta_7^2- \eta_7 \eta_7' a_{11}d, \ \ \forall (x_1,x_2)\in \Omega_2,\\\no
&&|\eta_7^2 da_6|\leq C_*\delta_0\eta_7^2\leq C_* \delta_*\eta_7^2,
\end{eqnarray}
due to the Sobolev embedding $H^3(\Omega_2)\subset C^{1,\alpha}(\overline{\Omega_2})$ with $\alpha\in(0,1)$.

Then one can conclude from \eqref{aw200} that
\be\no
&&\sigma \iint_{\Omega_2}\eta_7^2 (\p_{x_1}^2 W_2)^2 dx_1 dx_2+ \iint_{\Omega_2}\eta_7^2 |\nabla W_2|^2 dx_1 dx_2\\\label{aH33}
&&\leq C_*\iint_{\Omega_2}\eta_7^2F_2^2+|\eta_7'(x_1)|^2 |\nabla W_2|^2 dx_1 dx_2.
\ee
Since the support of $\eta_7'(x_1)$ is contained in $[\frac{5}{8}L_0,\frac{3}{8}L_0]\cup [L_1+14\ell, L_1+15\ell]$, one can use \eqref{aH32} to control the term $\iint_{\Omega_2}|\eta_7'(x_1)|^2 |\nabla W_2|^2 dx_1 dx_2 $. Note also that
\be\no
&&\|\eta_7F_2\|_{L^2}\leq C_*\|F_0\|_{H^2}+\|2\p_{x_1} a_1+\p_{x_1}^2a_{11}\|_{L^{\infty}}\|\eta_7 W_2\|_{L^2}+\|\p_{x_1}^2 a_{12}\|_{L^{\infty}}\|\eta_7\p_{x_2} W_1\|_{L^2}+\|\p_{x_1}^2 a_1 \|_{L^4}\|W_1\|_{L^4}\\\no
&&\leq C_*\|F_0\|_{H^2(\Omega_2)}+C_{\sharp}\|F_0\|_{H^1}+\|a_1\|_{H^3}\|\p_{x_1}\Psi\|_{H^1}\leq C_{\sharp}\|F_0\|_{H^2}.
\ee
Then the estimate \eqref{aH31} follows from these and \eqref{aH33}.

\end{proof}

\begin{lemma}({\bf Interior $H^4$ estimate.})\label{aH4}
{\it There exists a constant $\delta_*>0$ depending only on the background flow, such that if $0<\delta_0\leq \delta_*$ in \eqref{ae}, the classical solution $\Psi^{\sigma}$ to
\eqref{au2} satisfies
\be\label{aH41}
\sigma\int_{\frac{1}{4}L_0}^{L_1+12\ell}\int_{-1}^1 |\p_{x_1}^5\Psi^{\sigma}|^2 dx_2 dx_1+ \int_{\frac{1}{4}L_0}^{L_1+12\ell}\int_{-1}^1 |\nabla\p_{x_1}^3\Psi^{\sigma}|^2 dx_2 dx_1 \leq C_{\sharp}\|F_0\|_{H^3(\Omega_2)}^2,
\ee
where the constant $C_{\sharp}$ depends only on the $C^3(\overline{\Omega_2})$ norms of the coefficients $a_{11}, a_{12}$ and $a_1$.
}\end{lemma}

\begin{proof}

Define smooth cutoff functions $0\leq \eta_j(x_1)\leq 1$ on $[L_0,L_2]$ for $j=8,9$ such that
\be\no
\eta_8(x_1)=\begin{cases}
0,\ \ \ & L_0\leq x_1\leq \frac{1}{2}L_0,\\
1,\ \ \ & \frac{3}{8}L_0\leq x_1\leq \frac{1}{4}L_0,\\
0,\ \ \ & \frac{1}{8}L_0\leq x_1\leq L_2,
\end{cases}\ \ \ \ \eta_9(x_1)=\begin{cases}
0,\ \ \ & L_0\leq x_1\leq L_1+11\ell,\\
1,\ \ \ & L_1+12\ell\leq x_1\leq L_1+13\ell,\\
0,\ \ \ & L_1+14\ell\leq x_1\leq L_2.
\end{cases}
\ee

Multiplying the equation \eqref{aw2} by $\eta_j^2 \p_{x_1}^2 W_2$ for $j=8,9$ respectively and integrating by parts yield that
\be\no
&&\iint_{\Omega_2} \eta_j^2 (a_{11}(\p_{x_1}^2 W_2)^2 +(\p_{x_1x_2}^2 W_2)^2) =\sigma \iint_{\Omega_2}\eta_j \eta_j' (\p_{x_1}^2 W_2)^2 dx_1 dx_2-\iint_{\Omega_2} 2 \eta_j^2
a_{12} \p_{x_1x_2}^2 W_2\p_{x_1}^2 W_2 \\\no
&& -\iint_{\Omega_2} 2 \eta_j \eta_j'  \p_{x_1x_2}^2W_2 \p_{x_2}W_2 dx_1 dx_2 +\iint_{\Omega_2}\eta_j^2(F_2 - a_5\p_{x_1} W_2-a_6\p_{x_2}W_2) \p_{x_1}^2 W_2 dx_1 dx_2.
\ee
Since the supports of $\eta_j'(x_1)$ are contained in $[\frac{1}{2}L_0, L_1+14\ell]$, one could use \eqref{aH31} to control the first term on the right hand side above to get
\be\label{aH42}
&&\int_{\frac{3}{8}L_0}^{\frac{1}{4}L_0} \int_{-1}^1 |\nabla \p_{x_1}W_2|^2 dx_2 dx_1+ \int_{L_1+12\ell}^{L_1+13\ell}\int_{-1}^1 |\nabla\p_{x_1}W_2|^2 dx_2 dx_1\\\no
&&\leq C_{\sharp}\|F_0\|_{H^2}^2+ C_*\iint_{\Omega_2}\eta_j^2(|\nabla W_2|^2 +F_2^2) dx_1 dx_2\leq C_{\sharp}\|F_0\|_{H^2}^2.
\ee

Set $W_3=\p_{x_1}W_2$. Then $W_3$ solves
\be\label{aw3}\begin{cases}
\sigma\p_{x_1}^3 W_3+\displaystyle\sum_{i,j=1}^2 a_{ij}\p_{x_i x_j}^2 W_3 + a_7 \p_{x_1} W_3 + a_8 \p_{x_2} W_3=F_3,\ \ &\text{on }\ \ \Omega_2,\\
\p_{x_2} W_3(x_1,\pm 1)=0,\ \ \ &\text{on }\ \ \ x_1\in [L_0,L_2],
\end{cases}\ee
where
\be\no
&&a_7= a_5 + \p_{x_1} a_{11}= a_1 + 3\p_{x_1} a_{11},\ \ \ \ a_8=a_6+ 2 \p_{x_1} a_{12}=6 \p_{x_1} a_{12},\\\no
&&F_3= \p_{x_1}^3 F_0-(3\p_{x_1}a_1+\p_{x_1} a_{11}+ 2\p_{x_1}^2 a_{11})W_3-(3\p_{x_1}^2 a_1+\p_{x_1}^3 a_{11}) W_2\\\no
&&\quad\quad\quad\quad-2\p_{x_1}^3 a_{12}\p_{x_2} W_1-\p_{x_1}^3 a_1 W_1-6\p_{x_1}^2a_{12}\p_{x_2}W_2.
\ee
Define a smooth cutoff function $0\leq \eta_{10}(x_1)\leq 1$ on $[L_0,L_2]$ such that
\be\no
\eta_{10}(x_1)=\begin{cases}
0,\ \ \ & L_0\leq x_1\leq \frac{3}{8}L_0,\\
1,\ \ \ & \frac{1}{4}L_0\leq x_1\leq L_1+12\ell,\\
0,\ \ \ & L_1+13\ell\leq x_1\leq L_2.
\end{cases}
\ee
Multiplying the equation \eqref{aw2} by $\eta_{10}^2 d(x_1) \p_{x_1} W_2$ and integrating over $\Omega_2$ show that
\be\no
&&-\sigma \iint_{\Omega_2}\eta_{10}^2 d(x_1) (\p_{x_1}^2 W_3)^2 dx+ \iint_{\Omega_3}[\eta_{10}^2d a_7 -\frac{1}{2}\p_{x_1}(\eta_{10}^2d a_{11})-\eta_{10}^2d \p_{x_2} a_{12}](\p_{x_1} W_3)^2 dx\\\no
&&+\iint_{\Omega_2} [\frac{\eta_{10}^2}{2}d'(x_1)+\eta_{10}\eta_{10}' d ](\p_{x_2} W_3)^2+ \eta_{10}^2  a_8 d \p_{x_1} W_3 \p_{x_2} W_3  dx_1
dx_2\\\no
&&\quad= \iint_{\Omega_2} \eta_{10}^2 d \p_{x_1} W_3 F_3 dx_1 dx_2+ \sigma \iint_{\Omega} (\eta_{10}^2 d)'(x_1)\p_{x_1} W_3 \p_{x_1}^2 W_3 dx_1 dx_2.
\ee
It follows from \eqref{apos3}-\eqref{apos4} for $j=3$ that
\be\no
&&\sigma \iint_{\Omega_2}\eta_{10}^2 (\p_{x_1}^2 W_3)^2 dx_1 dx_2+ \iint_{\Omega_2}\eta_{10}^2 |\nabla W_3|^2 dx_1 dx_2\\\no
&&\leq C_*\iint_{\Omega_2}\eta_{10}^2F_3^2+|\eta_{10}'| |\nabla W_3|^2 dx_1 dx_2.
\ee
The term $\iint_{\Omega_2}|\eta_{10}'| |\nabla W_3|^2 dx$ can be controlled by \eqref{aH42}, since the support of $\eta_{10}'(x_1)$ is contained in $[\frac{1}{2}L_0,\frac{3}{8}L_0]\cup [L_1+12\ell, L_1+13\ell]$. And $\|\eta_{10} F_3\|_{L^2}$ can be estimated as
\be\no
&&\|\eta_{10}F_3\|_{L^2}\leq \|F_0\|_{H^3}+\|3\p_{x_1}a_1+\p_{x_1} a_{11}+ 2\p_{x_1}^2 a_{11}\|_{L^{\infty}}\|\eta_{10} W_3\|_{L^2}+ \|3\p_{x_1}^2 a_1+\p_{x_1}^3 a_{11}\|_{L^{\infty}}\|W_2\|_{L^2}\\\no
&&\quad +\|\p_{x_1}^3 a_{12}\|_{L^{\infty}}\|\p_{x_2} W_1\|_{L^2}+\|\p_{x_1}^3 a_{1}\|_{L^{\infty}}\|W_1\|_{L^2}+\|\p_{x_1}^2 a_{12}\|_{L^{\infty}}\|\eta_{10}\p_{x_2} W_2\|_{L^2}\\\no
&&\leq \|\p_{x_1}^3F_0\|_{L^2}+ C_{\sharp} \|F_0\|_{H^2}\leq C_{\sharp} \|F_0\|_{H^3}.
\ee
Here one has used \eqref{aH31} to control $\|\eta_{10} \nabla W_2\|_{L^2}$ since the support of $\eta_{10}$ is contained in $(\frac{3}{8}L_0,L_1+13\ell)$. Thus  \eqref{aH41} follows.

\end{proof}

Now we can improve the regularity of the solution $\psi$ to \eqref{qlinearized2} to be $H^4(\Omega)$.
\begin{lemma}\label{H4first}
{\it The $H^2$ strong solution $\psi$ to \eqref{qlinearized2} belongs to $H^4(\Omega)$ such that
\be\label{aH400}
\|\psi\|_{H^4(\Omega)}\leq C_{\sharp}\|G_0\|_{H^3(\Omega)},
\ee
where the constant $C_{\sharp}$ depends only on the $C^3(\overline{\Omega})$ norms of the coefficients $k_{11}, k_{12}$ and $k_1$.
}\end{lemma}

\begin{proof}
The estimates in Lemmas \ref{aH2},\ref{aH3} and \ref{aH4} can be obtained for the finite Fourier series approximation $\Psi^{N,\sigma}=\sum_{j=1}^N A_j^{N,\sigma}(x_1) b_j(x_2)$ by same arguments as in Lemma \ref{exist1}. These estimates are uniformly in $N$. Thus one can extract a weakly convergent subsequence such that its weak limit coincides with the $H^2$ strong solution $\Psi^{\sigma}$ to \eqref{au2} due to the uniqueness. Moreover, the following estimate holds
\be\label{aH401}
\|\Psi^{\sigma}\|_{H^2(\Omega_2)}^2+\int_{\frac{1}{4}L_0}^{L_1+12\ell}\int_{-1}^1 |\nabla\p_{x_1}^2\Psi^{\sigma}|^2+ |\nabla\p_{x_1}^3\Psi^{\sigma}|^2 dx_2 dx_1 \leq C_{\sharp}\|F_0\|_{H^3(\Omega_2)}^2\leq C_{\sharp}\|G_0\|_{H^3(\Omega)}^2.
\ee
The estimate \eqref{aH401} is uniformly in $\sigma$, thus one can extract a subsequence $\{\Psi^{\sigma_j}\}_{j=1}^{\infty}$ which converges weakly to a function $\widetilde{\Psi}$ with the estimate
\be\label{aH402}
\|\widetilde{\Psi}\|_{H^2(\Omega_2)}^2+\int_{\frac{1}{4}L_0}^{L_1+12\ell}\int_{-1}^1 |\nabla\p_{x_1}^2\widetilde{\Psi}|^2+ |\nabla\p_{x_1}^3\widetilde{\Psi}|^2 dx_2 dx_1 \leq C_{\sharp}\|G_0\|_{H^3(\Omega)}^2.
\ee
The function $\widetilde{\Psi}$ coincides with the solution $\Psi$ constructed in Lemma \ref{equality} due to the uniqueness of the solution to \eqref{au1}. Moreover, by Lemma \ref{equality}, one has $\Psi=\psi$ in $\Omega$, therefore $\psi$ satisfies
\be\label{aH403}
\int_{\frac{1}{4}L_0}^{L_1}\int_{-1}^1 |\nabla\p_{x_1}^2\psi|^2+ |\nabla\p_{x_1}^3\psi|^2 dx_2 dx_1 \leq C_{\sharp}\|G_0\|_{H^3(\Omega)}^2.
\ee
This, together with \eqref{qH31}, yields that
\be\label{aH404}
\|\nabla\p_{x_1}^2\psi\|_{L^2(\Omega)}^2+\|\nabla\p_{x_1}^3\psi\|_{L^2(\Omega)}\leq C_{\sharp}\|G_0\|_{H^3(\Omega)}.
\ee
Since the following equality holds almost everywhere
\be\label{aH45}
\p_{x_2}^2\psi= G_0-k_{11}\p_{x_1}^2\psi- 2 k_{12}\p_{x_1x_2}^2\psi- k_1\p_{x_1}\psi,
\ee
one can further prove that the weak derivatives $\p_{x_1}\p_{x_2}^2\psi,\p_{x_1}\p_{x_2}^2\psi$ and $\p_{x_1}^2\p_{x_2}^2\psi, \p_{x_1}\p_{x_2}^3\psi, \p_{x_2}^4\psi$ exist and satisfy similar estimates as in \eqref{aH404}. Thus \eqref{aH400} is proved.

\end{proof}

Finally, we show that the constant $C_{\sharp}$ in \eqref{aH400} can be replaced by a constant $C_*$ which depends only on the $H^3(\Omega)$ norms of the coefficients $k_{11}, k_{12}$ and $k_1$.
\begin{lemma}({\bf $H^4$ estimate.})\label{qH4}
{\it There exists a constant $\delta_*>0$ depending only on the background flow, such that if $0<\delta_0\leq \delta_*$ in \eqref{qcoe}, the solution to
\eqref{qlinearized2} satisfies the compatibility condition
\be\label{cp3}
\p_{x_2}^3 \psi(x_1,\pm 1)=0 \ \ \text{in the sense of $H^1(\Omega)$ trace},
\ee
and the estimates
\be\label{f3}
&&\|\psi\|_{H^3(\Omega)} \leq C_* \|G_0\|_{H^2(\Omega)},\\\label{f4}
&&\|\psi\|_{H^4(\Omega)} \leq C_* \|G_0\|_{H^3(\Omega)},
\ee
where the constant $C_*$ depends only on the $H^3(\Omega)$ norms of the coefficients $k_{11}, k_{12}$ and $k_1$.
}\end{lemma}

\begin{proof}

Since it has shown that $\psi\in H^4(\Omega)$, then $w_1=\p_{x_1}\psi$ satisfies the following equation almost everywhere
\be\label{f1}\begin{cases}
\mathcal{L}_1 w_1:=\displaystyle\sum_{i,j=1}^2 k_{ij}\p_{x_i x_j}^2 w_1 + k_3 \p_{x_1} w_1+ k_4 \p_{x_2} w_1=G_1,\\
\p_{x_2} w_1(x_1,\pm 1)=0,\ \ \ \forall x_1\in [L_0,L_1],
\end{cases}\ee
where
\be\no
&&k_3= k_1 + \p_{x_1} k_{11},\ \ \ \ k_4=2 \p_{x_1} k_{12},\\\no
&&G_1=\p_{x_1} G_0-\p_{x_1} k_1 w_1.
\ee
Let $\eta$ be a monotone increasing smooth cutoff function on $[L_0,L_1]$ such that $0\leq \eta\leq 1$ and
\be\no
\eta(x_1)=\begin{cases}
0, \ \ &L_0\leq x_1\leq \frac{3L_0}{4},\\
1, \ \ &\frac{L_0}{2}\leq x_1\leq L_1.
\end{cases}\ee
Then $\tilde{w}_1=\eta w_1$ would satisfy
\be\label{f11}\begin{cases}
\mathcal{L}_1 \tilde{w}_1:=\displaystyle\sum_{i,j=1}^2 k_{ij}\p_{x_i x_j}^2 \tilde{w}_1 + k_3 \p_{x_1} \tilde{w}_1+ k_4 \p_{x_2} \tilde{w}_1=\tilde{G}_1, \ \ &\text{on }\ \ \Omega,\\
\tilde{w}_1(L_0,x_2)=0,\ \ &\text{on }\ \  (-1,1),\\
\p_{x_2} \tilde{w}_1(x_1,\pm 1)=0,\ \ &\text{on }\ \ [L_0,L_1],
\end{cases}\ee
where $\tilde{G}_1=\eta G_1+ \sum_{i,j=1}^2 k_{ij}(\p_{x_i}\eta\p_{x_j} w_1+ \p_{x_i} w_1\p_{x_j}\eta)+ w_1\mathcal{L}_1\eta$.

Note that if $0<\delta_0\leq \delta_*$ in \eqref{qcoe}, then there holds for any $(x_1,x_2)\in\Omega$
\be\no
&& 2k_3-\p_{x_1} k_{11}= 2k_1+ \p_{x_1} k_{11}\leq 2\bar{k}_1+\bar{k}_{11}'+2\|k_1-\bar{k}_{11}\|_{L^{\infty}}+\|\p_{x_1} k_{11}-\bar{k}_{11}'\|_{L^{\infty}}\leq -\kappa_*<0,\\\no
&& 2k_3+\p_{x_1} k_{11}= 2k_1+ 3\p_{x_1} k_{11}\leq 2\bar{k}_1+3\bar{k}_{11}'+2\|k_1-\bar{k}_{11}\|_{L^{\infty}}+3\|\p_{x_1} k_{11}-\bar{k}_{11}'\|_{L^{\infty}}\leq -\kappa_*<0.
\ee
Then as in Lemma \ref{exist2}, one can show that there exists a unique strong solution $v_1\in H^2(\Omega)$ to
\be\label{f12}\begin{cases}
\mathcal{L}_1v_1=\displaystyle\sum_{i,j=1}^2 k_{ij}\p_{x_i x_j}^2 v_1 + k_3 \p_{x_1} v_1 + k_4 \p_{x_2} v_1=\tilde{G}_1,\ \ &\text{on }\ \Omega,\\
v_1(L_0,x_2)=0,\ \ \ &\text{on }\ \  (-1,1),\\
\p_{x_2} v_1(x_1,\pm 1)=0,\ \ \ &\text{on }\ \  [L_0,L_1],
\end{cases}\ee
with the estimate
\be\label{f13}
\|v_1\|_{H^2(\Omega)}\leq C_*\|\tilde{G}_1\|_{H^1(\Omega)}.
\ee
By the uniqueness, $v_1=\tilde{w}_1$ holds a.e. in $\Omega$. Thus
\be\label{f14}
&&\left(\int_{\frac{L_0}{2}}^{L_1}\int_{-1}^1 |\nabla^2 w_1|^2 dx_2 dx_1\right)^{\frac{1}{2}}\leq \|\tilde{w}_1\|_{H^2(\Omega)}\leq C_*\|\tilde{G}_1\|_{H^1(\Omega)}\\\no
&&\leq C_*(\|G_0\|_{H^2}+\|\p_{x_1}k_1\|_{H^2}\|\p_{x_1}\psi\|_{H^1}+\|k_{ij}\|_{H^2}\|\eta' \nabla w_1\|_{H^1}+\|\mathcal{L}_1\eta\|_{H^2}\|w_1\|_{H^1})\\\no
&&\leq C_*\|G_0\|_{H^2(\Omega)}.
\ee
Combining this with \eqref{qH31}, one can conclude that
\be\label{f15}
\|\nabla^2 \p_{x_1}\psi\|_{L^2(\Omega)}\leq C_*\|G_0\|_{H^2(\Omega)}.
\ee
Using the equation in \eqref{aH45}, one can infer that $\p_{x_2}^3 \psi$ also satisfies the same inequality. Thus \eqref{f3} holds.

Set $w_2=\p_{x_1}w_1$. Then it holds that
\be\label{f2}\begin{cases}
\mathcal{L}_2 w_2:=\displaystyle\sum_{i,j=1}^2 k_{ij}\p_{x_i x_j}^2 w_2 + k_5 \p_{x_1} w_2 + k_6 \p_{x_2} w_2=G_2,\ \ &\text{on }\ \ \Omega,\\
\p_{x_2} w_2(x_1,\pm 1)=0,\ \ \ \text{on } \ \  [L_0,L_1],
\end{cases}\ee
where
\be\no
&&k_5= k_1 + 2\p_{x_1} k_{11},\ \ \ \ k_6=4 \p_{x_1} k_{12},\\\no
&&G_2=\p_{x_1}^2 G_0-(2\p_{x_1}k_1+\p_{x_1}^2 k_{11}) w_2-2\p_{x_1}^2 k_{12}\p_{x_2} w_1- \p_{x_1}^2 k_1 w_1.
\ee

Set $\tilde{w}_2= \eta w_2$. Then $\tilde{w}_2$ solves
\be\label{f21}\begin{cases}
\mathcal{L}_2 \tilde{w}_2=\displaystyle\sum_{i,j=1}^2 k_{ij}\p_{x_i x_j}^2 \tilde{w}_2 + k_5 \p_{x_1} \tilde{w}_2 + k_6 \p_{x_2} \tilde{w}_2=\tilde{G}_2,\ &\text{on }\ \Omega,\\
\tilde{w}_2(L_0,x_2)=0,\ \ \ &\text{on } \ \ (-1,1),\\
\p_{x_2} \tilde{w}_2(x_1,\pm 1)=0,\ \ \ &\text{on }\ \  [L_0,L_1],
\end{cases}\ee
where
$$
\tilde{G}_2=\eta G_2+ \sum_{i,j=1}^2 k_{ij}(\p_{x_i}\eta\p_{x_j} w_2+\p_{x_i}w_2\p_{x_j} \eta)+ w_2\mathcal{L}_2\eta.
$$
Note that if $0<\delta_0\leq \delta_*$ in \eqref{qcoe}, then there holds that for any $(x_1,x_2)\in\Omega$
\be\no
&& 2k_5-\p_{x_1} k_{11}= 2k_1+ 3\p_{x_1} k_{11}\leq 2\bar{k}_1+3\bar{k}_{11}'+2\|k_1-\bar{k}_{11}\|_{L^{\infty}}+3\|\p_{x_1} k_{11}-\bar{k}_{11}'\|_{L^{\infty}}\leq -\kappa_*<0,\\\no
&& 2k_5+\p_{x_1} k_{11}= 2 k_1+ 5\p_{x_1} k_{11}\leq 2\bar{k}_1+5\bar{k}_{11}'+2\|k_1-\bar{k}_{11}\|_{L^{\infty}}+5\|\p_{x_1} k_{11}-\bar{k}_{11}'\|_{L^{\infty}}\leq -\kappa_*<0.
\ee
Then as in Lemma \ref{exist2}, one can show that $\tilde{w}_2$ is the unique $H^2(\Omega)$ strong solution to \eqref{f21} with the estimate
\be\no
&&\left(\int_{\frac{L_0}{2}}^{L_1}\int_{-1}^1 (|\nabla w_2|^2+|\nabla^2 w_2|^2) dx_2 dx_1\right)^{\frac{1}{2}}\leq \|\tilde{w}_2\|_{H^2(\Omega)}\leq C_*\|\tilde{G}_2\|_{H^1(\Omega)}\\\no
&&\leq C_*(\|G_0\|_{H^3}+\|(\bar{k}_1'+\bar{k}_{11}'')w_2\|_{H^1}+\|\p_{x_1} k_1-\bar{k}_1+ \p_{x_1}^2 k_{11}-\bar{k}_{11}''\|_{H^1}\|w_2\|_{H^2}\\\no
&&\quad\quad+\|\p_{x_1}^2 k_{12}\|_{H^1}\|\p_{x_2} w_1\|_{H^2}+\|\p_{x_1}^2 k_{11}\|_{H^1}\|w_1\|_{H^2}+\|k_{ij}\|_{H^2}\|\eta' \nabla w_2\|_{H^1}+\|\mathcal{L}_2\eta\|_{H^2}\|w_2\|_{H^1})\\\no
&&\leq C_*(\|G_0\|_{H^3(\Omega)}+\|\psi\|_{H^3(\Omega)}+\delta_0\|\psi\|_{H^4(\Omega)})\leq C_*(\|G_0\|_{H^3(\Omega)}+\delta_0\|\psi\|_{H^4(\Omega)}).
\ee
It follows from \eqref{qH31} and the equation \eqref{aH45} that
\be\no
\|\psi\|_{H^4(\Omega)}\leq C_*(\|G_0\|_{H^3(\Omega)}+\delta_0\|\psi\|_{H^4(\Omega)}).
\ee
Let $0<\delta_0\leq \delta_*$ so that $C_*\delta_*\leq \frac{1}{2}$. Then \eqref{f4} follows.

It remains to prove the compatibility condition \eqref{cp3}. By \eqref{c1}, it suffices to show that \eqref{cp3} holds on $[\frac{1}{8}L_0, L_1]$. Suppose $k_{11}, k_{12}$ and $k_1\in C^4(\overline{\Omega})$, then the coefficients $a_{11}, a_{12}$ and $a_1\in C^4(\overline{\Omega_2})$. One may obtain the $L^2$ estimate of $\nabla \p_{x_1}^4\Psi^{\sigma}$ on the domain $D_3:=(\frac{L_0}{8}, L_1+10\ell)\times (-1,1)$ as in Lemma \ref{aH4}, and then derive the estimate of $\|\Psi\|_{H^5(D_3)}$, which implies that $\Psi\in C^{3,\alpha}(\overline{D_3})$ and $\psi\in C^{3,\alpha}([\frac{L_0}{8},L_1]\times [-1,1])$. Then the fact $\p_{x_2}^3\psi(x_1,\pm 1)=0$ for any $x_1\in [\frac{1}{8}L_0, L_1]$ follows by differentiating the equation \eqref{aH45} with respect to $x_2$ and evaluating at $x_2=\pm 1$. The general case follows by a density argument.

\end{proof}

\subsection{Proof of Theorem \ref{q-irro}}\noindent

We are now ready to prove Theorem \ref{q-irro}. For any $\hat{\psi}\in \Sigma_{\delta_0}$, then \eqref{qcoe} holds. By Lemma \ref{exist2} and Lemma \ref{qH4}, there exists a unique solution $\psi\in H^4(\Omega)$ to \eqref{qlinearized2} with the estimate
\be\no
\|\psi\|_{H^4(\Omega)}\leq C_*\|G_0(\nabla \hat{\psi})\|_{H^3(\Omega)}\leq m_*\|G_0(\nabla \hat{\psi})\|_{H^3(\Omega)}.
\ee
Here the constant $C_*$ depends only on the $H^3(\Omega)$ norms of the coefficients $k_{11}, k_{12}$ and $k_1$, which can be bounded by a constant $m_*$ depends on the $C^3([L_0,L_1])$ norm of $\bar{k}_{11},\bar{k}_1$ and the boundary data. In the following, the constant $m_*$ will always denote a constant depending only on the background solutions and the boundary data.

Recall the definition of $G_0(\nabla \hat{\psi})$ in \eqref{g0} and note that the support of $\eta_0(x_1)$ defined in \eqref{eta} is contained in $[L_0,\frac{7}{8}L_0]$. By the $H^4$ estimate \eqref{qH31} in Lemma \ref{qH3} and carefully checking the estimates obtained in Lemmas \ref{aH3},\ref{aH4}, \ref{H4first} and \ref{qH4}, one can get a better estimate as
\be\no
\|\psi\|_{H^4(\Omega)}&\leq& m_*\bigg(\|G(\nabla \hat{\psi})\|_{H^3(\Omega)}+ \epsilon(\|k_{11}(\nabla\hat{\psi})\|_{H^2(\Omega)}+\|k_1(\nabla\hat{\psi})\|_{H^2(\Omega)})\bigg\|\int_{-1}^{x_2} h_1(s) ds\bigg\|_{H^2((-1,1))}\\\no
&\quad&\quad+\epsilon\|k_{12}(\nabla\hat{\psi})\|_{H^2(\Omega)}\|h_1\|_{H^2((-1,1))}+\epsilon\|h_1'\|_{H^2((-1,1))}\bigg)\\\no
&\leq& m_*(\|\hat{\psi}\|_{H^4(\Omega)}^2+ \epsilon \|h_1\|_{H^3((-1,1))})\leq m_*(\epsilon +\delta_0^2).
\ee
Here only the norm $\|h_1\|_{H^3((-1,1))}$ is needed, that is the reason introducing the cut-off function $\eta_0$ in \eqref{g0}.

Let $\delta_0=\sqrt{\epsilon}$, then if $0<\epsilon\leq \epsilon_0=\min\{\frac{1}{4m_*^2}, \delta_*^2\}$, then
\be\no
\|\psi\|_{H^4(\Omega)}\leq m_*(\epsilon+ \delta_0^2)= 2m_*\epsilon\leq \delta_0.
\ee
By \eqref{cp3}, $\psi\in \Sigma_{\delta_0}$. Hence one can define the operator $\mathcal{T}\hat{\psi}=\psi$, which maps $\Sigma_{\delta_0}$ to itself. It remains to show that the mapping $\mathcal{T}$ is contractive in a low order norm for a sufficiently small $\epsilon_0$. Suppose that $\psi^{(i)}=\mathcal{T}\hat{\psi}^{(i)} (i=1,2)$ for any $\hat{\psi}^{(1)}$ and $\hat{\psi}^{(2)}\in \Sigma_{\delta_0}$. Then for $k=1,2$,
\begin{align}\no
\begin{cases}
\mathcal{L}^{(k)}\psi^{(k)}\equiv \displaystyle \sum_{i,j=1}^2 k_{ij}(\nabla \hat{\psi}^{(k)})\partial_{x_i x_j}^2\psi^{(k)}+ k_1(\nabla \hat{\psi}^{(k)})\partial_{x_1}\psi^{(k)}=G_0(\nabla \hat{\psi}^{(k)}),\\
\psi^{(k)}(L_0,x_2)=0,\ \ \ \forall x_2\in (-1,1),\\
\partial_{x_2}\psi^{(k)}(x_1,-1)=\partial_{x_2}\psi^{(k)}(x_1,1)=0,\ \ \forall x_1\in (L_0,L_1).
\end{cases}
\end{align}
Thus
\begin{align}\no
\begin{cases}
\mathcal{L}^{(1)}(\psi^{(1)}-\psi^{(2)})=G_0(\nabla\hat{\psi}^{(1)})-G_0(\nabla\hat{\psi}^{(2)})-(\mathcal{L}^{(1)}-\mathcal{L}^{(2)})\psi^{(2)}\\%
(\psi^{(1)}-\psi^{(2)})(L_0,x_2)=0,\ \ \ \forall x_2\in (-1,1),\\
\partial_{x_2}(\psi^{(1)}-\psi^{(2)})(x_1,-1)=\partial_{x_2}(\psi^{(1)}-\psi^{(2)})(x_1,1)=0,\ \ \forall x_1\in (L_0,L_1).
\end{cases}
\end{align}
Since $\psi^{(i)}$ and $\hat{\psi}^{(i)}\in \Sigma_{\delta_0}$, for $i=1,2$, the $H^1$ estimate in Lemma \ref{qH1estimate} yields that
\begin{eqnarray}\no
&&\|\mathcal{T}\hat{\psi}^{(1)}-\mathcal{T}\hat{\psi}^{(2)}\|_{H^1(\Omega)}=\|\psi^{(1)}-\psi^{(2)}\|_{H^1(\Omega)}\\\no
&&\leq C_*\|G_0(\nabla\hat{\psi}^{(1)})-G_0(\nabla\hat{\psi}^{(2)})-(\mathcal{L}^{(1)}-\mathcal{L}^{(2)})\psi^{(2)}\|_{L^2(\Omega)}\\\no
&&\leq m_* \delta_0\|\hat{\psi}^{(1)}-\hat{\psi}^{(2)}\|_{H^1(\Omega)}\leq \frac{1}{2}\|\hat{\psi}^{(1)}-\hat{\psi}^{(2)}\|_{H^1(\Omega)},
\end{eqnarray}
Therefore $\mathcal{T}$ is a contractive mapping in $H^1$-norm. Then there exists a unique $\psi\in \Sigma_{\delta_0}$ to $\mathcal{T}\psi=\psi$.

In conclusion, we have shown that there exists a small $\epsilon_0>0$ such that for any $0<\epsilon<\epsilon_0$, the problem \eqref{q18} has a unique solution $\psi\in \Sigma_{\delta_0}$ with the estimate $\|\psi\|_4\leq m_*\epsilon$. That is, the background transonic flow is structurally stable within irrotational flows under perturbations of boundary conditions in \eqref{qpbcs}.

Finally, we examine the location of all the sonic points where $|{\bf M}(x_1,x_2)|^2=1$ with ${\bf M}=(M_1,M_2)^t:=(\frac{u_1}{c(\rho)},\frac{u_2}{c(\rho)})^t$. It follows from \eqref{q2d10} and the Sobolev embedding $H^3(\Omega)\hookrightarrow C^{1,\alpha}(\overline{\Omega})$ for any $\alpha\in (0,1)$ that
\be\no
\||{\bf M}|^2-\bar{M}^2\|_{C^{1,\alpha}(\overline{\Omega})}\leq \||{\bf M}|^2-\bar{M}^2\|_{H^3(\Omega)}\leq m_*\epsilon.
\ee
Note that
\be\no
|\bar{M}(L_0)|^2<1, \ |\bar{M}(L_1)|^2>1, \ \ \displaystyle \sup_{x_1\in[L_0,L_1]} \frac{d}{d x_1}\bar{M}^2>0.
\ee
Thus for sufficiently small $\epsilon$, one still has
\be\no
|{\bf M}(L_0,x_2)|^2<1, \ |{\bf M}(L_1,x_2)|^2>1, \ \ \ \forall x_2\in [-1,1],
\ee
and
\be\no
\frac{\p}{\p x_1}|{\bf M}(x_1,x_2)|^2>0, \ \ \ \forall (x_1,x_2)\in\Omega.
\ee
Therefore for each $x_2\in [-1,1]$, there exists a unique $\xi(x_2)\in (L_0,L_1)$ such that $|{\bf M}(\xi(x_2),x_2)|^2=1$. Also by the implicit function theorem, the function $\xi\in C^{1}([-1,1])$. Furthermore, since
\be\no
||{\bar M}(\xi(x_2))|^2-|{\bar M}(0)|^2|&=&|{\bar M}(\xi(x_2))|^2-|{\bf M}(\xi(x_2),x_2)|^2|\\\no
&\leq& \||{\bf M}|^2-|{\bar M}|^2\|_{C^{1,\alpha}(\overline{\Omega})} \leq m_*\epsilon,
\ee
one can deduce that $|\xi(x_2)|\leq m_*\epsilon$ for any $x_2\in [-1,1]$. Differentiating the identity $|{\bf M}(\xi(x_2),x_2)|^2=1$ with respect to $x_2$ yields
\be\no
\xi'(x_2)=-\left(\frac{\p}{\p x_1}|{\bf M}|^2 (\xi(x_2),x_2) \right)^{-1}\frac{\p}{\p x_2}|{\bf M}|^2 (\xi(x_2),x_2)
\ee
and thus the estimate \eqref{qsonic} holds. The proof of Theorem \ref{q-irro} is completed.

\section{The existence of smooth transonic flows with nonzero vorticity}\label{rotation}\noindent

Now we turn to the case of rotational flows and prove Theorem \ref{2dmain}. As the steady two dimensional Euler system, the system \eqref{q2deuler} is elliptic-hyperbolic coupled in subsonic region and changes type when the flow changes smoothly from subsonic to supersonic. To resolve the system \eqref{q2deuler}, one needs to decouple effectively the elliptic and hyperbolic modes for further mathematical analysis. Here we will employ the deformation-curl decomposition developed by the authors in \cite{WengXin19,weng2019} to deal with the elliptic-hyperbolic coupled structure for the quasi two dimensional model \eqref{q2deuler}. The Bernoulli's law yields
\be\label{qdensity}
\rho=\rho(|{\bf u}|^2,B)=\left(\frac{\gamma-1}{\gamma} (B-\frac{1}{2} |{\bf u}|^2)\right)^{\frac{1}{\gamma-1}}.
\ee
It is easy to show that if a smooth flow does not contain the vacuum and the stagnation points, then the system \eqref{q2deuler} is equivalent to the following system
\be\label{qbc7}\begin{cases}
\frac{c^2(\rho)-u_1^2}{c^2(\rho)-u_2^2}\p_{x_1} u_1 -\frac{u_1 u_2}{c^2(\rho)-u_2^2}(\p_{x_1} u_2+\p_{x_2} u_1)+\p_{x_2} u_2+\frac{b(x_1) c^2(\rho)u_1}{c^2(\rho)-u_2^2}\\
=-\frac{1}{c^2(\rho)-u_2^2}(u_1\p_{x_1} B+ u_2 \p_{x_2} B),\\
\p_{x_1} u_2 -\p_{x_2} u_1 =-\frac{\p_{x_2} B}{u_1},\\
a(x_1)\rho(|{\bf u}|^2,B)(u_1\p_{x_1}+ u_2 \p_{x_2}) B=0,
\end{cases}\ee
where $c^2(\rho)=\gamma \rho^{\gamma-1}=(\gamma-1)(B-\frac{1}{2} |{\bf u}|^2)$. Indeed, the first equation in \eqref{qbc7} is obtained from substituting \eqref{qdensity} into the first equation in \eqref{q2deuler}, while the second equation in \eqref{qbc7} just follows from the momentum equations in \eqref{q2deuler}.

Let
\be\label{qd}
v_1= u_1- \bar{u},\ \ v_2 = u_2,\ \ \ Q= B- B_0.
\ee
Then ${\bf v}=(v_1,v_2)$ and $Q$ solve
\begin{eqnarray}\label{qaa}
\begin{cases}
k_{11}({\bf v},Q)\p_{x_1} v_1+\p_{x_2} v_2+k_{12}({\bf v},Q)(\p_{x_1}v_2+\p_{x_2} v_1)+\bar{k}_1(x_1)v_1=F_1({\bf v},Q),\\
\p_{x_1} v_2-\p_{x_2} v_1=F_2({\bf v},Q),\\
a(x_1)\rho({\bf v},Q)((\bar{u}+v_1)\p_{x_1}  + v_2\p_{x_2}) Q=0,\\
Q(L_0,x_2)=\epsilon B_{in}(x_2),\forall x_2\in (-1,1),\\
v_2(L_0,x_2)=\epsilon h_{1}(x_2),\forall x_2\in (-1,1),\\
v_2(x_1,\pm 1)=0,\ \ \forall x_1\in [L_0,L_1],
\end{cases}
\end{eqnarray}
where
\begin{eqnarray}\label{qf}
\begin{cases}
k_{11}({\bf v},Q)=\frac{c^2(\rho)-(\bar{u}+v_1)^2}{c^2(\rho)-v_2^2}, k_{12}({\bf v},Q)=-\frac{(\bar{u}+v_{1})v_{2}}{c^2(\rho)-v_2^2},\\
c^2(\rho)= (\gamma-1)(B_0+Q-\frac{1}{2} ((\bar{u}+v_1)^2+v_2^2)),\ \ \rho(v,Q)=\left(\frac{\gamma-1}{\gamma}(B_0+Q-\frac{1}{2} ((\bar{u}+v_1)^2+v_2^2))\right)^{\frac{1}{\gamma-1}},\\
\bar{k}_1(x_1)=b(x_1)-\frac{(\gamma-1)\bar{u}\bar{u}'}{c^4(\bar{\rho})}(\bar{u}^2+\frac{2c^2(\bar{\rho})}{\gamma-1})=\frac{2+(\gamma-1) \bar{M}^4}{1-\bar{M}^2} b(x_1),\\
F_1({\bf v},Q)=-\frac{(\gamma-1)\bar{u}^2\bar{u}'}{c^4(\bar{\rho})} Q-\frac{\bar{u}'}{c^2(\bar{\rho})(c^2(\rho)-v_2^2)}\{(c^2(\bar{\rho})-\bar{u}^2)v_2^2-(\gamma-1)(B_0 v_1^2+ \frac{1}{2}\bar{u}^2 v_2^2)\}\\
\quad +\frac{(\gamma-1)\bar{u}'}{c^4(\bar{\rho})(c^2(\rho)-v_2^2)}(c^2(\rho)-v_2^2-c^2(\bar{\rho}))\{\bar{u}^2 Q-(\bar{u}^2+\frac{2c^2(\bar{\rho})}{\gamma-1})\bar{u} v_1\}-\frac{(\bar{u}+v_1)\p_{x_1} Q + v_2\p_{x_2} Q}{c^2(\rho)-v_2^2},\\
F_2({\bf v},Q)=-\frac{\p_{x_2} Q}{\bar{u}+v_1}.
\end{cases}
\end{eqnarray}

Note that $F_2({\bf v},Q)$ only belongs to $H^2(\Omega)$ in general for $({\bf v}, Q)\in H^3(\Omega)$, the first two equations in \eqref{qaa} can be regarded as a first order system for $(v_1,v_2)$, which change types when crossing the sonic curve, the energy estimates obtained in the previous section for irrotational flows
indicate that the regularity of the solutions $v_1, v_2$ would be at best same as the source terms on the right hand sides in general. Thus it seems that only $H^2(\Omega)$ regularity for $(v_1,v_2)$ is possible and there appears a loss of derivatives. To recover the loss of derivatives, we require that one order higher regularity of the Bernoulli's quantity at the entrance. Using the continuity equation, we introduce the stream function which has the advantage of one order higher regularity than the velocity field. The Bernoulli's quantity can be represented as a function of the stream function. However, this function involves the inverse function of the restriction of the stream function at the entrance. There is still a loss of $\frac{1}{2}$ derivatives if the stream function only belongs to $H^4(\Omega)$. We further observe that the regularity of the flows in the subsonic region can be improved be $C^{3,\alpha}$ if the data at
the entrance have better $C^{3,\alpha}$ regularity so that the regularity of the stream function near the entrance can be improved to be $C^{4,\alpha}$. This will enable us to overcome the possibility of losing derivatives. To achieve this, we will choose some appropriate function spaces and design an elaborate two-layer iteration scheme to prove Theorem \ref{2dmain}.

Define $\Omega_{1/2}=\{(x_1,x_2): L_0<x_1<\frac{L_0}{2}, x_2 \in (-1,1)\}$ and
\be\no
&&\Omega_{1/3}=\left\{(x_1,x_2): L_0<x_1<\frac{L_0}{3}, x_2 \in (-1,1)\right\},\\\no
&&\Omega_{1/4}=\left\{(x_1,x_2): L_0<x_1<\frac{L_0}{4}, x_2 \in (-1,1)\right\}.
\ee
Then $\Omega_{1/2}\subset \Omega_{1/3}\subset \Omega_{1/4}$. Set
\be\nonumber
\Sigma_1&=&\bigg\{(v_1,v_2)\in (H^3(\Omega))^2: \sum_{j=1}^2\|v_j\|_{H^3(\Omega)}\leq \delta_0, v_2(x_1,\pm 1)= \p_{x_2}v_1(x_1,\pm 1)=\p_{x_2}^2v_2(x_1,\pm 1)=0\bigg\},\\\no
\Sigma_2&=&\bigg\{Q\in H^4(\Omega)\cap C^{3,\alpha}(\overline{\Omega_{1/3}})\cap C^{4,\alpha}(\overline{\Omega_{1/2}}):\\\no
&\quad&\quad\|Q\|_{H^4(\Omega)}+ \|Q\|_{C^{3,\alpha}(\overline{\Omega_{1/3}})}+ \|Q\|_{C^{4,\alpha}(\overline{\Omega_{1/2}})}\leq \delta_1, \p_{x_2} Q(x_1,\pm 1)=\p_{x_2}^3 Q(x_1,\pm 1)=0\bigg\}
\ee
with positive constants $\delta_0$ and $\delta_1>0$ to be specified later. For fixed $\tilde{Q} \in\Sigma_2$, we first construct an operator $\mathcal{T}^{\tilde{Q}}$:
$(\tilde{v}_1,\tilde{v}_2)\in \Sigma_1\mapsto (v_1,v_2)\in \Sigma_1$ by resolving the following boundary value problem
\begin{eqnarray}\label{q72}
\begin{cases}
k_{11}(\tilde{{\bf v}},\tilde{Q})\p_{x_1} v_1+\p_{x_2} v_2+k_{12}(\tilde{{\bf v}},\tilde{Q})(\p_{x_1}v_2+\p_{x_2} v_1)+\bar{k}_1 v_1= F_1(\tilde{{\bf v}},\tilde{Q}),\\
\p_{x_1} v_2-\p_{x_2} v_1=F_2(\tilde{{\bf v}},\tilde{Q})=-\frac{\p_{x_2} \tilde{Q}}{\bar{u}+\tilde{v}_1},\\
v_2(L_0,x_2)=\epsilon h_{1}(x_2),\ \ \ \forall x_2\in (-1,1),\\
v_2(x_1,\pm 1)=0, \ \ \forall x_1\in [L_0,L_1].
\end{cases}
\end{eqnarray}
Note that the first two equations in \eqref{q72} form a linear first order mixed type system with coefficients given in \eqref{qf}.

Since $\tilde{Q}\in \Sigma_2, \tilde{{\bf v}}\in \Sigma_1$, there holds that
\begin{eqnarray}\no
\|F_1(\tilde{{\bf v}},\tilde{Q})\|_{H^3(\Omega)}\leq m_*(\delta_1+\epsilon \delta_0+\delta_0\delta_1+\delta_0^2),\ \ \ \|F_2(\tilde{{\bf
v}},\tilde{Q})\|_{H^3(\Omega)}\leq m_*\delta_1.
\end{eqnarray}

Let $\psi_1(x_1,x_2)$ be the unique solution to the following problem
\begin{align}\no
\begin{cases}
(\p_{x_1}^2+\p_{x_2}^2)\psi_1=F_2(\tilde{{\bf v}},\tilde{Q})\in H^3(\Omega),\ &\text{on }\ \ \Omega,\\
\p_{x_1}\psi_1(L_0,x_2)=\p_{x_1}\psi_1(L_1,x_2)=0,\ \ &\text{on }\ \ \ x_2\in [-1,1],\\
\psi_1(x_1,\pm 1)=0, \ &\text{on }\ \ \ x_1\in [L_0,L_1].
\end{cases}
\end{align}
Since $\p_{x_2}\tilde{Q}(x_1,\pm 1)=\p_{x_2}^3\tilde{Q}(x_1,\pm 1)=0$ and $\p_{x_2} \tilde{v}_1(x_1,\pm 1)=0$, there holds
$$F_2(\tilde{{\bf v}},\tilde{Q})(x_1,\pm 1)=\p_{x_2}^2 (F_2(\tilde{{\bf v}},\tilde{Q}))(x_1,\pm 1)=0.$$
As in Lemma \ref{qH3}, one may use the symmetric extension technique to show that $\psi_1\in H^{5}(\Omega)$, $\p_{x_2}^2 \psi_1(x_1,\pm 1)=\p_{x_2}^4\psi_1(x_1,\pm 1)=0$ and the following estimate holds
\begin{align}\no
\|\psi_1\|_{H^{5}(\Omega)}\leq m_*\|F_2(\tilde{{\bf v}},\tilde{Q})\|_{H^3(\Omega)}\leq m_*\delta_1.
\end{align}

Next, we show the well-posedness of the classical solution $(w_1,w_2)\in H^3(\Omega)$ to the problem
\begin{eqnarray}\label{qs21}
\begin{cases}
k_{11}(\tilde{{\bf v}},\tilde{Q})\p_{x_1} w_1+\p_{x_2} w_2+k_{12}(\tilde{{\bf v}},\tilde{Q})(\p_{x_1}w_2+\p_{x_2} w_1)+\bar{k}_1 w_1= F_3(\tilde{{\bf v}},\tilde{Q},\psi_1),\ &\text{on }\ \ \Omega,\\
\p_{x_1} w_2-\p_{x_2} w_1=0,\ &\text{on }\ \ \Omega,\\
w_2(L_0,x_2)=\epsilon h_{1}(x_2),\ \ &\text{on } [-1,1],\\
w_2(x_1,\pm 1)=0,\ \ \ &\text{on } \ \ [L_0,L_1],
\end{cases}
\end{eqnarray}
where
\be\no
F_3(\tilde{{\bf v}},\tilde{Q},\psi_1)= F_1(\tilde{{\bf v}},\tilde{Q})+(k_{11}(\tilde{{\bf v}},\tilde{Q})-1)\p_{x_1x_2}^2\psi_1- k_{12}(\tilde{{\bf
v}},\tilde{Q})(\p_{x_1}^2\psi_1 -\p_{x_2}^2\psi_1)+\bar{k}_1 \p_{x_2}\psi_1\in H^3(\Omega).
\ee

Thanks to the second equation in \eqref{qs21}, there is a potential function $\psi_2$ such that $w_j=\p_{x_j} \psi_2$ for $j=1,2$ and $\psi_2(L_0,-1)=0$. Then the system \eqref{qs21} is reduced to the following second-order linear mixed type equation for $\psi_2$:
\begin{eqnarray}\label{qs22}
\begin{cases}
k_{11}(\tilde{{\bf v}},\tilde{Q})\p_{x_1}^2\psi_2+\p_{x_2}^2\psi_2+2k_{12}(\tilde{{\bf v}},\tilde{Q})\p_{x_1x_2}^2\psi_2+\bar{k}_1\p_{x_1}\psi_2
=F_3(\tilde{{\bf v}},\tilde{Q},\psi_1),\\
\p_{x_2}\psi_2(L_0,x_2)=\epsilon h_1(x_2),\ \ \forall x_2\in (-1,1),\\
\p_{x_2}\psi_2(x_1,\pm 1)=0,\ \ \ \forall x_1\in [L_0,L_1].
\end{cases}
\end{eqnarray}

It is easy to verify that the following compatibility conditions hold for a.e. $x_1\in (L_0,L_1)$:
\be\no
&&\p_{x_2} (F_3(\tilde{{\bf v}},\tilde{Q},\psi_1))(x_1,\pm 1)=\p_{x_2} k_{11}(\tilde{{\bf v}},\tilde{Q})(x_1,\pm 1)=0,\ \, \forall x_1\in [L_0,L_1],\\\no
&&\p_{x_2}^2 (k_{12}(\tilde{{\bf v}},\tilde{Q}))(x_1,\pm 1)=0,\ \, \forall x_1\in [L_0,L_1],\\\no
&&\|F_3(\tilde{{\bf v}},\tilde{Q},\psi_1)\|_{H^3(\Omega)}\leq C_0(\delta_1+\delta_0^2).
\ee

The problem \eqref{qs22} is a slight variation of \eqref{qlinearized1} whose coefficients satisfy also \eqref{qcoe}. One can adapt the same ideas in previous section to show the existence and uniqueness of a classical solution $\psi_2\in H^4(\Omega)$ to \eqref{qs22} satisfying $\p_{x_2}^3\psi_2(x_1,\pm 1)=0$ in the sense of $H^1(\Omega)$ trace and the estimate
\be\label{qs23}
\|\psi_2\|_{H^4(\Omega)}&\leq& C_*(\|F_3(\tilde{{\bf v}},\tilde{Q},\psi_1)\|_{H^3(\Omega)}+\epsilon\|h_1\|_{H^{3}([-1,1])})\\\no
&\leq& m_*(\epsilon+\delta_1+\delta_0^2).
\ee
Then $v_1:=\p_{x_1}\psi_2-\p_{x_2}\psi_1$ and $v_2:=\p_{x_2}\psi_2+ \p_{x_1}\psi_1$ are the unique $H^3(\Omega)$ solution to \eqref{q72} with
\be\no
&&v_2(x_1,\pm 1)= \p_{x_2}v_1(x_1,\pm 1)=\p_{x_2}^2v_2(x_1,\pm 1)=0,\\\no
&&\sum_{j=1}^2\|v_j\|_{H^3(\Omega)}\leq \sum_{j=1}^2\|\psi_j\|_{H^4(\Omega)}\leq m_*(\epsilon+\delta_1+\delta_0^2).
\ee

Choose $\delta_0=\sqrt{\epsilon+\delta_1}$ and $\sqrt{\epsilon+\delta_1}\leq \frac{1}{2m_*}$. Then $\mathcal{T}^{\tilde{Q}}$ is a well-defined operator from $\Sigma_1$ to itself. It remains to show that the operator $\mathcal{T}^{\tilde{Q}}$ is contractive in a low order norm for sufficiently small $\epsilon+\delta_1$. Set ${\bf
v}^{(i)}=\mathcal{T}^{\tilde{Q}}(\tilde{{\bf v}}^{(i)}) (i=1,2)$ for any $\tilde{\bf v}^{(1)}$, $\tilde{\bf v}^{(2)}\in \Sigma_1$ and denote
$\tilde{v}_j^{(1)}-\tilde{v}_j^{(2)}$ by $\tilde{V}_j$ and $v_j^{(1)}-v_j^{(2)}$ by $V_j$ for $j=1,2$. Then it follows from \eqref{q72} that
\begin{eqnarray}\no
\begin{cases}
k_{11}(\tilde{{\bf v}}^{(1)},\tilde{Q})\p_{x_1}V_1+ \p_{x_2} V_2+ k_{12}(\tilde{{\bf v}}^{(1)},\tilde{Q})(\p_{x_1} V_2+ \p_{x_2} V_1)+\bar{k}_1 V_1=\mathbb{F}(\tilde{{\bf
v}}^{(1)},\tilde{{\bf v}}^{(2)}, \tilde{Q}),\\
\p_{x_1}V_2-\p_{x_2} V_1=F_2(\tilde{{\bf v}}^{(1)},\tilde{Q})-F_2(\tilde{{\bf v}}^{(2)},\tilde{Q}),\\
V_2(L_0,x_2)=0, \ \forall x_2\in (-1,1),\\
V_2(x_1,\pm 1)=0,\ \ \ \forall x_1\in [L_0,L_1],
\end{cases}
\end{eqnarray}
where
\be\no
&&\mathbb{F}(\tilde{{\bf v}}^{(1)},\tilde{{\bf v}}^{(2)}, \tilde{Q})=F_1(\tilde{{\bf v}}^{(1)},\tilde{Q})-F_1(\tilde{{\bf v}}^{(2)},\tilde{Q})-(k_{11}(\tilde{{\bf
v}}^{(1)},\tilde{Q})-k_{11}(\tilde{{\bf v}}^{(2)},\tilde{Q}))\p_{x_1} v_1^{(2)}\\\no
&&\quad-(k_{12}(\tilde{{\bf v}}^{(1)},\tilde{Q})-k_{12}(\tilde{{\bf v}}^{(2)},\tilde{Q}))(\p_{x_1} v_2^{(2)}+\p_{x_2} v_1^{(2)}).
\ee

To estimate $V_1$ and $V_2$, one can decompose $V_1$ and $V_2$ as
\be\no
V_1=-\p_{x_2} \psi_3+ \p_{x_1} \psi_4,\ \ \ V_2=\p_{x_1}\psi_3 + \p_{x_2}\psi_4,
\ee
where $\psi_3$ and $\psi_4$ solve the following boundary value problems respectively:
\be\no\begin{cases}
(\p_{x_1}^2+\p_{x_2}^2)\psi_3= F_2(\tilde{{\bf v}}^{(1)},\tilde{Q})-F_2(\tilde{{\bf v}}^{(2)},\tilde{Q}),\  \text{in }\Omega,\\
\p_{x_1}\psi_3(L_0,x_2)=\p_{x_1}\psi_3(L_1,x_2)=0,\ \ x_2\in (-1,1),\\
\psi_3(x_1,\pm 1)=0,\ \ \forall x_1\in [L_0,L_1],
\end{cases}\ee
and
\be\no\begin{cases}
k_{11}(\tilde{{\bf v}}^{(1)},\tilde{Q})\p_{x_1}^2 \psi_4+\p_{x_2}^2 \psi_4+ 2k_{12}(\tilde{{\bf v}}^{(1)},\tilde{Q})\p_{x_1x_2}^2 \psi_4+\bar{k}_1\p_{x_1}\psi_4\\
=\mathbb{F}(\tilde{{\bf v}}^{(1)},\tilde{{\bf v}}^{(2)},\tilde{Q})+(k_{11}(\tilde{{\bf v}}^{(1)},\tilde{Q})-1)\p_{12}^2\psi_3 - k_{12}(\tilde{{\bf
v}}^{(1)},\tilde{Q})(\p_{x_1}^2\psi_3-\p_{x_2}^2\psi_3)+\bar{k}_1 \p_{x_2} \psi_3,\ \text{in }\Omega,\\
\psi_4(L_0,x_2)=0,\ \ \forall x_2\in (-1,1),\\
\p_{x_2} \psi_4(x_1,\pm 1)=0,\ \ \forall x_1\in [L_0,L_1].
\end{cases}\ee
Then combining the $H^2(\Omega)$ estimate of $\psi_3$ and the $H^1(\Omega)$ estimate of $\psi_4$ leads to
\begin{align}\label{qqv1}
\|(V_1,V_2)\|_{L^2(\Omega)}&\leq m_*(\delta_0+\delta_1) \|(\tilde{V}_1,\tilde{V}_2)\|_{L^2(\Omega)}.
\end{align}
Choose $\delta_0=\sqrt{\epsilon+\delta_1}$ small enough such that
\begin{align}\no
\|(V_1,V_2)\|_{L^2(\Omega)}\leq \frac12 \|(\tilde{V}_1,\tilde{V}_2)\|_{L^2(\Omega)}
\end{align}
Then $\mathcal{T}^{\tilde{Q}}$ is a contractive mapping in $L^2(\Omega)$-norm and there exists a unique fixed point $\tilde{{\bf v}}\in \Sigma_1$ to
$\mathcal{T}^{\tilde{Q}}$. Therefore, $(\tilde{u}_1,\tilde{u}_2)=(\bar{u}+\tilde{v}_1,\tilde{v}_2)$ solves the following problem

\begin{eqnarray}\label{q374}
\begin{cases}
\frac{c^2(\tilde{\rho})-\tilde{u}_1^2}{c^2(\tilde{\rho})-\tilde{u}_2^2}\p_{x_1} \tilde{u}_1 -\frac{\tilde{u}_1 \tilde{u}_2}{c^2(\tilde{\rho})-\tilde{u}_2^2}(\p_{x_1} \tilde{u}_2+\p_{x_2} \tilde{u}_1)+\p_{x_2}
\tilde{u}_2+\frac{c^2(\tilde{\rho})\tilde{u}_1}{c^2(\tilde{\rho})-\tilde{u}_2^2} b(x_1)\\
=-\frac{1}{c^2(\tilde{\rho})-\tilde{u}_2^2}(\tilde{u}_1\p_{x_1} \tilde{Q}+ \tilde{u}_2 \p_{x_2} \tilde{Q}),\\
\p_{x_1} \tilde{u}_2 -\p_{x_2} \tilde{u}_1=-\frac{\p_{x_2} \tilde{Q}}{\tilde{u}_1},\\
\tilde{u}_2(L_0,x_2)= \epsilon h_1(x_2),\ \ x_2\in (-1,1),\\
\tilde{u}_2(x_1, \pm 1)=0,\ \ \forall x_1\in [L_0,L_1],.
\end{cases}
\end{eqnarray}

Note that when $\epsilon+\delta_1$ is suitably small, \eqref{q374} can be regarded as a uniformly first order elliptic system in $\Omega_{1/4}$. Since $\tilde{{\bf v}}\in
\Sigma_1$ and $\tilde{Q}\in \Sigma_2$, so the coefficients in \eqref{q374} belong to $H^3(\Omega)\subset C^{1,\alpha_1}(\overline{\Omega})$ for each $\alpha_1\in (0,1)$,
the terms on the right hand side of the equations in \eqref{q374} belong to $H^3(\Omega)$ and $\tilde{u}_2(L_0,\cdot)\in C^{3,\alpha}([-1,1])$. Thus by the standard interior and boundary regularity
estimates for elliptic systems, one can improve the regularity of $\tilde{{\bf u}}\in H^{4}(\Omega_{1/3})\subset C^{2,\alpha}(\overline{\Omega_{1/3}})$. This, together
with the assumption $\tilde{Q}\in C^{3,\alpha}(\overline{\Omega_{1/3}})$, implies that the terms on the right hand side of the equations in \eqref{q374} belong to
$C^{2,\alpha}(\overline{\Omega_{1/3}})$. The interior and boundary Schauder estimates to elliptic systems in $\Omega_{1/3}$ yield that $\tilde{{\bf u}}\in
C^{3,\alpha}(\overline{\Omega_{1/2}})$. In particular, $(\rho(\tilde{{\bf v}},\tilde{Q})(\bar{u}+\tilde{v}_1)(L_0,\cdot)\in C^{3,\alpha}([-1,1])$.

It follows from the first equation in \eqref{q374} that
\be\label{qdensity300}
\p_{x_1}(a(x_1)\rho(\tilde{{\bf v}},\tilde{Q})(\bar{u}+\tilde{v}_1))+\p_{x_2}(a(x_1)\rho(\tilde{{\bf v}},\tilde{Q})\tilde{v}_2)=0.
\ee
Therefore one may define a stream function on $[L_0,L_1]\times [-1,1]$ as
\begin{align}\no
\phi(x_1,x_2)=\int_{-1}^{x_2}a(L_0)(\rho(\tilde{{\bf v}},\tilde{Q})(\bar{u}+\tilde{v}_1))(L_0,\tau)d\tau-\int_{L_0}^{x_1}a(\tau)(\rho(\tilde{{\bf
v}},\tilde{Q})\tilde{v}_2)(\tau,x_2)d\tau.
\end{align}
The function $\phi$ has the properties
\begin{align}\no\begin{cases}
\p_{x_1}\phi(x_1,x_2)=-a(x_1)(\rho(\tilde{{\bf v}},\tilde{Q})\tilde{v}_2)(x_1,x_2)\in H^3(\Omega)\cap C^{2,\alpha}(\overline{\Omega_{1/3}})\cap
C^{3,\alpha}(\overline{\Omega_{1/2}}),\\
\p_{x_2}\phi(x_1,x_2)=a(x_1)\rho(\tilde{{\bf v}},\tilde{Q})(\bar{u}+\tilde{v}_1)(x_1,x_2)\in H^3(\Omega)\cap C^{2,\alpha}(\overline{\Omega_{1/3}})\cap
C^{3,\alpha}(\overline{\Omega_{1/2}}),\\
\phi(x_1,x_2)\in H^4(\Omega)\cap C^{3,\alpha}(\overline{\Omega_{1/3}})\cap C^{4,\alpha}(\overline{\Omega_{1/2}}).
\end{cases}
\end{align}
Since $\p_{x_1}\phi(x_1,\pm 1)=0$, so $\phi(x_1,-1)=\phi(L_0,-1)$ and $\phi(x_1,1)=\phi(L_0,1)$. Note that $\bar{u}(x_1)>0$ for every $x_1\in [L_0,L_1]$, and $\tilde{{\bf v}}\in \Sigma_1$,  $\p_{x_2}\phi(x_1,x_2)=a(x_1)\rho(\tilde{{\bf
v}},\tilde{Q})(\bar{u}+\tilde{v}_1)(x_1,x_2)>0$, then $\phi(x_1,x_2)$ is an strictly increasing function of $x_2$ for each fixed $x_1\in [L_0, L_1]$. These imply that the closed interval $[\phi(x_1,-1),\phi(x_1,1)]$ is simply equal to $[\phi(L_0,-1), \phi(L_0,1)]$. Denote the inverse function of $\phi(L_0,\cdot): [-1,1]\to [\phi(L_0,-1), \phi(L_0,1)]$ by
$\phi_{L_0}^{-1}(\cdot): [\phi(L_0,-1), \phi(L_0,1)]\to [-1,1]$. Define the function
\begin{align}\label{qQ}
Q(x_1,x_2)=\epsilon B_{in}(\phi_{L_0}^{-1}(\phi(x_1,x_2))).
\end{align}
Then one can easily verify that $Q$ solves the following transport problem
\begin{eqnarray}\label{q71}
\begin{cases}
a(x_1)\rho(\tilde{{\bf v}},\tilde{Q})((\bar{u}+\tilde{v}_1)\partial_{x_1} +\tilde{v}_2\partial_{x_2}) Q=0,\\
Q(L_0,x_2)=\epsilon B_{in}(\cdot)\in C^{4,\alpha}([-1,1]),\\
\end{cases}
\end{eqnarray}
Furthermore, due to $(\rho(\tilde{{\bf v}},\tilde{Q})(\bar{u}+\tilde{v}_1)(L_0,\cdot)\in C^{3,\alpha}([-1,1])$, there holds $\phi_{L_0}^{-1}(\cdot)\in
C^{4,\alpha}([\phi(L_0,-1), \phi(L_0,1)])$ and
\begin{align}\no
\|Q\|_{H^4(\Omega)}+\|Q\|_{C^{3,\alpha}(\overline{\Omega_{1/3}})}+\|Q\|_{C^{4,\alpha}(\overline{\Omega_{1/2}})}\leq m_*\epsilon.
\end{align}
Namely, for any $\tilde{Q}\in \Sigma_2$, we have constructed an operator $\mathcal{P}$: $\tilde{Q}\in \Sigma_2\mapsto Q\in \Sigma_2$ if one selects $\delta_1=\sqrt{\epsilon}$ and
$\sqrt{\epsilon}\leq \frac{1}{2m_*}$.

It remains to show that the mapping $\mathcal{P}$ is contractive in a low order norm for suitably small $\epsilon$. Let $Q^{(i)}=\mathcal{P}(\tilde{Q}^{(i)}) (i=1,2)$ for
any $\tilde{Q}^{(i)}\in \Sigma_2, (i=1,2)$ and denote $\tilde{Q}^{(1)}-\tilde{Q}^{(2)}$ by $\tilde{Q}_d$ and $Q^{(1)}-Q^{(2)}$ by $Q_d$, respectively. Then, it follows from \eqref{qQ} that
\begin{align}\no
Q^{(i)}(x_1,x_2)=\epsilon B_{in}\circ(\phi^{(i)}_{L_0})^{-1}\circ \phi^{(i)}(x_1,x_2),
\end{align}
where
\begin{align}\no
\phi^{(i)}(x_1,x_2)=\int_{-1}^{x_2}a(L_0)(\rho(\tilde{{\bf v}}^{(i)},\tilde{Q}^{(i)})(\bar{u}+\tilde{v}_1^{(i)}))(L_0,\tau)d\tau-\int_{L_0}^{x_1}a(\tau)(\rho(\tilde{{\bf
v}}^{(i)},\tilde{Q}^{(i)})\tilde{v}^{(i)}_2)(\tau,x_2)d\tau,
\end{align}
$(\phi_{L_0}^{(i)})^{-1}$: $t\in [\phi^{(i)}(L_0,-1), \phi^{(i)}(L_0,1)]\mapsto x_2\in [-1,1]$ is the inverse function of $\phi^{(i)}(L_0,\cdot)$: $x_2\in [-1,1]\mapsto t\in [\phi^{(i)}(L_0,-1), \phi^{(i)}(L_0,1)]$ and
$\tilde{{\bf v}}^{(i)}$ is the unique fixed point of $\mathcal{T}^{(\tilde{Q}^{(i)})}$ for $i=1,2$.
Thus,
\begin{align}\no
|Q_d|=|Q^{(1)}-Q^{(2)}|\leq\epsilon \|B'_{in}\|_{L^\infty([-1,1])}|\beta^{(1)}(x_1,x_2)-\beta^{(2)}(x_1,x_2)|
\end{align}
where $\beta^{(i)}(x_1,x_2)=(\phi^{(i)}_{L_0})^{-1}\circ \phi^{(i)}(x_1,x_2)\in[-1,1]$. It follows from the definitions that
\begin{align*}
&\int_{\beta^{(2)}(x_1,x_2)}^{\beta^{(1)}(x_1,x_2)}a(L_0)\rho(\tilde{{\bf v}}^{(1)},\tilde{Q}^{(1)})(\bar{u}+\tilde{v}_1^{(1)}))(L_0,\tau)d\tau
=\phi^{(1)}(x_1,x_2)-\phi^{(2)}(x_1,x_2)\\
&\quad\quad\quad-\int_{-1}^{\beta^{(2)}(x_1,x_2)}a(L_0)\{\rho(\tilde{{\bf v}}^{(1)},\tilde{Q}^{(1)})(\bar{u}+\tilde{v}_1^{(1)})-
\rho(\tilde{{\bf v}}^{(2)},\tilde{Q}^{(2)})(\bar{u}+\tilde{v}_1^{(2)})\}(L_0,\tau)d\tau,
\end{align*}
which implies
\begin{align*}
&\underline{m}^{(1)}|\beta^{(1)}(x_1,x_2)-\beta^{(2)}(x_1,x_2)|\\
\leq& |\phi^{(1)}(x_1,x_2)-\phi^{(2)}(x_1,x_2)|+\int_{-1}^{1}a(L_0)|\rho(\tilde{{\bf v}}^{(1)},\tilde{Q}^{(1)})(\bar{u}+\tilde{v}_1^{(1)}))-
\rho(\tilde{{\bf v}}^{(2)},\tilde{Q}^{(2)})(\bar{u}+\tilde{v}_1^{(2)}))|(L_0,\tau)d\tau
\end{align*}
with $\underline{m}^{(i)}:=\displaystyle\min_{x_2\in[-1,1]}a(L_0)(\rho(\tilde{{\bf v}}^{(i)},\tilde{Q}^{(i)})(\bar{u}+\tilde{v}_1^{(i)}))(L_0,x_2)>0$. Noting that
$Q_d(L_0,x_2)\equiv 0$, one has
\begin{align}\no
\|Q_d\|_{L^2(\Omega)}\leq m_*\epsilon\bigg(\|(\tilde{{\bf v}}^{(1)}-\tilde{{\bf v}}^{(2)},\tilde{Q}_d)\|_{L^2(\Omega)}+\|(\tilde{{\bf v}}^{(1)}-\tilde{{\bf
v}}^{(2)})(L_0,\cdot)\|_{L^2([-1,1])}\bigg).
\end{align}
Also there holds
\begin{align*}
|\p_{x_1} Q_d|=&\epsilon|B_{in}'(\beta^{(1)}(x_1,x_2))\p_{x_1} \beta^{(1)}(x_1,x_2)-B_{in}'(\beta^{(2)}(x_1,x_2))\p_{x_1} \beta^{(2)}(x_1,x_2)|\\\no
=&\epsilon|\big(B_{in}'(\beta^{(1)}(x_1,x_2))-B_{in}'(\beta^{(2)}(x_1,x_2))\big)\p_{x_1} \beta^{(1)}+ B_{in}'(\beta^{(2)}(x_1,x_2))(\p_{x_1} \beta^{(1)}-\p_{x_1}
\beta^{(2)})|\\\no
\leq&\epsilon \|B''_{in}\|_{L^\infty([-1,1])}|\beta^{(1)}(x_1,x_2)-\beta^{(2)}(x_1,x_2)|\frac{1}{\underline{m}^{(1)}}\|\nabla
\phi^{(1)}(x_1,x_2)\|_{L^\infty(\Omega)}\\\no
&+\epsilon \|B'_{in}\|_{L^\infty([-1,1])}\frac{\|\nabla
\phi^{(1)}(x_1,x_2)\|_{L^\infty(\Omega)}}{\underline{m}^{(1)}\underline{m}^{(2)}}\bigg|(\rho(\tilde{{\bf
v}}^{(1)},\tilde{Q}^{(1)})(\bar{u}+\tilde{v}_1^{(1)}))(L_0,\beta^{(1)})\\
&\quad-(\rho(\tilde{{\bf v}}^{(2)},\tilde{Q}^{(2)})(\bar{u}+\tilde{v}_1^{(2)}))(L_0,\beta^{(2)})\bigg|\\
&+\epsilon \|B'_{in}\|_{L^\infty([-1,1])}\frac{1}{\underline{m}^{(2)}} \bigg|\nabla \phi^{(1)}(x_1,x_2)-\nabla \phi^{(2)}(x_1,x_2)\bigg|,
\end{align*}
and similar computations are valid for $\p_{x_2} Q_d$. Then one has
\begin{align}\no
\|\nabla Q_d\|_{L^2(\Omega)}\leq m_*\epsilon\bigg(\|(\tilde{{\bf v}}^{(1)}-\tilde{{\bf v}}^{(2)}\|_{L^2(\Omega)}+\|\tilde{Q}_d)\|_{H^1(\Omega)}+\|(\tilde{{\bf v}}^{(1)}-\tilde{{\bf
v}}^{(2)})(L_0,\cdot)\|_{L^2([-1,1])}\bigg).
\end{align}

Therefore,
\begin{align}\label{qdd10}
\|Q_d\|_{H^1(\Omega)}\leq m_*\epsilon\bigg(\|(\tilde{{\bf v}}^{(1)}-\tilde{{\bf v}}^{(2)},\tilde{Q}_d)\|_{L^2(\Omega)}+\|(\tilde{{\bf v}}^{(1)}-\tilde{{\bf
v}}^{(2)})(L_0,\cdot)\|_{L^2([-1,1])}\bigg).
\end{align}

It remains to show
\begin{eqnarray}\label{qdd11}
\|(\tilde{{\bf v}}^{(1)}-\tilde{{\bf v}}^{(2)})\|_{L^2(\Omega)}+\|(\tilde{{\bf v}}^{(1)}-\tilde{{\bf v}}^{(2)})(L_0,\cdot)\|_{L^2([-1,1])}\leq
m_*\|\tilde{Q}_d\|_{H^1(\Omega)}.
\end{eqnarray}

Indeed, set $\tilde{v}_j^{(1)}-\tilde{v}_j^{(2)}$ by $Y_{j}$ for $j=1,2$ respectively. It follows from \eqref{q72} that
\begin{eqnarray}\label{qcontract}
\begin{cases}
k_{11}(\tilde{{\bf v}}^{(1)},\tilde{Q}^{(1)})\p_{x_1} Y_{1}+\p_{x_2} Y_{2}
+k_{12}(\tilde{{\bf v}}^{(1)},\tilde{Q}^{(1)})(\p_{x_1} Y_{2}+\p_{x_2} Y_{1})+\bar{k}_1(x_1)Y_{1}=R_1,\\
\p_{x_1} Y_{2}-\p_{x_2} Y_1=R_2,\\
Y_{2}(L_0,x_2)=0,\\
Y_{2}(x_1,\pm 1)=0,
\end{cases}
\end{eqnarray}

where $R_1$ and $R_2$ are given by
\be\no
R_1&=& F_1(\tilde{{\bf v}}^{(1)},\tilde{Q}^{(1)})-F_1(\tilde{{\bf v}}^{(2)},\tilde{Q}^{(2)})-(k_{11}(\tilde{{\bf v}}^{(1)},\tilde{Q}^{(1)})-k_{11}(\tilde{{\bf
v}}^{(2)},\tilde{Q}^{(2)}))\p_{x_1} \tilde{v}_1^{(2)}\\\no
&\quad&\quad-(k_{12}(\tilde{{\bf v}}^{(1)},\tilde{Q}^{(1)})-k_{12}(\tilde{{\bf v}}^{(2)},\tilde{Q}^{(2)}))(\p_{x_1} \tilde{v}_2^{(2)}+\p_{x_2} \tilde{v}_1^{(2)}),\\\no
R_2&=& -\bigg(\frac{\p_{x_2} \tilde{Q}^{(1)}}{\bar{u}+\tilde{v}_1^{(1)}}-\frac{\p_{x_2} \tilde{Q}^{(2)}}{\bar{u}+\tilde{v}_1^{(2)}}\bigg).
\ee
There holds
\begin{align}\no
\begin{cases}
\|R_1\|_{L^2(\Omega)}\leq m_*\bigg(\|\tilde{Q}_d\|_{H^1(\Omega)}+\delta_0\|(\tilde{{\bf v}}^{(1)}-\tilde{{\bf v}}^{(2)})\|_{L^2(\Omega)}\bigg),\\
\|R_2\|_{L^2(\Omega)}\leq m_*\bigg(\|\tilde{Q}_d\|_{H^1(\Omega)}+\delta_1\|(\tilde{{\bf v}}^{(1)}-\tilde{{\bf v}}^{(2)})\|_{L^2(\Omega)}\bigg).
\end{cases}
\end{align}

Similar arguments as for \eqref{qqv1} yield
\be\no
\|(\tilde{{\bf v}}^{(1)}-\tilde{{\bf v}}^{(2)})\|_{L^2(\Omega)}\leq m_*\bigg(\|\tilde{Q}_d\|_{H^1(\Omega)}+(\delta_0+\delta_1)\|(\tilde{{\bf v}}^{(1)}-\tilde{{\bf
v}}^{(2)})\|_{L^2(\Omega)}\bigg).
\ee
Since \eqref{qcontract} is uniformly elliptic in $\overline{\Omega_{1/3}}$, the interior and boundary $H^1$ estimates for elliptic systems yield that
\be\no
\|(\tilde{{\bf v}}^{(1)}-\tilde{{\bf v}}^{(2)})\|_{H^1(\overline{\Omega_{1/2}})}\leq m_*\bigg(\|\tilde{Q}_d\|_{H^1(\Omega)}+(\delta_0+\delta_1)\|(\tilde{{\bf
v}}^{(1)}-\tilde{{\bf v}}^{(2)})\|_{L^2(\Omega)}\bigg),
\ee
which further implies, by the trace Theorem, that
\be\no
\|(\tilde{{\bf v}}^{(1)}-\tilde{{\bf v}}^{(2)})(L_0,\cdot)\|_{L^2([-1,1])}\leq m_*\bigg(\|(\tilde{Q}_d\|_{H^1(\Omega)}+(\delta_0+\delta_1)\|(\tilde{{\bf
v}}^{(1)}-\tilde{{\bf v}}^{(2)})\|_{L^2(\Omega)}\bigg).
\ee
Choosing $\delta_0+\delta_1=\sqrt{\epsilon+\delta_1}+\delta_1=\sqrt{\epsilon+\sqrt{\epsilon}}+\sqrt{\epsilon}$ small enough such that $m_*(\delta_0+\delta_1)<1/2$, one obtains
\eqref{qdd11}.

Combining \eqref{qdd10} and \eqref{qdd11}, we obtain finally that
\begin{align}\no
\|Q_d\|_{H^1(\Omega)}&\leq m_*\epsilon \|\tilde{Q}_d\|_{H^1(\Omega)}\leq \frac12 \|\tilde{Q}_d\|_{H^1(\Omega)},
\end{align}
provided that $0<\epsilon \leq \frac{1}{2m_*}$. Hence $\mathcal{P}$ is a contractive mapping in $H^1(\Omega)$-norm and there exists a unique fixed point $Q\in \Sigma_2$.
Denote the fixed point of the mapping $\mathcal{T}^{Q}$ in $\Sigma_1$ by ${\bf v}$. Then $({\bf v}, Q)$ is a solution to the quasi 2-D steady Euler flow model
\eqref{q2deuler} with boundary conditions \eqref{qbcs}, which also satisfies the estimate \eqref{q2d10}. The uniqueness can be proved by a similar argument as for the
contraction of the two mappings $\mathcal{T}^{Q}$ and $\mathcal{P}$. The properties of the sonic curve can be proved as in Theorem \ref{q-irro}. The
proof of Theorem \ref{2dmain} is completed.

{\bf Acknowledgement.} Weng is supported in part by National Natural Science Foundation of China 11971307, 12071359, 12221001. Xin is supported in part by the Zheng Ge Ru Foundation, Hong Kong RGC Earmarked Research Grants CUHK-14301421, CUHK-14300917, CUHK-14302819 and CUHK-14300819, the key project of NSFC (Grant No.12131010) and by Guangdong Basic and Applied Basic Research Foundation 2020B1515310002.


\begin{thebibliography}{}


\bibitem{bdx}
M. Bae, B. Duan and C. Xie. {\it Two dimensional accelerating flows of the steady Euler-Poisson system with a $C^1$ transonic transition.} Preprint 2022.


\bibitem{Bers1958}
L. Bers. {\it Mathematical aspects of subsonic and transonic gas dynamics.}, Wiley, New York, 1958.


\bibitem{cf03}
G. Chen and M. Feldman. {\it Multidimensional transonic shocks and free boundary problems for nonlinear equations of mixed type}, J. Amer. Math. Soc. 16 (2003), 461-494.

\bibitem{cdsw07}
G. Chen, C. Dafermos, M. Slemrod, D. Wang. {\it On two-dimensional sonic-subsonic flow.} Comm. Math. Phys. 271 (2007), no. 3, 635-647.

\bibitem{chw16}
G. Chen, F. Huang, T. Wang. {\it Sonic-subsonic limit of approximate solutions to multidimensional steady Euler equations,} Arch. Rational Mech. Anal.,219 (2016),
719-740.


\bibitem{chen08}
S. Chen. {\it Transonic shocks in 3-D compressible flow passing a duct with a general section for Euler systems,} Trans. Amer. Math. Soc, 360 (2008), 5265-5289.

\bibitem{Courant1948}
R. Courant and K. O. Friedrichs. {\it Supersonic flow and shock waves},
  Interscience Publishers, Inc., New York, 1948.

\bibitem{egm84}
P. Embid, J. Goodman, A. Majda. {\it Multiple steady states for 1-D transonic flow,}
SIAM J. Sci. Stat. Comput. 5 (1984), 21-41.

\bibitem{fx21}
B. Fang and Z. Xin. {\it On admissible locations of transonic shock fronts for steady Euler flows in an almost flat finite nozzle with prescribed receiver pressure}, Comm. Pure Appl. Math. 74 (2021), no.7, 1493-1544.

\bibitem{frankl45}
F. Frankl. {\it On the problems of Chaplygin for mixed sub and supersonic flows}, Bull del'Acd.des Sciencede l'USSR 9 (1945), 121-142.

\bibitem{Friedrichs1958}
K. O. Friedrichs, {\it Symmetric positive linear differential equations}, Comm. Pure Appl. Math. \textbf{XI} (1958), 333-418.

\bibitem{gt}
D. Gilbarg and N. Trudinger. {\it Elliptic partial differential equations of second order.} 2nd Ed. Springer-Verlag: Berlin.

\bibitem{gmp84}
J. Glimm, G. Marshall, and B. Plohr. {\it A generalized Riemann problem for quasi-one-dimensional gas flow}, Adv. Appl. Math., 5 (1984), 1-30.

\bibitem{hww11}
F. Huang, T. Wang, and Y. Wang. {\it On multi-dimensional sonic-subsonic
flow}, Acta Mathematica Scientia, 31 (2011), 2131-2140.

\bibitem{Kuzmin2002}
Alexander G. Kuz'min. {\it Boundary value problems for transonic flow}, John
  Wiley ${\&}$ Sons, Ltd, West Sussex, 2002.

\bibitem{lxy09a}
J. Li, Z. Xin and H. Yin. {\it On transonic shocks in a nozzle with variable end pressures,} Comm. Math. Phys, 291 (2009), 111-150.


\bibitem{lxy09b}
J. Li, Z. Xin and H. Yin. {\it A free boundary value problem for the full Euler system and 2-D transonic shock in a large variable nozzle,} Math. Res. Lett., {\bf 16} (2009), 777-796.


%

\bibitem{lxy13}
J. Li, Z. Xin and H. Yin. {\it Transonic shocks for the full compressible Euler system in a general two-dimensional de Laval nozzle,} Arch. Ration. Mech. Anal. 207 (2013), no. 2, 533-581.

\bibitem{lr57}
H. W. Liepmann and A. Roshko. {\it Elements of gas dynamics.} GALCIT Aeronautical Series, New York, Wiley, 1957.


\bibitem{L82a} T. Liu. {\it Nonlinear stability and instability of transonic flows
through a nozzle,} Commun. Math. Phys. 83, 243-260 (1982).

\bibitem{L82b} T. Liu. {\it Transonic gas flow in a duct of varing area,} Arch. Ration. Mech. Anal. 80 (1982), 1-18.

\bibitem{RXY13}
J. Rauch, C. Xie, and Z. Xin. {\it Global stability of steady transonic Euler shocks in quasi-one-dimensional nozzles.} Journal de Mathematiques Pures et Appliques, Vol. 99, No.4, 2013, 395-408.



\bibitem{Morawetz1956}
C. Morawetz. {\it On the non-existence of continuous transonic flows past profiles I}, Comm. Pure Appl. Math. \textbf{9}
  (1956), no.~1, 45--68.





\bibitem{tricomi23}
F. G. Tricomi, {\it Sulle equazioni lineari alle derivate parziali di seconde ordine di tipo misto}, Rend. Reale Accad. Lincei, Ser 14 (1923), 133-247.

\bibitem{WX2013}
C. Wang and Z. Xin. {\it On a degenerate free boundary problem and continuous subsonic-sonic flows in a convergent nozzle.} Arch. Ration. Mech. Anal. 208 (2013), no.3,
911-975.


\bibitem{WX2016}
C. Wang and Z. Xin. {\it On sonic curves of smooth subsonic-sonic and transonic flows.} SIAM J. Math. Anal. 48 (2016), no. 4. 2414-2453.

\bibitem{WX2019}
C. Wang and Z. Xin. {\it Smooth transonic flows of Meyer type in
De Laval nozzles}, Arch. Ration. Mech. Anal. \textbf{232}
  (2019), no.~3, 1597--1647.

\bibitem{WX2020}
C. Wang and Z. Xin. {\it Regular subsonic-sonic flows in general nozzles},  Advances in Math. {\bf 380} (2021), 107578.




\bibitem{WengXin19}
S. Weng and Z. Xin. {\it A deformation-curl decomposition for three dimensional steady Euler equations (in Chinese).}
  Sci Sin Math, 2019, 49: 307-320, doi: 10.1360/N012018-00125.

\bibitem{weng2019}
S. Weng. {\it A deformation-curl-Poisson decomposition to the three dimensional steady Euler-Poisson system with applications.} J. Differential Equations 267 (2019), no. 11, 6574-6603.

\bibitem{wxx21} S. Weng, C. Xie, Z. Xin.  {\it Structural stability of the
transonic shock problem in a divergent three-dimensional axisymmetric perturbed
nozzle}, SIAM J. Math. Anal. 53, 279-308 (2021).


\bibitem{WXY21a}
S. Weng, Z. Xin and H. Yuan. {\it Steady compressible radially symmetric flows in an annulus}, J. Differential Equations 286 (2021), 433-454.

\bibitem{WXY21b}
S. Weng, Z. Xin and H. Yuan.  {\it On some smooth symmetric transonic flows with nonzero angular velocity and vorticity.} Math Models and Methods in Applied Sci, Vol. 31, No. 13, 2773-2817 (2021).

\bibitem{wz21}
S. Weng and Z. Zhang. {\it Two dimensional subsonic and subsonic-sonic spiral flows outside a porous body.} Acta Mathematica Scientia 42, 1569-1584 (2022).

\bibitem{wx23}
S. Weng and Z. Xin. {\it Existence and stability of cylindrical transonic shock solutions under three dimensional perturbations}. arXiv:2304.02429.

\bibitem{xx07}
C. Xie and Z. Xin. {\it Global subsonic and subsonic-sonic flows through
infinitely long nozzles}, Indiana University Mathematical Journal, 56 (2007), 2991-3023.


\bibitem{xy05}
Z. Xin and H. Yin, {\it Transonic shock in a nozzle I: 2D case,} Comm. Pure Appl. Math., 58 (2005), 999-1050.

\bibitem{xy08a}
Z. Xin and H. Yin, {\it 3-diemsional transonic shock in a nozzle,} Pacific J. Math. 236, (2008), 139-193.
%
\bibitem{xy08b}
Z. Xin and H. Yin, {\it Transonic shock in a curved nozzle, 2-D and 3-D complete Euler systems,} J. Differential Equations. 245, (2008), 1014-1085.

\end{thebibliography}
\end{document}